\documentclass[final,onefignum,onetabnum]{siamart171218}
\usepackage{graphicx,subfig}
\usepackage{stmaryrd}
\usepackage{commath}

\usepackage{url}
\usepackage{xcolor}
\usepackage{color, colortbl}
\usepackage{amsmath}
\usepackage{booktabs}
\usepackage{amssymb}
\usepackage{mathrsfs}
\usepackage{enumitem}
\usepackage{bm}
\usepackage{footnote}
\usepackage{hyperref}
\usepackage[notref,notcite]{showkeys}
\usepackage[textsize=small]{todonotes}
\usetikzlibrary{patterns}

\definecolor{LightRed}{rgb}{0.9,0.55,0.5}

\makeatletter
\def\ps@pprintTitle{%
 \let\@oddhead\@empty
 \let\@evenhead\@empty
 \def\@oddfoot{}%
 \let\@evenfoot\@oddfoot}

\newcommand{\tnorm}{\@ifstar\@tnorms\@tnorm}

\newcommand{\@tnorm}[2][]{%
	\mathopen{#1|\mkern-1.5mu#1|\mkern-1.5mu#1|}
	#2
	\mathclose{#1|\mkern-1.5mu#1|\mkern-1.5mu#1|}
}
\newcommand{\jump}[1]{\llbracket #1 \rrbracket}

\newcommand{\Bk}{\color{black}}

\newtheorem{assumption}{Assumption}

\newtheorem{remark}{Remark}
	\title{A Hybridizable Discontinuous Galerkin Method for the Miscible Displacement Problem Under Minimal Regularity}

\author{Keegan L.A. Kirk\thanks{Department of Mathematical Sciences and Center for Mathematics and Artificial Intelligence, George Mason University, Fairfax, VA 22030 (\email{kkirk6@gmu.edu})} \and Beatrice Riviere \thanks{The Ken Kennedy Institute, Rice University, Houston, TX 77005. (\email{riviere@rice.edu}).}\funding{This work was partially supported by the Natural Sciences and Engineering Research Council of Canada under the Postdoctoral Fellowship Program (PDF-568008) (KK), the Office of Naval Research (ONR) under Award NO: N00014-24-1-2147 (KK), the Air Force Office of Scientific Research (AFOSR) under Award NO: FA9550-22-1-0248 (KK) and the National Sciences Foundation Division of Mathematical Sciences under Award NO: NSF-DMS 2111459 (BR).}}

\headers{HDG for miscible displacement}{K. Kirk, B. Riviere}

\begin{document}

\maketitle
\begin{abstract}
A numerical method based on the hybridizable discontinuous Galerkin method in space and backward Euler in time  is formulated and analyzed for solving the miscible displacement problem. Under low regularity assumptions, convergence is established by proving that, up to a subsequence, the discrete pressure, velocity and concentration converge to a weak solution as the mesh size and time step tend to zero. The analysis is based on several key features: an H(div) reconstruction of the velocity, the skew-symmetrization of the concentration equation, the introduction of an auxiliary variable and the definition of a new numerical flux. Numerical examples demonstrate optimal rates of convergence for smooth solutions, and convergence for problems of low regularity.
\end{abstract}

\begin{keywords} Hybridizable discontinuous Galerkin, convergence,
compactness, H(div) projection, convergence, low regularity.
\end{keywords}

\begin{AMS}
65M12, 65M60, 76S05
\end{AMS}

\maketitle



	\section{Introduction}

This work formulates and analyzes a hybridizable discontinuous Galerkin (HDG) method for solving the miscible displacement problem.  This type of coupled flow and transport problem occurs when two miscible fluids move through a porous medium, for instance the flow of solvent in a subsurface saturated with a contaminant. To our knowledge, our work is the first one that proves theoretically the convergence of  an HDG-based method for the miscible displacement problem under low regularity.   The paper \cite{Fabien:2020a} demonstrated computationally that an HDG method (different from our proposed method) is a suitable and an effective method for problems of sufficient regularity, but it did not contain any theoretical analysis.

The numerical analysis of the miscible displacement problem is challenging for several reasons. First, the flow and transport equations are fully coupled as opposed to the one-way coupling of a tracer flow. Second, the dispersion-diffusion matrix in the transport equation depends on the velocity  and therefore its uniform boundedness cannot be assumed. Indeed, there is no guarantee that the weak solution for the velocity field be bounded in the sup-norm. 

Our main contribution is to propose and analyze a novel unconditionally stable HDG method for the miscible displacement in the general case where weak solutions are of low regularity and where the dispersion-diffusion matrix is not assumed to be bounded. We utilize a skew-symmetrization formulation to express the convection in half primal and half-dual form as in
 \cite{Bartels:2009,Li:2015}. We introduce a second auxiliary variable for the mixed formulation to handle the nonlinearity. 
 We introduce carefully designed numerical fluxes (inspired by the Lax--Friedrichs fluxes) to ensure positive-definiteness of the discrete convection term. To obtain a compatible flow discretization in the sense of \cite{DawsonSunWheeler2004}, we project the velocity into an H(div) conforming space. Our numerical examples demonstrate the importance of H(div) conforming velocity for an accurate numerical concentration.

Another contribution of our work is a compactness result for time-dependent HDG approximations, which is a by-product of our analysis and which can be applied to other time-dependent nonlinear PDEs.
As in previous works \cite{Bartels:2009,Li:2015,riviere_walkington_2011}, our convergence proof relies on a strong compactness result for the sequence of approximate concentrations. However, in those works, it is implicitly used that this sequence of concentrations is bounded uniformly in the space $\mathrm{BV}(\Omega)$. This fact is well established for interior penalty type stabilizations \cite{lew2004optimal,Buffa:2009,DiPietro:2010}, and extensions to Lehrenfeld--Sch\"oberl type stabilizations \cite{lehrenfeld2010hybrid} are possible using, e.g., the techniques of \cite{cockburnH1}. This permits a direct use of the Aubin--Lions lemma in the case of low order time discretizations \cite{Bartels:2009} or a slight generalization to the non-conforming setting in the case of higher order time discretizations \cite{Li:2015,riviere_walkington_2011}. Because the classic HDG stabilization term is independent of the mesh size, the BV bound does not hold. Instead, following the classic work of Simon \cite{Simon:1986}, we prove a uniform bound on time translations and deduce compactness \`{a} la Kolmogorov (see e.g. \cite{Brezis:book,Folland:book}). This requires a discrete Ehrling lemma as in \cite{droniou2018gradient}, which in turn relies on a recently established Rellich--Kondrachov type result for the HDG method \cite{Jiang:2023}.

The miscible displacement problem has been extensively studied. While the literature is vast on the convergence analysis of various numerical methods (finite element methods, mixed finite element methods, discontinuous Galerkin methods, virtual elements, etc.) in the case of sufficiently smooth weak solutions or in the case of the constant or uniformly bounded dispersion-diffusion matrix \cite{EwingMfw1980,douglasEwingWheeler1983,Russell85,EpshteynRiviere2008,Beirao2021,droniou2019unified}, the general case is however treated in a few papers only. 
Finite element with mixed finite elements are analyzed in \cite{riviere_walkington_2011}; mixed finite elements with discontinuous Galerkin in \cite{Bartels:2009,Jensen2010,Li:2015,Girault:2016}.

Hybridization is a popular technique because of its computational efficiency \cite{cockburn2009hybridizable,cockburn2023hybridizable}.  With static condensation, the resulting linear system that involves the facet degrees of freedom only, is a much smaller system than the one obtained with non-hybridized methods. Once the facet unknowns are solved for, an embarrassingly parallel algorithm can be used to recover the interior degrees of freedom. HDG methods have been applied to a variety of problems. For coupled flow and transport problems, we refer the reader to \cite{cesmelioglu2023hybridizable, kirk2024combinedflow}.  The one-way coupled flow and transport problem is discretized by a combination of hybrid mixed and HDG methods in
\cite{kirk2024combinedflow} and error bounds are derived for solutions with enough regularity.  The paper \cite{cesmelioglu2023hybridizable} contains the analysis of an (interior penalty) HDG method for a multiphysics problem fully coupling Stokes, Darcy and transport. Convergence of the method is shown by deriving a priori error estimates for solutions with enough regularity.

The outline of the paper is as follows: in the next section, the model problem and its weak solution are presented. Section~\ref{sec:scheme} introduces the numerical scheme for the flow and transport problems. Well-posedness is proved in Section~\ref{sec:wellpos}. To prove convergence, compactness is established in Section~\ref{sec:compact}, and passing to the limit in the discrete formulation is done in Section~\ref{sec:limit}. Finally, numerical examples show the accuracy and robustness of the proposed scheme. Conclusions follow. 

\section{The miscible displacement problem}
Let $\Omega \subset \mathbb{R}^d, d=2,3$ be a bounded polygonal or polyhedral Lipschitz domain.  
The miscible displacement problem in $\Omega\times (0,T)$ is modeled by the following coupled equations
\begin{align} \label{eq:Darcy_prob_a}
		\bm{u} &= - \bm{K}(\cdot,c) \nabla p,& \\
		\nabla \cdot \bm{u} &= f_I - f_P,&  \label{eq:Darcy_prob_b}\\
  \phi \partial_t c - \nabla \cdot (\bm{D}(\bm{u})\nabla c - \bm{u}c)  &= f_I \, \overline{c} - f_P \,c. \label{eq:transport_prob} 
\end{align}

\begin{assumption}
Throughout, we make the following assumptions on the data: \vspace{2.5mm}
	\begin{enumerate}[noitemsep,topsep=2pt,leftmargin=!,labelwidth=\widthof{(ii)}]
		\item[(i)] The injection and production functions satisfy $f_I, f_P \in L^\infty(0,T;L^2(\Omega))$, $f^I, f^P \ge 0$, and for a.e. $t \in [0,T]$,
		\begin{equation}\label{eq:injectioncond}
		\int_\Omega f_I (x,t) \dif x = \int_\Omega f_P (x,t) \dif x.
		\end{equation}
		\item[(ii)] There exist constants $0 < \phi_0 < \phi_1$ such that the porosity $\phi \in L^\infty(\Omega)$ satisfies
		\begin{equation}\label{eq:porosity}
		\phi_0 \le \phi(\bm{x}) \le \phi_1, \quad \text{for a.e. } \bm{x} \in \Omega.
		\end{equation}
		\item[(iii)] The concentration of the injected solvent $\overline{c}$ belongs to $ L^2(0,T;L^{2d/(d-1)}(\Omega))$. 
		\item[(iv)]  The matrix $\bm{K}: \Omega \times \mathbb{R} \to \mathbb{R}^{d\times d}$ is symmetric, Carath\'{e}odory, uniformly bounded, and elliptic. Thus, there exist constants $0 < k_0 < k_1$ such that
		\[
		k_0 |\bm \xi|^2 \le \bm \xi^T \bm K(x, c) \bm \xi \le k_1 |\bm \xi|^2, \quad (x,c) \in \Omega \times \mathbb{R}, \, \bm \xi \in \mathbb{R}^d.
		\]
		 We assume that $\bm{K}(x,c) = \frac{1}{\mu(c)} \bm{\kappa}(x)$, where $\bm \kappa\in L^\infty(\Omega)^{d\times d}$ is the permeability field and $\mu(c)$ is the viscosity of the fluid mixture. For ease of notation, we suppress the spatial dependence of $\bm K$ below.
		\item[(v)] $\bm{D}: \mathbb{R}^d \to \mathbb{R}^{d\times d}$ is symmetric, Lipschitz continuous, and there exist constants $0 < d_0 < d_1$ such that
		\begin{equation}\label{eq:diffdisp}
		d_0 (1+|\bm u|) |\bm \xi|^2 \le \bm \xi^T \bm D(x,\bm u) \bm \xi \le d_1(1+|\bm u|) |\bm \xi|^2, \quad (x,\bm u) \in \Omega \times \mathbb{R}^d, \, \bm \xi \in \mathbb{R}^d.
		\end{equation}
		A typical example of the matrix $\bm D$ is:
  \begin{equation}
		\bm D(\bm{u}) = d_0 \bm{I} + |\bm{u}|\del{\alpha_l E(\bm{u}) + 
			\alpha_t(\bm{I} - E(\bm{u}))},
	\end{equation}
	where $E(\bm{u}) = \bm{u} \bm{u}^T/ |\bm{u}|^2$ and 
	$|\bm{u}|$ denotes the
	Euclidean norm of $\bm{u}$. 
	\end{enumerate}
\end{assumption}

Before introducing the HDG discretization of the transport problem \eqref{eq:transport_prob}, we follow \cite{Bartels:2009,Li:2015} by writing the first-order terms in half primal and half dual form. This is essential to ensure unconditionally stability of the numerical scheme. Formally, the following relation holds:
\[
2 \bm u \cdot \nabla c = \bm u \cdot \nabla c + \nabla \cdot(\bm u c) - (\nabla \cdot \bm u )c= \bm u \cdot \nabla c + \nabla \cdot (\bm u c) - (f_I - f_P) c,
\]
and consequently \eqref{eq:transport_prob} can be rewritten as:
\begin{equation} 
	\phi \partial_t c - \nabla \cdot (\bm{D}(\bm{u})\nabla c - \tfrac{1}{2}\bm{u}c) + \tfrac{1}{2} \bm u \cdot \nabla c + \tfrac{1}{2}(f_I + f_P) c= f_I \, \overline{c}.
	\quad 
	\text{ in } 
	\Omega \times (0,T).
\end{equation}

Typically, HDG methods for second-order elliptic problems are derived based on a mixed formulation, wherein a (formally) equivalent first order system is obtained by introduced an auxiliary variable $\bm q = \bm D(\bm u) \nabla c$ (see e.g. \cite{Fabien:2020a}). However, as the velocity $\bm u$ is not assumed bounded, $\bm |\bm D^{-1}(\bm u)|$ degenerates as $|\bm u| \to \infty$, leading to a loss of ellipticity. This makes the stability analysis particularly challenging. Instead, we propose an alternative three field formulation for the transport problem based on the introduction of \textit{two auxiliary variables} $\bm \theta$ and $\bm q$ satisfying the system
\begin{align*}
\bm{D}(\bm u) \bm \theta - \bm q &= 0, \\
\bm \theta &= - \nabla c.
\end{align*}
The transport equation can then be written as the following first order system:
\begin{subequations} \label{eq:transport_prob_rewritten} 
\begin{align}  \label{eq:transport_prob_rewritten_a} 
\bm D(\bm u) \bm \theta - \bm q = 0,& \quad \text{ in } \Omega \times (0,T), \\  \label{eq:transport_prob_rewritten_b} 
\bm \theta + \nabla c = 0,& \quad \text{ in } \Omega \times (0,T), \\  \label{eq:transport_prob_rewritten_c} 
		\phi \partial_t c + \nabla \cdot (\bm q+ \tfrac{1}{2}\bm{u}c) + \tfrac{1}{2} \bm u \cdot \nabla c + \tfrac{1}{2}(f_I + f_P) c= f_I \, \overline{c},& \quad \text{ in } \Omega 
		\times (0,T).
\end{align}
\end{subequations}
The boundary conditions and initial condition are given by
\begin{equation} \label{eq:transport_prob_rewritten_BC}
	\begin{split}
		\bm{u} \cdot \bm{n} &= 0, \quad \text{ on } \partial \Omega \times (0,T), \\
		\bm{\theta} \cdot \bm{n} &= 0, \quad \text{ on } \partial \Omega \times (0,T), \\
		c(\cdot,0) &= c_0, \quad \text{ in } \Omega \times \cbr{0}. \\
	\end{split}
\end{equation}
The pressure is unique up to a constant. To fix the constant, we assume $p \in L_0^2(\Omega)$, the subspace of $L^2(\Omega)$ functions with vanishing mean over $\Omega$. 
The weak formulation of \eqref{eq:Darcy_prob_a}-\eqref{eq:transport_prob} is now given: 
\begin{definition}[Weak formulation]
	A triplet $(\bm u, p, c)$ in $L^\infty(0,T;H_0({\rm div};\Omega)) \times L^\infty(0,T;L_0^2(\Omega)) \times (L^2(0,T;H^1(\Omega))\cap H^1(0,T;W^{1,2d}(\Omega)^\star)$ is said to be a weak solution of the system \eqref{eq:Darcy_prob_a}-\eqref{eq:transport_prob} if, for all $(\bm r, s) \in L^2(0,T;H_0({\rm div};\Omega)) \times L^2(0,T;L^2(\Omega))$ and for all $w \in H^1(0,T;H^2(\Omega)) \cap H^1(0,T;H^1(\Omega)^\star)$, it holds that
	\begin{align}
		\int_0^T \del{ (\bm{K}^{-1}(c) \bm u, \bm r)_\Omega - (p,\nabla \cdot \bm r)_\Omega }\dif t = 0, &\label{eq:weak1}\\
		\int_0^T (\nabla \cdot \bm u, s)_\Omega \dif t = \int_0^T (f_I - f_P, s)_\Omega &\dif t, \label{eq:weak2} \\
		\int_0^T \left(\langle \phi \partial_t c, w\rangle_{W^{1,2d}(\Omega)^*,W^{1,2d}(\Omega)} + (\bm{D}(\bm u) \nabla c, \nabla w)_\Omega - (\bm u  c, \nabla w)_\Omega +  (f_I c,w)_\Omega \right) &\dif t \nonumber\\ = 
  \int_0^T(f_I \overline{c},w)_\Omega &\dif t. \label{eq:weak3}
	\end{align}
\end{definition}
Various weak formulations of the miscible displacement problem have appeared in the literature with other spaces such as $H^2(\Omega)^*$ or  $W^{1,4}(\Omega)^*$, see e.g. \cite{Feng:1995,Bartels:2009,riviere_walkington_2011}. Our analysis below requires an HDG Poincar\'e inequality in $L^p(\Omega)$ with a tighter restriction on the range of $p$ than its continuous analogue, which
implies the use of the space $W^{1,2d}(\Omega)^*$.

\section{The numerical scheme}
\label{sec:scheme}

In this section, we introduce notation, the discrete spaces and the proposed numerical schemes.

\subsection{The finite element spaces}

The domain $\Omega$ is partitioned into a shape-regular simplicial mesh $\mathcal{E}_h$ with $h$
denoting the maximum of all element diameters $h_E$.  The boundary of an element $E$ is denoted by $\partial E$ and
its unit outward normal vector by $\bm n_E$.  The union of all the faces in the mesh is $\partial \mathcal{E}_h$
and the union of interior faces is $\partial\mathcal{E}_h^{\mathrm{int}}$. 

For any integer $k\geq 0$, the discrete spaces for the flow and transport equations are defined below. 

\begin{align*}
	\bm{V}_h &= \cbr{ \bm{v}_h \in (L^2(\Omega))^d \, : \, 
		\bm{v}_h|_{E} \in (\mathbb{P}_k(E))^d, \, \forall E \in \mathcal{E}_h},\\
    W_h &= \cbr{ w_h \in L^2(\Omega) \, : \, w_h|_E \in \mathbb{P}_{k}(E), \, \forall E \in \mathcal{E}_h},\\
    Q_h &= W_h \cap L_0^2(\Omega),\\
	M_h &= \cbr{ \widehat{w}_h \in L^2(\partial \mathcal{E}_h) \, : \, \widehat{w}_h|_e \in \mathbb{P}_k(e), \, \forall e 
		\in \partial \mathcal{E}_h}
\end{align*}

The $L^2$ inner-product on $\Omega$ is denoted by $(\cdot,\cdot)_\Omega$. We also denote by $(\cdot,\cdot)_E$ the $L^2$ inner-product on an element $E$ and by $\langle\cdot,\cdot\rangle_e$ the
$L^2$ inner-product on a face $e\subset\partial E$.  We will use the usual short-hand notation:
\[
(w,v)_{\mathcal{E}_h}  = \sum_{E\in\mathcal{E}_h} (w, v)_E, \quad
\langle w, v\rangle_{\partial\mathcal{E}_h} = \sum_{E\in\mathcal{E}_h} \langle w, v\rangle_{\partial E},
\]
with the usual modifications for vector-valued functions. The corresponding norms are 
$\Vert \cdot \Vert_{L^2(\Omega)}$ and $\Vert \cdot\Vert_{L^2(\partial\mathcal{E}_h)}$.
The jump of a scalar function $w$ across a face $e$ is denoted by $\jump{w}$; it is
uniquely defined by fixing a normal vector for each face $e$.

Let $\tau>0$ be the time step value and let $t^i = i \tau$
be the $i$-th discrete time such that $0<t^1<\dots<t^N = T$ forms a uniform partition of the time interval. The functions $f_P, f_I, \overline{c}$ evaluated at $t=t^i$ are denoted by
$f_P^i, f_I^i$ and $\overline{c}^i$ respectively.

We will also use the notation $A \lesssim B$ to denote that $A\leq C B$ with a constant $C$ independent of $h$ and $\tau$.

\subsection{The HDG method for the flow problem}
Our proposed discretization for the Darcy problem \crefrange{eq:Darcy_prob_a}{eq:Darcy_prob_b} can be cast as: given $c_h^{i-1} \in W_h$, find $(\bm u_h^i, p_h^i, \widehat p_h^{\,i}) \in \bm V_h \times Q_h \times M_h$ such that
\begin{subequations} \label{eq:HMHDG_darcy_step}
	\begin{align} 
		(\bm{K}^{-1}(c_h^{i-1}) \bm{u}_h^i, \bm{r}_h )_{\mathcal{E}_h} - (p_h^i, 
		\nabla
		\cdot \bm{r}_h)_{\mathcal{E}_h} + \langle \widehat{p}_h^{\,i}, \bm{r}_h \cdot 
		\bm{n} 
		\rangle_{\partial \mathcal{E}_h} &= 0, \label{eq:HMHDG_darcy_step_a}\\
		(\nabla \cdot\bm{u}_h^i, s_h)_{\mathcal{E}_h}   + \langle 
		\sigma_{u}(p_h^i - \widehat{p}_h^{\,i}), s_h \rangle_{\partial \mathcal{E}_h} &= (f_I^i - 
		f_P^i,s_h)_{\mathcal{E}_h},
		\label{eq:HMHDG_darcy_step_b} \\ 
		\langle \bm{u}_h^i \cdot 
		\bm{n} +
		\sigma_{u}(p_h^i - \widehat{p}_h^{\,i}),  \widehat{s}_h \rangle_{\partial \mathcal{E}_h} &= 0, \label{eq:HMHDG_darcy_step_c} 
	\end{align}
\end{subequations}
for all $(\bm r_h, s_h, \widehat s_h) \in \bm V_h \times Q_h \times M_h$. Here, $\sigma_u \in L^\infty(\partial \mathcal{E}_h)$ is a stabilization function. As $\bm{K}(\cdot)$ is uniformly elliptic, it is well known \cite{Cockburn:2008} that system \crefrange{eq:HMHDG_darcy_step_a}{eq:HMHDG_darcy_step_c} is well-posed provided $\sigma_u > 0 \text{ for a.e. } x \in \partial \mathcal{E}_h$. We suppose that $\sigma_u$ is a positive constant in the remainder of the article to simplify the presentation.

Due to the coupling between flow and transport, a compatibility condition is needed for the discrete velocity \cite{DawsonSunWheeler2004}.  
In particular, an $H(\text{div};\Omega)$-conforming approximation to the velocity can be obtained from $\bm u_h^i$ via the following element-wise post-processing \cite{Cockburn:2009b}: find $\bm U_h^{i} \in \text{RT}_k(\Omega)$ (Raviart-Thomas space) such that for all $E \in \mathcal{E}_h$,
\begin{subequations}
\begin{align}
	(\bm U_h^{i}, \bm v_h)_{E} &= (\bm u_h^i, \bm v_h)_E, \quad \forall \bm v_h \in (\mathbb{P}_{k-1}(E))^d, \label{eq:post_process_a} \\
	\langle \bm U_h^{i} \cdot \bm n, \mu_h\rangle_e = &  \langle \bm u_h^i \cdot \bm n
+\sigma_u (p_h^i -\widehat{p}_h^{\, i}), \mu_h \rangle_{e},  \quad \forall \mu_h\in \mathbb{P}_k(e), \forall e \subset \partial E. \label{eq:post_process_b}
\end{align}
\end{subequations}
The $L^2$ projection on $W_h$ is denoted by $\pi_k$ and the $L^2$ projection on $M_h$ is denoted by $\widehat{\pi}_k$. 
It is easy to check that $\nabla \cdot \bm U_h^i$ is the $L^2$ projection of $f_I^i-f_P^i$ onto $Q_h$:
\begin{equation}\label{eq:projdivu}
\nabla \cdot \bm U_h^i = \pi_k (f_I^i-f_P^i).
\end{equation}

\subsection{The HDG method for the transport problem}
Given $c_h^{i-1} \in W_h$, find $(\bm \theta_h^i, \bm q_h^i, c_h^i, \widehat{c}_h^{\,i}) \in \bm V_h \times \bm V_h \times W_h \times M_h$ such that
\begin{subequations} \label{eq:discrete_transport_step}
	\begin{align}
		(\bm D(\bm{U}_h^i) \bm{\theta}_h^i, \bm{z}_h)_{\mathcal{E}_h} - (\bm q_h^i, \bm z_h)_{\mathcal{E}_h}= 0, \label{eq:discrete_transport_step_a}\\
		(\bm{\theta}_h^i, \bm{v}_h)_{\mathcal{E}_h} - (c_h^i, \nabla \cdot \bm v_h)_{\mathcal{E}_h} + \langle \widehat c_h^{\,i}, \bm v_h \cdot \bm n \rangle_{\partial \mathcal{E}_h}= 0, \label{eq:discrete_transport_step_b}\\
		(\phi \delta_\tau c_h^i, w_h)_{\mathcal{E}_h} - 
		( \bm{q}_h^i + \tfrac{1}{2}\bm{U}_h^i c_h^i, \nabla w_h)_{\mathcal{E}_h}   + \tfrac{1}{2} (\bm U_h^i \cdot \nabla c_h^i, w_h)_{\partial \mathcal{E}_h} - \tfrac{1}{2}\langle \bm{U}_h^i \cdot \bm n c_h^i, w_h \rangle_{\partial \mathcal{E}_h} 
		\label{eq:discrete_transport_step_c}\\  
+  \langle 
		(\widehat{\bm{q}}^i_h + \tfrac{1}{2} \widehat {\bm{U}_h^i c_h^i}\cdot \bm{n} , w_h 
		\rangle_{\partial 
			\mathcal{E}_h}  + \tfrac{1}{2}((f_I^i + f_P^i)c_h^i, w_h)_{\mathcal{E}_h} =  (f_I^i 
		\overline{c}^i,w_h)_{\mathcal{E}_h}, \notag \\
		 \langle 
		(\widehat{\bm{q}}_h^i + \tfrac{1}{2} \widehat {\bm{U}_h^i c_h^i}\cdot \bm{n} , \widehat w_h 
		\rangle_{\partial 
			\mathcal{E}_h}  = 0, 
		\label{eq:discrete_transport_step_d}
	\end{align}
\end{subequations}
for all test functions $(\bm z_h, \bm v_h, w_h, \widehat w_h) \in \bm V_h \times \bm V_h \times W_h \times M_h$. In the above,  we use the shorthand notation: \[
\delta_\tau c_h^i = \frac{c_h^{i}-c_h^{i-1}}{\tau}.
\]
We select the numerical fluxes as
\begin{align*}
\widehat{\bm{q}}_h^{\, i} \cdot \bm n & = \bm q_h^i \cdot \bm n + \sigma_D^i (c_h^i - \widehat{c}_h^{\,i}),\\
\tfrac{1}{2} \widehat{\bm{U}_h^i c_h^i} \cdot \bm n & = \tfrac{1}{2} \bm{U}_h^i \cdot \bm n (c_h^i + \widehat c_h^{\,i})+  |\bm{U}_h^i \cdot \bm n |(c_h^i - \widehat c_h^{\,i}).
\end{align*}
where the diffusive stabilization function is taken to be
\begin{equation}\label{eq:diff_stab_func}
\sigma_D^i = \bm n^T \bm D(\bm U_h^i) \bm n  \geq d_0 >0.
\end{equation}
To initialize the scheme, we define $c_h^0 = \pi_k(c_0)$, $\widehat{c}_h^{\, 0} = \widehat{\pi}_k c_0$.

\begin{remark}
    It is well known that $\sigma_u = \mathcal{O}(1), \sigma_D^i = \mathcal{O}(1)$ is optimal for systems \crefrange{eq:HMHDG_darcy_step_a}{eq:HMHDG_darcy_step_c} and \crefrange{eq:discrete_transport_step_a}{eq:discrete_transport_step_d}\cite{Cockburn:2009}. This is in contrast to the family of interior penalty HDG methods \cite{Cockburn:2009,Fabien:2020b,Wells:2011}, the Lehrenfeld--Sch\"{o}berl HDG methods \cite{lehrenfeld2010hybrid,cockburnH1,Sayas:book}, and the HHO methods \cite{cockburn2016bridging,Pietro:book2} where $\sigma_u = \mathcal{O}(h^{-1}), \sigma_D^i = \mathcal{O}(h^{-1})$ is a typical choice. 
\end{remark}

\section{Discrete existence, uniqueness, and stability}
\label{sec:wellpos}

We begin by introducing in \Cref{ss:HDG_gradient} a number of technical results from \cite{Jiang:2023,Yue:2024} required for our analysis. In \Cref{ss:flow_stab}, we prove the stability and well-posedness of the HDG method for the Darcy problem. Finally, \Cref{ss:trans_stab} is devoted to the stability and well-posedness of the HDG scheme for the transport problem.
\subsection{The HDG Gradient}
\label{ss:HDG_gradient}

Following the ideas of \cite{Buffa:2009,DiPietro:2010,Jiang:2023}, we introduce a discrete HDG distributional gradient $\bm G_h : H^1(\mathcal{E}_h)\times L^2(\partial \mathcal{E}_h)\to \bm V_h$ via the lifting 
\begin{align} \label{eq:disc_grad}
	(\bm G_h( w, \widehat w), \bm v_h)_{\mathcal{E}_h} = (\nabla w, \bm v_h)_{\mathcal{E}_h} - \langle w - \widehat{w}, \bm v_h \cdot \bm n \rangle_{\partial \mathcal{E}_h}, \quad \forall \bm v_h \in\bm V_h.
\end{align}
It is obvious that $\bm G_h(w, \widehat w)$ exists and is uniquely defined. 
We remark that analogous HDG gradients have appeared previously in  \cite{Cockburn:nonlin,Kikuchi:2012,Kirk:2023a}.

For ease of notation, we define the following norms on $W_h \times M_h$ and $Q_h \times M_h$ that serve as discrete analogues of the norms on $H^1(\Omega)$ and $H^1(\Omega)\slash \mathbb{R}$, respectively:
\begin{align} \label{eq:disc_h1_norm}
\norm{ (w_h, \widehat{w}_h)}_{1,h} &= \del{\norm{w_h}_{L^2(\Omega)}^2 + \norm{\bm G_h(w_h,\widehat{w}_h)}_{L^2(\Omega)}^2 + \norm{w_h - \widehat w_h}_{L^2(\partial \mathcal{E}_h)}^2}^{1/2},\\
\norm{ (w_h, \widehat{w}_h)}_{1,h,0} &= \del{\norm{\bm G_h(w_h,\widehat{w}_h)}_{L^2(\Omega)}^2 + \norm{w_h - \widehat w_h}_{L^2(\partial \mathcal{E}_h)}^2}^{1/2}. 
\end{align}

\begin{lemma}[HDG Poincar\'{e} inequality] \label{lem:HDG_poincare}
There exists a constant $C>0$ such that, for all $(w_h,\widehat w_h) \in W_h \times M_h$, it holds that
\begin{equation} \label{eq:disc_embed}
		\norm{w_h}_{L^p(\Omega)} \le C \norm{ (w_h, \widehat{w}_h)}_{1,h}, \quad 2 \le p \le \frac{2d}{d-1}.
\end{equation}
Moreover, if $(w_h,\widehat w_h) \in Q_h \times M_h$, the bound for $p=2$ can be sharpened to
\begin{equation}\label{eq:poinc_on_Qh}
	\norm{w_h}_{L^2(\Omega)} \le C \norm{ (w_h, \widehat{w}_h)}_{1,h,0}.
\end{equation}
\end{lemma}
The proof is in \Cref{ss:appendix_poincare}.
 For later use, we recall a compactness result for HDG discretizations \cite[Lemma 4.4]{Jiang:2023}:
\begin{lemma}\label{lem:HDG_RK}
	Let $\mathcal{E}_h$ be a shape-regular triangulation of a bounded Lipschitz domain $\Omega \subset \mathbb{R}^d$, and $\bm G_h : H^1(\mathcal{E}_h) \times L^2(\partial \mathcal{E}_h) \to \bm V_h$ the discrete HDG gradient defined in \eqref{eq:disc_grad}. 
    Suppose that $\mathcal{H}$ is a countable family of mesh sizes whose unique accumulation point is $0$, $\cbr{(\mathcal{E}_h)}_{h \in \mathcal{H}}$ is a corresponding family of shape-regular triangulations of $\Omega$, and $(w_h, \widehat w_h) \in W_h \times M_h$ is a sequence such that $\norm{(w_h,\widehat{w}_h)}_{1,h}$ is uniformly bounded with respect to $h$, then there exists a (not relabeled) subsequence and a function $w \in H^1(\Omega)$ such that as $h\to 0$,
	\[
	w_h \to w \text{ in } L^2(\Omega), \quad \bm G_h(w_h, \widehat w_h) \rightharpoonup \nabla w \text{ in } (L^2(\Omega))^d, \quad \widehat{w}_h |_{\partial \Omega} \rightharpoonup w|_{\partial \Omega} \text{ in } L^2(\partial \Omega).
	\]
\end{lemma}	
\begin{proof}
 	See \Cref{ss:append_compact}.
\end{proof}

\subsection{Stability bounds for the flow problem}
\label{ss:flow_stab}

\begin{theorem} \label{lem:stab_flow}
	Suppose the triplet $(\bm u_h^i, p_h^i, \widehat p_h^{\,i}) \in  \bm V_h \times Q_h \times M_h$ solves the discrete Darcy problem \crefrange{eq:HMHDG_darcy_step_a}{eq:HMHDG_darcy_step_c} for $1 \le i \le N$. Then, there exists a constant $C>0$, independent of $h$ and $\tau$, such that
	\begin{multline} \label{eq:flow_stab_LDG-H}
	\max_{1 \le i \le N} \del{\norm[0]{\bm u_h^i}_{L^2(\Omega)} + \norm[0]{\bm G_h(p_h^i,\widehat p_h^{\,i})}_{L^2(\Omega)} + \norm[0]{ \sigma_u^{1/2}(p_h^i - \widehat p_h^{\,i})}_{L^2(\partial \mathcal{E}_h)}} \\ \le C\norm[0]{f_I - f_P}_{\ell^\infty(0,T;L^2(\Omega))}.
	\end{multline}
	Moreover, the post-processed velocity $\bm U_h^{i}$ satisfies
	\begin{equation} \label{eq:flow_stab-LDG-H-star}
	\max_{1 \le i \le N} \norm[0]{\bm U_h^{i}}_{L^2(\Omega)} \le C\norm[0]{f_I - f_P}_{\ell^\infty(0,T;L^2(\Omega))}.
	\end{equation}
\end{theorem}

\begin{proof}
 For fixed $1 \le i \le N$, set $(\bm r_h, s_h, \widehat s_h) = (\bm u_h^i, p_h^i, -\widehat p_h^i)$ in \crefrange{eq:HMHDG_darcy_step_a}{eq:HMHDG_darcy_step_c} and sum the resulting equations to find
    \begin{equation}\label{eq:flow_conv_pass}
			k_1^{-1} \norm[0]{\bm u_h^i}_{L^2(\Omega)}^2 + \norm[0]{ 
			\sigma_{u}^{1/2}(p_h^i - \widehat{p}_h^{\,i})}_{L^2(\partial \mathcal{E}_h)}^2 \le (f_I^i - 
			f_P^i,p_h^i)_{\mathcal{E}_h}.
	\end{equation}
From \cref{lem:HDG_poincare}, it holds that
\[
\Vert p_h^i \Vert_{L^2(\Omega)}^2 \lesssim \Vert\bm G_h(p_h^i,\widehat{p}_h^{\,i})\Vert_{L^2(\Omega)}^2 + \Vert p_h^i - \widehat p_h^{\,i}\Vert_{L^2(\partial \mathcal{E}_h)}^2.
\]
Now, note that by \eqref{eq:HMHDG_darcy_step_a} and by the definition of the HDG gradient, $\bm G(p_h^i,\widehat p_h^{\, i})$ satisfies:
\begin{equation}\label{eq:flowwithG}
(\bm G_h(p_h^i,\widehat p_h^{\,i}),\bm v_h)_{\mathcal{E}_h}  = - (\bm K^{-1}(c_h^{i-1}) \bm u_h^i, \bm v_h)_{\mathcal{E}_h}, \quad \forall \bm v_h\in \bm V_h.
\end{equation}
Choosing $\bm v_h = \bm G_h(p_h^i,\widehat p_h^{\,i})$ above and using the boundedness of $\bm K^{-1}$ yields
\[
\norm[1]{\bm G_h(p_h^i,\widehat{p}_h^{\, i})}_{L^2(\Omega)}  \lesssim \norm[0]{\bm u_h^i}_{L^2(\Omega)}.
\]
The result then follows by applying Cauchy--Schwarz's inequality and Young's inequality to \eqref{eq:flow_conv_pass}. 
		As for \cref{eq:flow_stab-LDG-H-star}, the triangle inequality and \cref{eq:flow_stab_LDG-H} yield
		\begin{equation}
			\norm[0]{ \bm U_h^{i} }_{L^2(\Omega)} \lesssim  \norm[0]{ \bm U_h^{i} - \bm u_h^i }_{L^2(\Omega)}  + \norm[0]{f_I^i-f_P^i}_{L^2(\Omega)}.
		\end{equation}
		Observe that $\bm \eta_h^i = \bm U_h^{i} - \bm u_h^i$ satisfies for all $E \in \mathcal{E}_h$,
		\begin{align*}
			(\bm \eta_h^i, \bm v_h)_{E} &= 0, &&\forall \bm v_h \in (\mathbb{P}_{k-1}(E))^d, \\
			\langle \bm \eta_h^i \cdot \bm n, \widehat{w}_h \rangle_{e} &= \langle \sigma_u(p_h^i - \widehat p_h^{\,i}), \widehat{w}_h \rangle_{e}, && \forall \widehat{w}_h \in \mathbb{P}_k(e), \forall e \subset \partial E,
		\end{align*}
		and thus by a finite-dimensional scaling argument (see \Cref{app:scaling} for details), 
		\begin{equation} \label{eq:par_el_calc}
			\norm[0]{ \bm \eta_h^i}_{L^2(\Omega)} \lesssim h^{1/2}  \norm[0]{\sigma_u^{1/2}(p_h^i - \widehat{p}_h^{\,i})}_{L^2(\partial \mathcal{E}_h)}.
		\end{equation}
		Consequently, 
        \cref{eq:flow_stab-LDG-H-star}  follows from \cref{eq:flow_stab_LDG-H}.
\end{proof}
		An immediate consequence of \Cref{lem:stab_flow} is the following:
		\begin{corollary}
			For each $1 \le i \le N$, there exists a unique solution $(\bm u_h^i, p_h^i, \widehat p_h^i) \in \bm{V}_h \times Q_h \times M_h$ to the discrete Darcy problem \crefrange{eq:HMHDG_darcy_step_a}{eq:HMHDG_darcy_step_c}.
		\end{corollary}
		\begin{proof}
			The result follows from \cref{eq:flow_stab_LDG-H} and the fact that $Q_h \subset L_0^2(\Omega)$, since for each fixed $c_h^{i-1} \in W_h$,  \crefrange{eq:HMHDG_darcy_step_a}{eq:HMHDG_darcy_step_c} is a square linear system in the finite dimensional space $\bm{V}_h \times Q_h \times M_h$. Indeed, if $f_I^i = f_P^i = 0$, then $\bm u_h^i = \bm 0$, $p_h^i = 0$, and $\widehat{p}_h^{\,i} = 0$.
		\end{proof}

\subsection{Stability for the transport problem}

\label{ss:trans_stab}
\begin{theorem} \label{lem:stab_transport}
Given an integer $1 \le m \le N$, 
the quadruplet $(\bm \theta_h^i, \bm q_h^i, c_h^i, \widehat c_h^{\,i}) \in \bm V_h \times \bm V_h \times W_h \times M_h$ solves
\crefrange{eq:discrete_transport_step_a}{eq:discrete_transport_step_d} for the transport problem \crefrange{eq:transport_prob_rewritten_a}{eq:transport_prob_rewritten_c}, \cref{eq:transport_prob_rewritten_BC}. Then, for all $1 \le m \le N$, it holds that
	\begin{align}
		\phi_0 \norm[1]{ c_h^m}_{L^2(\Omega)}^2 
+ \tau \sum_{i=1}^{m} \del{ \norm[1]{ (f_I^i)^{1/2} c_h^i}_{L^2(\Omega)}^2 + \Vert (\sigma_D^i +  |\bm{U}_h^i \cdot \bm n | )^{1/2} (c_h^i - \widehat{c}_h^{\,i})\Vert_{L^2(\partial\mathcal{E}_h)}^2}\nonumber\\
+ \tau \sum_{i=1}^{m} \del{ \tau \norm[1]{\phi^{1/2}\delta_\tau c_h^i}_{L^2(\Omega)}^2 + \norm[1]{\bm{D}^{1/2}(\bm U_h^i)\bm \theta_h^i}_{L^2(\Omega)}^2 }
				  \le   \phi_1 \norm[1]{  c_0}_{L^2(\Omega)}^2 \nonumber\\
      + \tau \sum_{i=1}^m \norm[1]{ (f_I^i)^{1/2} \overline{c}^i }_{L^2(\Omega)}^2.
      \label{eq:boundconc}
	\end{align}
    In addition, there exists a constant $C>0$, independent of $h$ and $\tau$, such that
  \begin{equation}\label{eq:norm1hconc}
  \tau \sum_{i=1}^N \Vert (c_h^i, \widehat c_h^{\,i})\Vert_{1,h}^2\lesssim C.
  \end{equation}
\end{theorem}

\begin{proof}
Note that \eqref{eq:discrete_transport_step_b} is equivalent to
\begin{equation}\label{eq:firsteq}
(\bm{\theta}_h^i, \bm{v}_h)_{\mathcal{E}_h} + (\nabla c_h^i, \bm v_h)_{\mathcal{E}_h} - \langle c_h^i - \widehat c_h^{\,i} , \bm v_h \cdot \bm n \rangle_{\partial \mathcal{E}_h}= 0.
\end{equation}
Choosing $\bm{v}_h  = \bm q_h^i$ above and $(\bm{z}_h, w_h, \widehat w_h) = (\bm{\theta}_h^i, c_h^i, -\widehat{c}_h^{\,i})$
in \eqref{eq:discrete_transport_step_a}, \eqref{eq:discrete_transport_step_c}, \eqref{eq:discrete_transport_step_d} and summing the resulting equations,
	\begin{align*}
		(\bm D(\bm{U}_h^i) \bm{\theta}_h^i, \bm \theta_h^i)_{\mathcal{E}_h} +	(\phi \delta_\tau c_h^i, c_h^i)_{\mathcal{E}_h} + \tfrac{1}{2}((f_I^i + f_P^i)c_h^i, c_h^i)_{\mathcal{E}_h}  \\
+ \langle (\sigma_D^i + |\bm{U}_h^i \cdot \bm n |) (c_h^i - \widehat{c}_h^{\,i}) , c_h^i - \widehat{c}_h^{\,i} \rangle_{\partial \mathcal{E}_h}  
   - \tfrac{1}{2} \langle \bm{U}_h^i \cdot \bm n \, {c}_h^i , c_h^i 
		\rangle_{\partial 
			\mathcal{E}_h}  \\
+ \langle \tfrac{1}{2} \bm{U}_h^i \cdot \bm n (\widehat {c}_h^i + c_h^i), c_h^i  - \widehat{c}_h^{\,i}
		\rangle_{\partial 
			\mathcal{E}_h}=  (f_I^i \,
		\overline{c}^i,c_h^i)_{\mathcal{E}_h}.
	\end{align*}
	Now, observe that
	\[
	- \tfrac{1}{2} \langle \bm{U}_h^i \cdot \bm n \, {c}_h^i , c_h^i 
	\rangle_{\partial 
		\mathcal{E}_h}  + \langle 
	\tfrac{1}{2} \bm{U}_h^i \cdot \bm n (\widehat {c}_h^{\,i} + c_h^i), c_h^i  - \widehat{c}_h^{\,i}
	\rangle_{\partial 
		\mathcal{E}_h} = - \tfrac{1}{2} \langle \bm U_h^i \cdot \bm n \, \widehat c_h^{\,i}, \widehat c_h^{\,i} \rangle_{\partial \mathcal{E}_h},
	\]
	and therefore since $\bm U_h^i \in H(\text{div};\Omega)$, the term above vanishes and we are left with
	\begin{align*} 
		(\bm D(\bm{U}_h^i) \bm{\theta}_h^i, \bm \theta_h^i)_{\mathcal{E}_h} +	(\phi \delta_\tau c_h^i, c_h^i)_{\mathcal{E}_h} + \tfrac{1}{2}((f_I^i + f_P^i)c_h^i, c_h^i)_{\mathcal{E}_h}  \\ \notag
+ \langle \del[1]{ \sigma_D^i +   |\bm{U}_h^i \cdot \bm n | }(c_h^i - \widehat{c}_h^{\,i}) , c_h^i - \widehat{c}_h^{\,i}
		\rangle_{\partial 
			\mathcal{E}_h} =  (f_I^i 
		\overline{c}^i,c_h^i)_{\mathcal{E}_h}.
	\end{align*}
Using the following identity
 \[
 (\phi \delta_\tau c_h^i, c_h^i)_{\mathcal{E}_h}
 = \frac{1}{2\tau} \left( \Vert \phi^{1/2} c_h^i \Vert_{L^2(\Omega)}^2
 - \Vert \phi^{1/2} c_h^{i-1} \Vert_{L^2(\Omega)}^2\right)
 +\frac{\tau}{2} \Vert \phi^{1/2} \delta_\tau c_h^i\Vert_{L^2(\Omega)}^2,
 \]
 we obtain
	\begin{align*}
 \tau^2 \norm[1]{\phi^{1/2}\delta_\tau c_h^i}_{L^2(\Omega)}^2 +\norm[0]{\phi^{1/2} c_h^{i}}_{L^2(\Omega)}^2 - \norm[0]{\phi^{1/2} c_h^{i-1}}_{L^2(\Omega)}^2 +
		2\tau (\bm D(\bm{U}_h^i) \bm{\theta}_h^i, \bm \theta_h^i)_{\mathcal{E}_h} \\ + 2\tau((f_I^i + f_P^i)c_h^i, c_h^i)_{\mathcal{E}_h}  
+ 2\tau\langle \del[1]{ \sigma_D^i +   |\bm{U}_h^i \cdot \bm n | }(c_h^i - \widehat{c}_h^{\,i}) , c_h^i - \widehat{c}_h^{\,i}
		\rangle_{\partial \mathcal{E}_h}  \le 2\tau(f_I^i 
		\overline{c}^i,c_h^i)_{\mathcal{E}_h}. 
	\end{align*}
	Next, with the assumption on the data (see \eqref{eq:porosity} and \eqref{eq:diffdisp}), we obtain
	%
	%
	%
	\begin{align*}
 \tau^2 \norm[1]{\phi^{1/2}\delta_\tau c_h^i}_{L^2(\Omega)}^2 +\norm[0]{\phi^{1/2} c_h^{i}}_{L^2(\Omega)}^2 - \norm[0]{\phi^{1/2} c_h^{i-1}}_{L^2(\Omega)}^2 
		+ 2\tau \norm[1]{\bm{D}^{1/2}(\bm U_h^i)\bm \theta_h^i}_{L^2(\Omega)}^2 \\ + 2\tau \norm[1]{ (f_I^i)^{1/2} c_h^i}_{L^2(\Omega)}^2 
+ 2\tau\langle \del[1]{ \sigma_D^i +     |\bm{U}_h^i \cdot \bm n | }(c_h^i - \widehat{c}_h^{\,i}) , c_h^i - \widehat{c}_h^i
		\rangle_{\partial \mathcal{E}_h}  \le  2\tau(f_I^i 
		\overline{c}^i,c_h^i)_{\mathcal{E}_h}. 
	\end{align*}
	Thus, with $\sigma_D^i>0$ (see \eqref{eq:diff_stab_func}) and by  Cauchy--Schwarz's and Young's inequalities, we have
        	\begin{align*}
 \tau^2 \norm[1]{\phi^{1/2}\delta_\tau c_h^i}_{L^2(\Omega)}^2 +\norm[0]{\phi^{1/2} c_h^{i}}_{L^2(\Omega)}^2 - \norm[0]{\phi^{1/2} c_h^{i-1}}_{L^2(\Omega)}^2 
		+ \tau \norm[1]{\bm{D}^{1/2}(\bm U_h^i)\bm \theta_h^i}_{L^2(\Omega)}^2 \\ + \tau \norm[1]{ (f_I^i)^{1/2} c_h^i}_{L^2(\Omega)}^2 
+ \tau \Vert (\sigma_D^i +     |\bm{U}_h^i \cdot \bm n |)^{1/2} (c_h^i - \widehat{c}_h^{\,i}) \Vert_{L^2(\partial\mathcal{E}_h)}^2  
  \le  \tau \norm[1]{ (f_I^i)^{1/2} \overline{c}^i}_{L^2(\Omega)}^2. 
	\end{align*}
Therefore, we can obtain \eqref{eq:boundconc} by summing from $i=1$ to $i=m$ and by using the stability of the $L^2$ projection. 

To prove \eqref{eq:norm1hconc}, it then suffices to bound the discrete gradient part. We have already proved that for all $\bm v_h \in \bm V_h$,
 \[
 		(\bm{\theta}_h^i, \bm{v}_h)_{\mathcal{E}_h} 
 =    - (\nabla c_h^i, \bm v_h)_{\mathcal{E}_h} + \langle c_h^i - \widehat c_h^{\,i}, \bm v_h \cdot \bm n \rangle_{\partial \mathcal{E}_h} 
 		= -(\bm G_h(c_h^i, \widehat c_h^{\,i}), \bm v_h)_{\mathcal{E}_h}.
 \]
This implies that $\bm\theta_h^i = - \bm G_h(c_h^i,\widehat c_h^{\,i})$.  With \eqref{eq:boundconc}  and the fact that
 \[
 \Vert  \bm\theta_h^i\Vert_{L^2(\Omega)}^2 \leq \frac{1}{d_0}
 \Vert \bm D^{1/2}(\bm U_h^i)\bm\theta_h^i\Vert_{L^2(\Omega)}^2,
 \]
we immediately obtain the bound
 \[
	\tau \sum_{i=1}^N \Vert \bm G_h(c_h^i, \widehat c_h^{\,i})\Vert_{L^2(\Omega)}^2 \le  \frac{\phi_1}{d_0}  \norm[1]{  c_0}_{L^2(\Omega)}^2 + \frac{\tau}{d_0} \sum_{i=1}^N \norm[1]{ (f_I^i)^{1/2} \overline{c}^i}_{L^2(\Omega)}^2.
 \]

\end{proof}

An immediate consequence of \Cref{lem:stab_transport} is the following:
\begin{corollary}
	For any  $1 \le i \le N$,
	there exists a unique solution $(\bm \theta_h^i, \bm q_h^i, c_h^i, \widehat c_h^{\,i})$ in $ \bm{V}_h \times \bm V_h \times W_h \times M_h$ to the discrete transport problem \crefrange{eq:discrete_transport_step_a}{eq:discrete_transport_step_d}.
\end{corollary}
\begin{proof}
	If $\overline c^i = 0$ and $c_0 = 0$, \Cref{lem:stab_transport} yields $\bm \theta_h^i = \bm 0$, $c_h^i = 0$, and $\widehat c_h^{\,i} = 0$. Then, it follows from \cref{eq:discrete_transport_step_a} that $\bm q_h^i = \bm 0$. The result now follows since \crefrange{eq:discrete_transport_step_a}{eq:discrete_transport_step_d} is a square linear system in the finite dimensional space $\bm V_h \times \bm V_h \times W_h \times M_h$.
\end{proof}

We conclude this section by deriving a bound on $\bm q_h^i$:
\begin{lemma} \label{lem:flux_Lp_bnd}
	There exists a constant $C>0$, independent of the mesh-size $h$ and time-step $\tau$, such that
	\[
	\tau \sum_{i=1}^{N} \Vert \bm q_h^i \Vert_{L^{2d/(2d-1)}(\Omega)}^2  \le C.
	\]
\end{lemma}
\begin{proof}
	With the Riesz Representation Theorem (for Lebesgue spaces), \cref{eq:discrete_transport_step_a} yields for any $1 \le i \le N$,
	\begin{align*}
		\norm[1]{ \bm q_h^i}_{L^{2d/(2d-1)}(\Omega)} 
			& = \sup_{\bm 0 \ne \bm z \in (L^{2d}(\Omega))^d} \frac{ (\bm D(\bm U_h^i) \bm \theta_h^i, \bm \pi_k\bm z)_{\mathcal{E}_h}}{\norm{\bm z}_{L^{2d}(\Omega)}}.
	\end{align*}
By Cauchy-Schwarz's inequality, H\"older's inequality and \eqref{eq:diffdisp}, we have 
\begin{align*}
(\bm D(\bm U_h^i) \bm \theta_h^i, & \bm \pi_k\bm z)_{\mathcal{E}_h}
\leq \Vert \bm D^{1/2}(\bm U_h^i) \bm \theta_h^i\Vert_{L^2(\Omega)}
\Vert \bm D^{1/2}(\bm U_h^i) \bm \pi_k \bm z \Vert_{L^2(\Omega)}
\\
&\lesssim \Vert \bm D^{1/2}(\bm U_h^i) \bm \theta_h^i\Vert_{L^2(\Omega)}
(\Vert \bm \pi_k \bm z \Vert_{L^2(\Omega)}^2
+ \Vert \bm U_h^i \Vert_{L^{d/(d-1)}(\Omega)}
\Vert \bm \pi_k \bm z\Vert_{L^{2d}(\Omega)}^2)^{1/2}.
\end{align*}
Since $d/(d-1)\leq 2$ and the velocity $\bm U_h^i$ is uniformly bounded in $L^2(\Omega)$ by Theorem~\ref{lem:stab_flow}, 
\[
(\bm D(\bm U_h^i) \bm \theta_h^i, \bm \pi_k\bm z)_{\mathcal{E}_h}
\lesssim \Vert \bm D^{1/2}(\bm U_h^i) \bm \theta_h^i\Vert_{L^2(\Omega)}
\, \Vert \bm \pi_k \bm z\Vert_{L^{2d}(\Omega)}.
\]
Owing to the stability of the $L^2$ projection in $L^{2d}(\Omega)$, we conclude:

	\begin{equation} \label{eq:qh_bnd_inter}
	\norm[1]{ \bm q_h^i}_{L^{2d/(2d-1)}(\Omega)} \lesssim \norm[1]{\bm D^{1/2}(\bm U_h^i) \bm \theta_h^i}_{L^2(\Omega)},
	\end{equation}
 which, with Theorem~\ref{lem:stab_transport}, gives us the result. 
\end{proof}

\section{Compactness}
\label{sec:compact} 

We require a number of technical results concerning the compactness properties of HDG approximations which we summarize in  \Cref{ss:compact_prelim} and apply it for the flow problem in \Cref{ss:compact_flow}.  
The time derivative of the discrete concentration is bounded in \Cref{sss:compact_time_deriv}, which is then utilized in \Cref{ss:compact_transport} to prove compactness of the concentration.

Our analysis will utilize projections valid for $k \ge 1$, which we now assume for the remainder of the paper.

\subsection{Preliminaries}
\label{ss:compact_prelim}

We recall the BDM projection \cite{Boffi:book,Sayas:book}, defined locally, for a vector function $\bm{q}$ 
 	sufficiently smooth, as  $\bm{\Pi}^{\text{BDM}} \bm{q} \in \bm{V}_h$; for any $E\in\mathcal{E}_h$:
 	\begin{subequations}\label{eq:BDM_proj}
 		\begin{align}		
 	 		\forall \bm{v}_h \in \mathcal{N}_{k-2}(E), \quad(\bm{\Pi}^{\text{BDM}} \bm{q}, \bm{v}_h)_E &= 
 			(\bm{q},\bm{v}_h)_E,  \label{eq:BDM_proj_1} \\
 	 		\forall \widehat{w}_h \in \mathcal{R}_k(\partial E),\quad	
 	 		\langle \bm{\Pi}^{\text{BDM}} \bm{q} \cdot \bm{n}, \widehat{w}_h \rangle_{\partial E} &= \langle 
 			\bm{q} 
 	 		\cdot \bm{n}, \widehat{w}_h \rangle_{\partial E},  \label{eq:BDM_proj_2}
 	 	\end{align}
 	 \end{subequations}
 	 where $\mathcal{N}_{k-2}(E)$ is the local N\'{e}d\'{e}lec space of the first kind 
 	 \cite{Nedelec:1980,Boffi:book}
 	 and $\mathcal{R}_k(\partial E) = \cbr{\mu \in L^2(\partial E) \, : \, \mu|_{e} \in \mathbb{P}_k(e), \, \forall e \subset \partial K }$. 
 	The following commutativity property holds, for any $E\in\mathcal{E}_h$: 
 	\begin{equation} \label{eq:commutativity}
 	 	\nabla \cdot \bm{\Pi}^{\text{BDM}} \bm{q} =\pi_{k-1} \nabla \cdot \bm{q}, \quad \text{ on } E.
 	 \end{equation}
 	In addition, there exists a constant $C>0$ such that, for $1 \le m \le k+1$, for $E\in\mathcal{E}_h$ 
 	 and any $\bm q \in  (H^m(E))^d$, we have
 	 \begin{equation}\label{eq:BDMapprox}
 		\norm[0]{\bm{q} - \bm{\Pi}^{\text{BDM}} \bm{q}}_{L^2(E)}  \le Ch^{m} 
 	 	|\bm{q}|_{H^{m}(E)}.
 	 \end{equation}

In what follows, it will be convenient to introduce the following space-time discrete spaces: 
\begin{align*}
		\bm{\mathcal{V}}_h &= \cbr{ \bm v \in L^2(0,T;(L^2(\Omega))^d) \, : \, \bm v|_{[t^{i-1},t^i)} \in \mathbb{P}_0(t^{i-1},t^i;\bm V_h), \, \forall \,1 \le i \le N}, \\
	\mathcal{W}_h &= \cbr{ w \in L^2(0,T;L^2(\Omega)) \, : \, w|_{[t^{i-1},t^i)} \in \mathbb{P}_0(t^{i-1},t^i;W_h), \, \forall \,1 \le i \le N},\\
 	\mathcal{Q}_h &= \cbr{ q \in L^2(0,T;L_0^2(\Omega)) \, : \, q|_{[t^{i-1},t^i)} \in \mathbb{P}_0(t^{i-1},t^i;Q_h), \, \forall \,1 \le i \le N},\\
	\mathcal{M}_h &= \cbr{ \widehat w \in L^2(0,T;L^2(\partial \mathcal{E}_h)) \, : \, \widehat w|_{[t^{i-1},t^i)} \in \mathbb{P}_0(t^{i-1},t^i; M_h), \, \forall \,1 \le i \le N}.
\end{align*}

\begin{lemma} \label{cor:spacetime_grad_conv}
	Let $(w_h, \widehat w_h) \in \mathcal{W}_h \times \mathcal{M}_h$ and suppose that for some constant $C>0$, independent of the mesh-size $h$ and time-step $\tau$,
	\[
	\int_0^T \norm{(w_h,\widehat w_h)}_{1,h}^2 \dif t \le C.
	\]
	There exists $w \in L^2(0,T;H^1(\Omega))$ such that, up to a subsequence, as $h\to 0$,
	\[
	w_h \rightharpoonup w, \quad \text{ in } L^2(0,T;L^2(\Omega)), \quad \text{and}\quad \bm G_h(w_h,\widehat w_h) \rightharpoonup \nabla w, \quad \text{ in } L^2(0,T;(L^2(\Omega))^d).
	\]
\end{lemma}
\begin{proof}
	Observe that the stated bound implies the existence of $w \in L^2(0,T;L^2(\Omega))$ and $\bm g \in L^2(0,T;(L^2(\Omega))^d)$ such that, up to a subsequence, as $h\to 0$,
	\[
	w_h \rightharpoonup w, \quad \text{ in } L^2(0,T;L^2(\Omega)), \quad \bm G_h(w_h,\widehat w_h) \rightharpoonup \bm g, \quad \text{ in } L^2(0,T;(L^2(\Omega))^d).
	\]
	To see that $\bm g = \nabla w$, consider an arbitrary $\bm \varphi \in C^\infty(0,T; C_c^\infty(\Omega))^d$, select $\bm \varphi_h(\cdot, t) = \bm \Pi^{\rm BDM} \bm \varphi(\cdot,t)$ in the \cref{eq:disc_grad}, integrate over $(0,T)$, and use weak--strong convergence:
	\begin{align*}
		\int_0^T(\bm g, \bm \varphi)_{\Omega} \dif t&= \lim_{h \to 0}  \int_0^T (\bm G_h(w_h,\widehat{w}_h),  \bm \Pi^{\rm BDM} \bm \varphi)_{\mathcal{E}_h} \dif t \\
		&=   -\lim_{h \to 0} \int_0^T (w_h,  \pi_{k-1}\nabla \cdot \bm \varphi)_{\mathcal{E}_h} \dif t
		\\ &= - \int_0^T (w,  \nabla \cdot \bm \varphi)_{\Omega}  \dif t.
	\end{align*}
\end{proof}

We define $\tilde{c}_h \in C(0,T; W_h) $ to be the piecewise affine interpolant in time of $c_h^i$:
\begin{equation} \label{eq:c_interp_time}
\tilde{c}_h = \frac{\del{t - t_{i-1}}}{\tau} c_h^i + \frac{\del{ t_i - t}}{\tau} c_h^{i-1}, \quad t_{i-1} < t \le t_i.
\end{equation}
We denote by $(c_h,\widehat c_h) \in \mathcal{W}_h \times\mathcal{M}_h$ and $(c_h^-,\widehat c_h^{\, - }) \in \mathcal{W}_h \times\mathcal{M}_h$ the following piecewise constant interpolants in time of $c_h^i$ and $\widehat c_h^i$:
\begin{align*}
(c_h(t), \widehat c_h(t)) &= (c_h^i,\widehat c_h^{\,i}), \quad t_{i-1} < t \leq t_i, \quad 1\leq i\leq N \\
(c_h^-(t), \widehat c_h^{\, - }(t)) &= (c_h^{i-1}, \widehat c_h^{\,i-1}),  \quad t_{i-1} < t \leq t_i, \quad 1\leq i\leq N
\end{align*}
Similarly, we will denote by $\bm q_h, \bm \theta_h,   \bm u_h, \bm U_h  \in \bm{\mathcal{V}}_h$  and $p_h \in \mathcal{Q}_h$ the piecewise constant interpolants in time of $\bm q_h^i$, $\bm \theta_h^i$, $\bm u_h^i$, $\bm U_h^i$  and $p_h^i$, respectively.  

\subsection{Compactness results for the flow problem}
\label{ss:compact_flow}

 In this subsection we discuss the compactness properties of the HDG approximation of the flow problem \crefrange{eq:HMHDG_darcy_step_a}{eq:HMHDG_darcy_step_c}. 
From \Cref{lem:stab_flow} there exists a constant $C>0$, independent of $h$ and $\tau$, such that
	\begin{align} \label{eq:flow_linfl2_bounds} 
	 \norm[0]{\bm u_h}_{L^\infty(0,T;L^2(\Omega))} &+ \norm[0]{\bm G_h(p_h,\widehat p_h)}_{L^\infty(0,T;L^2(\Omega))} \\ &+ \norm[0]{ \sigma_u^{1/2}(p_h - \widehat p_h)}_{L^\infty(0,T;L^2(\partial \mathcal{E}_h))} + \norm{\bm U_h}_{L^\infty(0,T;L^2(\Omega))} \le C. \notag
	\end{align}

\begin{theorem}\label{thm:43}
There exists a pair $(\check{\bm u},\check p) \in L^\infty(0,T;L^2(\Omega))^d\times L^\infty(0,T;H^1(\Omega) \cap L_0^2(\Omega))$ such that, up to a (not relabeled) subsequence,
\begin{align*}
\bm u_h &\rightharpoonup^\star \check{\bm u}, \quad \text{ in } L^\infty(0,T;L^2(\Omega)^d), \\
p_h &\rightharpoonup^\star \check p, \quad \text{ in } L^\infty(0,T;L_0^2(\Omega)), \\ p_h & \rightharpoonup \check p, \quad \text{ in } L^2(0,T;L_0^2(\Omega)),
\\ \bm G_h(p_h,\widehat p_h) &\rightharpoonup \nabla \check p, \quad \text{ in } L^2(0,T;L^2(\Omega)^d).
\end{align*}
In addition, there exists a function $\check{\bm U} \in L^\infty(0,T;L^2(\Omega)^d)$ such that
\[
\bm U_h \rightharpoonup^\star \check{\bm U}, \quad \text{ in } L^\infty(0,T;L^2(\Omega)^d).
\]
\end{theorem}
\begin{proof}
The bound \cref{eq:flow_linfl2_bounds} 
implies the weak-star convergences.
Since $p_h^i \in Q_h$, \eqref{eq:poinc_on_Qh}, \eqref{eq:flow_linfl2_bounds}, and Lemma~\ref{cor:spacetime_grad_conv} yield a function $p\in L^2(0,T;H^1(\Omega))$
such that $p_h \rightharpoonup p$
in $L^2(0,T;L^2(\Omega))$ and 
$\bm G_h(p_h,\widehat p_h) \rightharpoonup \nabla p$
in $L^2(0,T;L^2(\Omega)^d)$. Since $p_h \in L_0^2(\Omega)$, weak convergence implies $p \in L_0^2(\Omega)$.
\end{proof}

\subsection{Bounding the discrete time derivative}
\label{sss:compact_time_deriv}
First, we require an approximation result taken from \cite{Bartels:2009} for a weighted $L^2-$projection in $W^{1,p}( \mathcal{E}_h)$ to establish the compactness of the concentration approximations in $L^2(0,T;L^2(\Omega))$. Define $\pi_{k,\phi} : L^2(\Omega) \to W_h$  the weighted $L^2$ projection satisfying 
\[
\forall w_h \in W_h, \quad \forall E\in\mathcal{E}_h, \quad (\phi \pi_{k,\phi} w, w_h)_{E} 
= (\phi w, w_h)_{E}.
\]
Let $2 \le p \le 2d$; fix $E\in\mathcal{E}_h$ and fix $w \in W^{1,p}(E)$; it holds that
\begin{equation} \label{eq:approx_prop_weighted_L2}
\norm{ w - \pi_{k,\phi} w}_{L^p(E)} + h_E \Vert \nabla( w - \pi_{k,\phi} w)\Vert_{L^p(E)} \lesssim h_E |w|_{W^{1,p}(E)}.
\end{equation}
Recall also the approximation properties of the Scott-Zhang quasi-interpolant of a function $w$, denoted by $\Pi^{\rm SZ} w \in C^0(\overline\Omega) \cap W_h$: for integers $\ell$ and $m$ satisfying $1 \le m \le k+1$, $0 \le \ell \le m$, and $\ell - d/p \le m - d/q$, we have
\begin{equation} \label{eq:SZ_approx}
|w -  \Pi^{\rm SZ} w|_{W^{\ell,p}(E)} \lesssim h_E^{m-\ell+d/p -d/q} |w|_{W^{m,q}(\Delta_E)},
\end{equation}
where $\Delta_E$ is a macro-element containing $E$.
We prove a stability bound in 
the following norm on the broken Sobolev space $W^{1,p}( \mathcal{E}_h) \times L^p(\partial \mathcal{E}_h)$:
\[
\norm{(w,\widehat{w})}_{W^{1,p}(\mathcal{E}_h)} = \del[2]{\norm{w}_{L^p(\Omega)}^p + \norm{\nabla_h w}_{L^p(\Omega)}^p +  \sum_{E \in \mathcal{E}_h} h_E^{1-p} \norm{w - \widehat w}_{L^p(\partial E)}^p}^{1/p}.
\]
\begin{proposition} \label{prop:approx_prop_weighted_L2_SZ}
	Let $2  \le p \le \infty$.
	There exists a constant $C>0$, independent of $h$, such that for all $w \in W^{1,p}(\Omega)$,
	\begin{equation} \label{eq:W1p_H2_bnd}
	\norm[1]{ (w- \pi_{k,\phi} w, w - \Pi^{\rm SZ} w)}_{W^{1,p}(\mathcal{E}_h)}  \le C \norm{w}_{W^{1,p}(\Omega)}. 
	\end{equation}
\end{proposition}

\begin{proof}
In light of 
\cref{eq:approx_prop_weighted_L2}, it suffices to bound the facet terms. 
From the multiplicative trace inequality (see e.g. \cite[Lemma 12.15]{Ern:booki}), that $(a+b)^p \le 2^{p-1} (a^p + b^p)$ for $a,b \ge 0$, and \cref{eq:approx_prop_weighted_L2},
\begin{align*}
\sum_{E \in \mathcal{E}_h} h_E^{1-p} \norm{w - \pi_{k,\phi} w}_{L^p(\partial E)}^p
&\lesssim 
\sum_{E \in \mathcal{E}_h} \del{ h_E^{-p}  h_E^{p} \vert w \vert_{W^{1,p}(E)}^p+ h_E^{1-p} h_E^{p-1} \vert w\vert_{W^{1,p}(E)}^p } 
 \lesssim  |w|_{W^{1,p}(\Omega)}^p.
\end{align*}

\noindent 
Observe that $\Pi^{\rm SZ} w \in C^0(\overline\Omega) \cap W_h$, and hence has a single-valued trace on $\partial \mathcal{E}_h$. 
It follows that $(\Pi^{\rm SZ} w)|_{\partial \mathcal{E}_h} \in M_h$. 
Proceeding as above, 
\begin{align*}
	&\sum_{E \in \mathcal{E}_h} h_E^{1-p} \norm[1]{w - \Pi^{\rm SZ}  w}_{L^p(\partial E)}^p
            \lesssim \sum_{E \in \mathcal{E}_h} |w|_{W^{1,p}(\Delta_E)}^p 
              \lesssim |w|_{W^{1,p}(\Omega)}^p,
\end{align*}
where we have used the approximation properties of the Scott--Zhang quasi-interpolant stated in \cref{eq:SZ_approx}.
The result easily follows from the bounds above.
\end{proof}

\begin{lemma}[Time derivative bound] \label{lem:time_der_bnd}
Let $c_h \in \mathcal{W}_h$ be the discrete concentration solution to the system \crefrange{eq:discrete_transport_step_a}{eq:discrete_transport_step_d}, and $\tilde c_h$ be its piecewise affine interpolant in time defined in \cref{eq:c_interp_time}. Then,
there exists a constant $C>0$, independent of $h$ and $\tau$, such that
	\begin{equation}\label{eq:timederbound}
	\norm{\partial_t \tilde{c}_h}_{L^{2}(0,T;W^{1,2d}(\Omega)^\star)} \le C.
	\end{equation}
\end{lemma}

\begin{proof}
Owing to the assumption that $0 < \phi_0 \le \phi \le \phi_1$, we have
 \[
 \Vert \partial_t \tilde{c}_h\Vert_{L^{2}(0,T;W^{1,2d}(\Omega)^\star)}
 \lesssim 
 \Vert \phi^{1/2}\partial_t \tilde{c}_h\Vert_{L^{2}(0,T;W^{1,2d}(\Omega)^\star)},
 \]
 and we observe that for any $w\in C^\infty(0,T;W^{1,2d}(\Omega))$
 \[
 \int_0^T \int_\Omega \phi \partial_t \tilde{c}_h w
  = \sum_{i=1}^N \int_{t_{i-1}}^{t_i} \phi \partial_\tau c_h^i \, w.
 \]

\noindent	\textbf{Step one.}
For the moment, consider \cref{eq:discrete_transport_step_c} and \cref{eq:discrete_transport_step_d} with $(w_h,\widehat w_h) \in L^2(0,T;W_h) \times L^2(0,T;M_h)$ kept arbitrary.
	%
	Since $ \nabla_h \cdot (\bm U_h^i c_h^i) = \bm U_h^i \cdot \nabla_h c_h^i + c_h^i\nabla_h \cdot \bm U_h^i$ and $\tfrac{1}{2}\bm U_h^i \cdot \bm n \widehat c_h^{\,i} = -\tfrac{1}{2}\bm U_h^i \cdot \bm n (c_h^i -\widehat c_h^{\,i}) + \tfrac{1}{2}\bm U_h^i \cdot \bm n c_h^i $, integrate by parts to find
    	\begin{align*}
		(\phi \delta_\tau c_h^i, w_h)_{\mathcal{E}_h}  =
		( \bm{q}_h^i, \nabla w_h)_{\mathcal{E}_h} + (\bm{U}_h^i c_h^i, \nabla w_h)_{\mathcal{E}_h} +\tfrac{1}{2}((\nabla \cdot \bm U_h^i)c_h^i, w_h)_{\mathcal{E}_h}  \\  
  -  \langle \bm{q}_h^i \cdot \bm{n}, w_h - \widehat w_h \rangle_{\partial \mathcal{E}_h} -  \langle (\sigma_D^i + |\bm U_h^i \cdot \bm n| - \tfrac{1}{2}\bm U_h^i \cdot \bm n  ) (c_h^i - \widehat{c}_h^{\,i}), w_h - \widehat w_h \rangle_{\partial \mathcal{E}_h} \\
  - \langle \bm U_h^i \cdot \bm n \, c_h^i, w_h - \widehat w_h \rangle_{\partial \mathcal{E}_h}    - \tfrac{1}{2}((f_I^i + f_P^i) c_h^i, w_h)_{\mathcal{E}_h} + (f_I^i 
		\overline{c}^i,w_h)_{\mathcal{E}_h}  \\
= T_1 + \dots + T_8.
	\end{align*}
	%
	From H\"{o}lder's inequality and \eqref{eq:qh_bnd_inter}, we have
	\begin{align*}
	T_1 &\le \norm[0]{\bm q_h^i}_{L^{2d/(2d-1)}(\Omega)} \norm[1]{ \nabla_h w_h}_{L^{2d}(\Omega)} 
    \\
	& 
    \lesssim  \norm[1]{\bm{D}^{1/2}(\bm U_h^i) \bm \theta_h^i}_{L^2(\Omega)} \norm[1]{(w_h,\widehat w_h)}_{W^{1,2d}(\mathcal{E}_h)}.
	\end{align*}
	From H\"{o}lder's inequality and the discrete Sobolev inequality \cref{eq:disc_embed},
	\begin{align*}
	T_2 &\le \norm[1]{\bm U_h^i}_{L^2(\Omega)}  \norm[1]{c_h^i}_{L^{2d/(d-1)}(\Omega)} \norm[1]{\nabla_h w_h}_{L^{2d}(\Omega)} \\
	& \lesssim \norm[1]{\bm U_h^i}_{L^2(\Omega)}  \norm[1]{ (c_h^i, \widehat{c}_h^{\,i})}_{1,h} \norm[1]{(w_h,\widehat{w}_h)}_{W^{1,2d}(\mathcal{E}_h)} .
	\end{align*}
	From H\"older's inequality, the fact that $\nabla \cdot \bm U_h^i = \pi_k(f_I^i-f_P^i)\in Q_h$, the stability of $\pi_k$ in $L^2(\Omega)$, and \cref{eq:disc_embed}, 
	\begin{align*}
	T_3 &\lesssim \norm[1]{\nabla \cdot \bm U_h^i}_{L^2(\Omega)}  \norm[1]{c_h^i}_{L^{2d/(d-1)}(\Omega)} \Vert w_h \Vert_{L^{2d} (\Omega)}  \\
    &\lesssim \norm[1]{f_I^i - f_P^i}_{L^2(\Omega)}  \norm[1]{ (c_h^i, \widehat{c}_h^{\,i})}_{1,h}\norm[1]{(w_h,\widehat w_h)}_{W^{1,2d}(\mathcal{E}_h)}. 
	\end{align*}
 	Next, H\"{o}lder's inequality, followed by a discrete trace inequality yield 
	\begin{align*}
		T_4 & \le \sum_{E \in \mathcal{E}_h} \norm[0]{\bm q_h^i \cdot \bm n}_{L^{2d/(2d-1)}(\partial E)}  \norm{w_h - \widehat w_h}_{L^{2d}(\partial E)} \\
		&  \lesssim \sum_{E \in \mathcal{E}_h} h_E^{(1-2d)/2d} \norm[0]{\bm q_h^i}_{L^{2d/(2d-1)}(E)}  \norm{w_h - \widehat w_h}_{L^{2d}(\partial E)}. 
	\end{align*}  
	Then by H\"older's inequality and \eqref{eq:qh_bnd_inter}, we have
	\begin{align*}
		T_4 
		&  \lesssim \del[3]{ \sum_{E \in \mathcal{E}_h}  \norm[0]{\bm q_h^i}_{L^{2d/(2d-1)}(E)}^{2d/(2d-1)}}^{(2d-1)/(2d)} \del[3]{\sum_{E \in \mathcal{E}_h} h_E^{1-2d} \norm{w_h - \widehat w_h}_{L^{2d}(\partial E)}^{2d}}^{1/(2d)}
        \\
		& \lesssim  
        \norm[1]{\bm{D}^{1/2} (\bm U_h^i) \bm \theta_h^i}_{L^2(\Omega)} \norm{(w_h,\widehat w_h)}_{W^{1,2d}(\mathcal{E}_h)}. 
	\end{align*}  
	Turning to the term $T_5$, the Cauchy-Schwarz's inequality and the fact that $|\sigma_D^i + |\bm{U}_h^i \cdot \bm n | - \tfrac{1}{2} \bm{U}_h^i \cdot \bm n |\le \tfrac{3}{2}(\sigma_D^i + |\bm{U}_h^i \cdot \bm n |)$ yield
	\begin{align*}
		T_5 & \lesssim	\Vert  ( \sigma_D^i +   |\bm{U}_h^i \cdot \bm n | )^{1/2} (c_h^i - \widehat{c}_h^{\,i}) \Vert_{L^2(\partial	\mathcal{E}_h)}\, \Vert (\sigma_D^i + |\bm U_h^i \cdot \bm n|)^{1/2}(w_h- \widehat w_h)\Vert_{L^{2}(\partial \mathcal{E}_h)}.
	\end{align*}
	To bound the second factor on the right-hand side, observe that by the definition of the stabilization function \eqref{eq:diff_stab_func} and the assumption \eqref{eq:diffdisp} on the tensor $\bm D(\cdot)$, we have
	\begin{align*}
		\norm[1]{(\sigma_D^i + |\bm U_h^i \cdot \bm n|)^{1/2}(w_h-\widehat w_h)}_{L^{2}(\partial \mathcal{E}_h)}^2 
		\lesssim \sum_{E \in \mathcal{E}_h}  \int_{\partial E} (1+|\bm U_h^i|) (w_h - \widehat{w}_h)^2 \dif s.
	\end{align*}
	Further, H\"{o}lder's inequality yields
	\begin{align}
		\sum_{E \in \mathcal{E}_h} &\int_{\partial E} (1+|\bm U_h^i|) (w_h - \widehat{w}_h)^2 \dif s \nonumber\\ 
		& \le \sum_{E \in \mathcal{E}_h} \norm{w_h - \widehat w_h}_{L^2(\partial E)}^2 +\sum_{E \in \mathcal{E}_h}\norm[1]{ \bm U_h^i }_{L^{d/(d-1)}(\partial E)}\norm{w_h - \widehat w_h}_{L^{2d}(\partial E)}^{2}.\label{eq:T5int}
	\end{align}
	To estimate the first term on the right-hand side of \eqref{eq:T5int}, we again apply H\"{o}lder's inequality to find
	\begin{align*}
		\sum_{E \in \mathcal{E}_h} \norm{w_h - \widehat w_h}_{L^2(\partial E)}^2 
        & \lesssim \sum_{E \in \mathcal{E}_h} \del{ |\partial E|^{(d-1)/d} \norm{w_h - \widehat{w}_h}_{L^{2d}(\partial E)}^{2}} \\
		& \lesssim  \del[3]{\sum_{E \in \mathcal{E}_h} |E| }^{(d-1)/d}\del[3]{\sum_{E \in \mathcal{E}_h}  \frac{|\partial E|^{d-1}}{|E|^{d-1}}\norm{w_h - \widehat{w}_h}_{L^{2d}(\partial E)}^{2d}}^{1/d} \\
		& \lesssim  |\Omega|^{1-1/d}\norm{(w_h,\widehat w_h)}_{W^{1,2d}(\mathcal{E}_h)}^2.
	\end{align*}
	For the second term on the right-hand side of \eqref{eq:T5int}, we apply a discrete trace inequality, H\"older's inequality (for sums), and the continuity of the embedding of $L^2(\Omega)$ into $L^{d/(d-1)}(\Omega)$ for $d \ge 2$:
	\begin{align*}
		\sum_{E \in \mathcal{E}_h}\norm[1]{  \bm U_h^i }_{L^{d/(d-1)}(\partial E)}\norm{w_h - \widehat w_h}_{L^{2d}(\partial E)}^{2}  
		&\lesssim  |\Omega|^{1/2 - 1/d}\norm[1]{ \bm U_h^i}_{L^2(\Omega)} \norm{(w_h,\widehat w_h)}_{W^{1,2d}(\mathcal{E}_h)}^2.
	\end{align*}
	Therefore, combining the bounds above, we obtain
	\begin{align*}
		T_5 \lesssim \Vert ( \sigma_D^i +   |\bm U_h^i \cdot \bm n | )^{1/2} (c_h^i - \widehat{c}_h^{\,i}) 
			\Vert_{L^2(\partial 
				\mathcal{E}_h)} (1 + \norm[1]{ \bm U_h^i}_{L^2(\Omega)}^{1/2}) \norm{(w_h,\widehat w_h)}_{W^{1,2d}(\mathcal{E}_h)}.
	\end{align*}
    Further, by H\"{o}lder's inequality and a discrete trace inequality,
    \begin{align*}
        T_6
                &\lesssim \sum_{E \in \mathcal{E}_h} h_E^{1/(2d) - 1}\norm[1]{\bm U_h^i}_{L^2(E)} \norm[0]{c_h^i}_{L^{2d/(d-1)}(E)} \norm{w_h - \widehat w_h}_{L^{2d}(\partial E)}, 
    \end{align*}
    and therefore, with H\"older's inequality and \eqref{eq:disc_embed}, we can write 
    \begin{align*}
              T_6  & \lesssim \norm[1]{\bm U_h^i}_{L^2(\Omega)} \norm[0]{c_h^i}_{L^{2d/(d-1)}(\Omega)}   
              \del[2]{\sum_{E \in \mathcal{E}_h} h_E^{1 - 2d} \norm{w_h - \widehat w_h}_{L^{2d}(\partial E)}^{2d}}^{1/2d}\\
                & \lesssim \norm[1]{\bm U_h^i}_{L^2(\Omega)} \norm[0]{(c_h^i,\widehat c_h^{\,i})}_{1,h}   \norm{(w_h,\widehat{w}_h)}_{W^{1,2d}(\mathcal{E}_h)}.
    \end{align*}
	As for remaining terms $T_7$ and $T_8$, we directly apply H\"older's inequality and \cref{eq:disc_embed}:
	\begin{align*}
	T_7 
    & \lesssim \norm[1]{f_I^i + f_P^i}_{L^2(\Omega)} \norm[1]{ (c_h^i, \widehat{c}_h^{\,i})}_{1,h} \norm[1]{(w_h,\widehat w_h)}_{W^{1,2d}(\mathcal{E}_h)}, \\
	T_8 
    &\lesssim  \norm[1]{f_I^i}_{L^2(\Omega)} \Vert\overline{c}^i\Vert_{L^{2d/(d-1)}(\Omega)}\norm[1]{(w_h,\widehat w_h)}_{W^{1,2d}(\mathcal{E}_h)}.
	\end{align*}
	Combining the above bounds and using Theorem~\ref{lem:stab_flow} and that $f_I,f_P \in \ell^{\infty}(0,T;L^2(\Omega))$,
	\begin{align}
	(\phi  \delta_\tau c_h^i, & w_h)_{\mathcal{E}_h} \nonumber\\
 \lesssim & \del[2]{\norm[1]{\bm{D}^{1/2}(\bm U_h^i) \bm \theta_h^i}_{L^2(\Omega)} + \norm[1]{(c_h^i, \widehat c_h^{\,i})}_{1,h}+  \Vert\overline{c}^i\Vert_{L^{2d/(d-1)}(\Omega)}}\norm[1]{(w_h, \widehat w_h)}_{W^{1,2d}(\mathcal{E}_h) }  \notag  \\
 & + \Vert (\sigma_D^i +\vert \bm U_h^i \cdot \bm n\vert)^{1/2}
 (c_h^i-\widehat{c}_h^{\,i})\Vert_{L^2(\partial\mathcal{E}_h)}
 \norm[1]{(w_h, \widehat w_h)}_{W^{1,2d}(\mathcal{E}_h) }.
 \label{eq:time_der_bnd_step_one}
	\end{align}
Therefore, we have
\begin{align*}
&\int_0^T (\phi\partial_t \tilde{c}_h, w_h)_{\mathcal{E}_h}
= \sum_{i=1}^N \int_{t_{i-1}}^{t_i} (\phi \delta_\tau c_h^i, w_h)_{\mathcal{E}_h}
\lesssim \Vert (w_h, \widehat w_h)\Vert_{L^2(0,T;W^{1,2d}(\mathcal{E}_h))}\\
& \times 
\del[3]{ \tau \sum_{i=1}^N (\Vert \bm D(\bm U_h^i)^{1/2}\bm \theta_h^i \Vert_{L^2(\Omega)}^2
+ \Vert (c_h^i, \widehat{c}_h^{\,i})\Vert_{1,h}^2
+ \Vert \overline{c}^i \Vert_{L^{2d/(d-1)}(\Omega)}^2}^{1/2}\\
&+ \Vert (w_h, \widehat w_h)\Vert_{L^2(0,T;W^{1,2d}(\mathcal{E}_h))} \del[3]{\tau
\sum_{i=1}^N  \Vert (\sigma_D^i +\vert \bm U_h^i \cdot \bm n\vert)^{1/2}
 (c_h^i-\widehat{c}_h^{\,i})\Vert_{L^2(\partial\mathcal{E}_h)}^2}^{1/2}.
\end{align*}
With Theorem~\ref{lem:stab_transport},  and the assumption that $\overline{c}\in L^2(0,T;L^{2d/(d-1)}(\Omega))$, we conclude for any $(w_h,\widehat w_h)\in L^2(0,T;W_h)\times L^2(0,T;M_h)$
\[
\int_0^T (\phi\partial_t \tilde{c}_h, w_h)_{\mathcal{E}_h}
\lesssim \Vert (w_h, \widehat w_h)\Vert_{L^2(0,T;W^{1,2d}(\mathcal{E}_h))}.
\]
	\noindent
	\textbf{ Step two.} 
    Fix $w \in C^{\infty}(0,T;C^\infty(\Omega))$.  For any $t>0$, denote $(w_h(t,\cdot), \widehat w_h(t,\cdot)) = (\pi_{k,\phi} w(t,\cdot), \Pi^{\rm SZ} w(t,\cdot))\in W_h\times M_h$. Since $\delta_\tau c_h^i\in W_h$, 
\begin{align*}
 \int_0^T \int_\Omega \phi \partial_t \tilde{c}_h w
 & = \sum_{i=1}^N \int_{t_{i-1}}^{t_i} (\phi \delta_\tau c_h^i, w_h)_{\mathcal{E}_h}
  \lesssim \Vert (w_h, \widehat w_h)\Vert_{L^2(0,T;W^{1,2d}(\mathcal{E}_h))}.
 \end{align*}
Using the triangle inequality and \cref{eq:W1p_H2_bnd} with $p=2d$, 
\begin{align*}
\Vert (w_h, \widehat w_h)\Vert_{L^2(0,T;W^{1,2d}(\mathcal{E}_h))}
\lesssim & \Vert (w, w)\Vert_{L^2(0,T;W^{1,2d}(\mathcal{E}_h))} = \Vert w \Vert_{L^2(0,T;W^{1,2d}(\Omega))},
 \end{align*}
 and we conclude with a density argument.
\end{proof}

\subsection{Compactness results for the transport problem}
\label{ss:compact_transport}

The bounds obtained in \Cref{lem:flux_Lp_bnd} and \Cref{sss:compact_time_deriv} are
the starting point for the compactness result:
\begin{theorem}\label{thm:compconc}
	There exist functions $\check{\bm q} \in L^{2}(0,T;L^{2d/(2d-1)}(\Omega)^d) $ and \\ $\check{c}\in L^\infty(0,T;L^2(\Omega)) \cap L^2(0,T;H^1(\Omega))$, such that, upon passing to a subsequence if necessary, 
 \begin{subequations}
	\begin{align}
		c_h &\rightharpoonup^\star \check{c}, \quad \text{ in } L^\infty(0,T;L^2(\Omega)), \label{eq:c_weakstar_conv}\\
		c_h & \to \check{c}, \quad \text{ in } L^2(0,T;L^2(\Omega)), \label{eq:c_strong_conv}\\
  		\tilde c_h & \to \check{c}, \quad \text{ in } L^2(0,T;L^2(\Omega)), \label{eq:c_tilde_strong_conv}\\
        c_h^{-} & \to \check{c}, \quad \text{ in } L^2(0,T;L^2(\Omega)), \label{eq:c_bar_strong_conv}\\
  		 \partial_t &\tilde c_h \rightharpoonup \partial_t \check{c}, \quad \text{ in } L^2(0,T;H^2(\Omega)^\star), \label{eq:time_der_conv}\\
		\bm G_h(c_h, \widehat c_h) & \rightharpoonup \nabla \check{c}, \quad \text{ in } L^2(0,T;L^2(\Omega)^d), \label{eq:grad_weak_conv}\\
		\bm q_h & \rightharpoonup \check{\bm q}, \quad \text{ in } L^{2}(0,T;L^{2d/(2d-1)}(\Omega)^d). \label{eq:flux_weak_conv}
	\end{align}
 \end{subequations}
\end{theorem}
\begin{proof}
From the stability bound in \Cref{lem:stab_transport}, 
the Banach--Alaoglu theorem (see, e.g. \cite[Corollary 3.30]{Brezis:book})
guarantees a $\check{c} \in L^\infty(0,T;L^2(\Omega))$ such that, up to a subsequence,
$c_h \rightharpoonup^\star \check{c}$ in $L^\infty(0,T;L^2(\Omega))$.
The bounds \eqref{eq:boundconc}, \eqref{eq:norm1hconc}, \eqref{eq:timederbound} with \Cref{lem:convChapp} imply the existence of $\tilde c \in L^2(0,T;L^2(\Omega))$ such that, passing to a further subsequence,
$c_h \to \tilde c$ in $L^2(0,T;L^2(\Omega))$.
That $\tilde{c} = c$ follows from the uniqueness of distributional limits, since we can conclude that $(\tilde c_h)$ converges in the weak-$\star$ topology of $\mathcal{D}'((0,T)\times \Omega))$ to both $c$ and $\tilde c$. 
Moreover, by \cref{lem:stab_transport},
\begin{align*}
\int_0^T \del{\norm[0]{c_h -\tilde c_h}_{L^2(\Omega)}^2+  \norm[0]{c_h^- -\tilde c_h}_{L^2(\Omega)}^2}\dif t
&= \frac{2\tau^2}{3} \sum_{i=1}^N \tau \norm[1]{ \delta_{\tau} c_h^i}_{L^2(\Omega)}^2 
\lesssim \tau, 
\end{align*}
which yields \eqref{eq:c_tilde_strong_conv} and \eqref{eq:c_bar_strong_conv}.
Next, we prove \cref{eq:time_der_conv}. Due to the bound in \Cref{lem:time_der_bnd}, 
there exists a $\dot{c} \in L^2(0,T;W^{1,2d}(\Omega)^\star)$ such that, up to a further subsequence,  
$\partial_t \tilde c_h \rightharpoonup \dot{c}$ in $L^2(0,T;W^{1,2d}(\Omega)^\star)$.
Once again, we use the uniqueness of distributional limits to conclude $\dot c = \partial_t \tilde c$. 
Since $c_h - \tilde c_h \to 0$ in $L^2(0,T;L^2(\Omega))$, it follows that $\tilde c = \check{c}$ and $\partial_t \tilde c = \partial_t \check{c}$. Next, we deduce from \eqref{eq:norm1hconc} and \Cref{cor:spacetime_grad_conv} the existence of a further subsequence such that
$\bm G_h (c_h, \widehat c_h) \rightharpoonup \bm \nabla c$ in $L^2(0,T;L^2(\Omega)^d)$.
Finally, \cref{eq:flux_weak_conv} is a direct consequence of \Cref{lem:flux_Lp_bnd}.
\end{proof}

\section{Convergence}
\label{sec:limit}

In the previous section, we established that sequences of discrete solutions $(\bm u_h, p_h, \widehat p_h)$ and $(\bm \theta_h, \bm q_h, c_h, \widehat c_h)$ to the discrete flow and transport problems, respectively, converge up to a subsequence in a suitable sense. In the present section, we identify their limits as solutions to the weak formulation \eqref{eq:weak1}-\eqref{eq:weak3}. We first state the main convergence results; its proof is a corollary of Theorem~\ref{thm:convHDGflow} and Theorem~\ref{thm:convconclim} below.

\begin{theorem}
 Let $k\geq 1$ and let $(\bm u_h, p_h, c_h)$ be the discrete velocity, pressure, and concentration satisfying the HDG scheme.
 Upon passage to a subsequence, the triple $(\bm u_h, p_h, c_h)$ converges to a weak solution $(\bm u, p, c)$ of \eqref{eq:weak1}-\eqref{eq:weak3}.  
\end{theorem}

\subsection{Passing to the limit in the flow problem}

\begin{theorem}\label{thm:convHDGflow}
The pair of functions $(\check{\bm u}, \check{p})$ defined in Theorem~\ref{thm:43}, satisfies the equations \eqref{eq:weak1} and \eqref{eq:weak2}, with $c$ being replaced by $\check{c}$ defined in Theorem~\ref{thm:compconc}.
In addition, both the velocity $\bm u_h$ and the reconstructed velocity $\bm U_h$ 
converge strongly in $L^2(0,T;L^2(\Omega)^d)$ to $\check{\bm u}$.
\end{theorem}
\begin{proof}

By Theorem~\ref{thm:compconc}, $c_h^{-}$ converges strongly to $\check{c}$ in $L^2(0,T;L^2(\Omega))$. 
Since $\bm K^{-1}$ is Carath\'{e}odory, passing to a subsequence we have $\bm K^{-1}(c_h^-(x,t)) \to \bm K^{-1}(\check{c}(x,t))$ for a.e. $(x,t) \in \Omega\times(0,T)$ as $h,\tau \to 0$.
Thus, the dominated convergence theorem ensures 
\begin{align*}
\lim_{h,\tau \to 0} \int_0^T \norm[1]{  \bm K^{-1}(c_h^-)}_{L^2(\Omega)}^2 \dif t 
&= \int_0^T \norm[1]{  \bm K^{-1}(\check{c})}_{L^2(\Omega)}^2 \dif t.
\end{align*}
Let $\bm \varphi \in C^\infty(0,T; C_0^\infty(\overline{\Omega})^d)$ be arbitrary.  Taking $\bm v_h = \bm \pi_k \bm \varphi$, using
  \eqref{eq:flowwithG}  and the weak convergence of $\bm G_h(p_h,\widehat{p}_h)$
to $\nabla p$ (see Theorem~\ref{thm:43}):
\begin{align*}
-\int_0^T (\bm K^{-1}(\check{c}) \check{\bm u}, \bm \varphi)_{\Omega}\dif t &= -\lim_{h,\tau \to 0} \int_0^T (\bm K^{-1}(c_h^-) \bm u_h, \bm \pi_k \bm \varphi)_{\mathcal{E}_h}\dif t 
\\
&= -\int_0^T (\check{p}, \nabla \cdot \bm \varphi)_\Omega \dif t.
\end{align*}
Subtracting \eqref{eq:HMHDG_darcy_step_c} from \eqref{eq:HMHDG_darcy_step_b} and using \eqref{eq:disc_grad} yields for any $(s_h,\widehat{s}_h)\in Q_h \times M_h$:
\begin{equation}\label{eq:convUhint1}
- (\bm u_h^i, \bm G_h(s_h,\widehat{s}_h))_{\mathcal{E}_h}
+ \langle \sigma_u (p_h^i - \widehat{p}_h^{\, i}), s_h -\widehat{s}_h \rangle_{\partial\mathcal{E}_h}
= (f_I^i-f_P^i, s_h)_{\mathcal{E}_h}.
\end{equation}
Let $\varphi \in C^\infty([0,T] \times \overline{\Omega})$ be arbitrary. For any $t_{i-1} \leq t \leq t_i$, choose $(s_h,\widehat{s}_h) = (\pi_k \varphi(t), \widehat \pi_k \varphi(t))$ in \eqref{eq:convUhint1}  and integrate over $(0,T)$:
\begin{multline}\label{eq:convUhint2}
-\int_0^T (\bm u_h, \bm G_h(\pi_k \varphi, \widehat \pi_k \varphi))_{\mathcal{E}_h}  \dif t
+ \int_0^T\langle 
	\sigma_{u}(p_h - \widehat{p}_h ), \pi_k \varphi -\widehat \pi_k \varphi\rangle_{\partial \mathcal{E}_h} \dif t \\ 
 = \int_0^T (f_I - f_P,\pi_k \varphi)_{\mathcal{E}_h} \dif t.
\end{multline}
Since $\pi_k \varphi \to \varphi$ in $L^\infty(0,T;L^2(\Omega))$,  
\[
\lim_{h,\tau \rightarrow 0} \int_0^T (f_I - f_P,\pi_k \varphi)_{\mathcal{E}_h} \dif t
 = \int_0^T (f_I - f_P, \varphi)_{\mathcal{E}_h} \dif t
\]
We now evaluate the first term of \eqref{eq:convUhint2}. From the definition of the $L^2$ projections $\pi_k$ and $\widehat\pi_k$, 
\begin{align*}
 (\bm G_h(   \pi_k \varphi, \widehat \pi_k \varphi), \bm v_h)_{\mathcal{E}_h}
 = & (\nabla \varphi, \bm v_h)_{\mathcal{E}_h}, \quad \forall \bm v_h \in \bm V_h,
\end{align*}
i.e., $\bm G_h(   \pi_k \varphi, \widehat \pi_k \varphi) =\pi_k\nabla \varphi$. This implies that $\bm G_h(   \pi_k \varphi, \widehat \pi_k \varphi)\to \nabla \varphi$ in $L^\infty(0,T;L^2(\Omega)^d)$ Since $\bm u_h\rightharpoonup^\star \check{\bm u}$ in 
$L^\infty(0,T;L^2(\Omega)^d)$, 
\[
\lim_{h,\tau \rightarrow 0} -\int_0^T (\bm u_h, \bm G_h(\pi_k \varphi, \widehat \pi_k \varphi))_{\mathcal{E}_h}  \dif t
= -\int_0^T (\check{\bm u}, \nabla \varphi)_{\mathcal{E}_h}  \dif t.
\]
Next, by \eqref{eq:flow_linfl2_bounds} and the approximation properties of the $L^2$ projections $\pi_k$ and $\widehat \pi_k$,
\begin{align*}
\int_0^T\langle 
	\sigma_{u}(p_h - \widehat{p}_h ), \pi_k \varphi -\widehat \pi_k \varphi\rangle_{\partial \mathcal{E}_h} \dif t 
 \lesssim  h^{1/2} \Vert \varphi \Vert_{H^1(\Omega)},
\end{align*}
To conclude, passing to the limit in \eqref{eq:convUhint2} yields
\[
-\int_0^T (\check{\bm u}, \nabla \varphi)_{\mathcal{E}_h}  \dif t = \int_0^T (f_I - f_P, \varphi)_{\mathcal{E}_h} \dif t.
\]
By restricting $\varphi$ to have compact support in $\Omega$, we obtain that the limit $\check{\bm u}$ satisfies \eqref{eq:weak2}. 

What remains is to show $\bm u_h \to \check{\bm u}$ in $L^2(0,T;L^2(\Omega)^d)$. Testing with $\bm v_h = \bm u_h$, $w_h = p_h$, $\widehat w_h = \widehat p_h$
\[
		\int_0^T \del{(\bm{K}^{-1} (c_h^{-}) \bm{u}_h, \bm u_h )_{\mathcal{E}_h} +  \langle \sigma_u( p_h -  \widehat{p}_h ),p_h -  \widehat{p}_h
\rangle_{\partial \mathcal{E}_h}} \dif t = ( f^I - f^P, p_h )_{\mathcal{E}_h} 
\]
so passing to the limit we find
\begin{align*}
	\lim_{h,\tau \to 0} &	\int_0^T \big((\bm{K}^{-1} (c_h^{-}) \bm{u}_h, \bm u_h )_{\mathcal{E}_h}  
    +  \langle \sigma_u( p_h -  \widehat{p}_h ),p_h -  \widehat{p}_h
		\rangle_{\partial \mathcal{E}_h}\big) \dif t \\&= \int_0^T( f^I - f^P, \check{p} )_{\Omega} \dif t 
= \int_0^T (\nabla \cdot \check{\bm u}, p)_\Omega \dif t      
		= 	\int_0^T (\bm K^{-1}(\check{c}) \check{\bm u}, \check{\bm u})_{\Omega} \dif t.
\end{align*}
Weak lower semi-continuity of norms then yields 
\begin{align*}
	\int_0^T &(\bm K^{-1}(\check{c}) \check{\bm u}, \check{\bm u})_{\Omega} \dif t \\
    	&\le \liminf_{h,\tau \to 0} 	\int_0^T \big((\bm{K}^{-1} (c_h) \bm{u}_h, \bm u_h )_{\mathcal{E}_h} \dif t  \\
	&= \liminf_{h,\tau \to 0} 	\int_0^T \big((\bm{K}^{-1} (c_h) \bm{u}_h, \bm u_h )_{\mathcal{E}_h} \dif t 
    + \limsup_{h,\tau \to 0} \int_0^T \langle \sigma_u( p_h -  \widehat{p}_h ),p_h -  \widehat{p}_h
	\rangle_{\partial \mathcal{E}_h}\dif t 
     \\
    & \quad 
    - \limsup_{h,\tau \to 0} \int_0^T \langle \sigma_u( p_h -  \widehat{p}_h ),p_h -  \widehat{p}_h
	\rangle_{\partial \mathcal{E}_h} \dif t \\
	&\le \limsup_{h,\tau \to 0} \int_0^T \big((\bm{K}^{-1} (c_h) \bm{u}_h, \bm u_h )_{\mathcal{E}_h} + \langle \sigma_u( p_h -  \widehat{p}_h ),p_h -  \widehat{p}_h
	\rangle_{\partial \mathcal{E}_h}\big) \dif t \\
    & \quad\quad - \limsup_{h,\tau \to 0} \int_0^T \langle \sigma_u( p_h -  \widehat{p}_h ),p_h -  \widehat{p}_h
	\rangle_{\partial \mathcal{E}_h}\dif t \\
&\le \int_0^T (\bm K^{-1}(\check{c}) \check{\bm u}, \check{\bm u})_{\Omega} \dif t	- \limsup_{h,\tau \to 0} \int_0^T \langle \sigma_u( p_h -  \widehat{p}_h ),p_h -  \widehat{p}_h
	\rangle_{\partial \mathcal{E}_h} \dif t.
\end{align*}
whence we conclude 
\[
\lim_{h,\tau \to 0} \int_0^T \langle \sigma_u( p_h -  \widehat{p}_h ),p_h -  \widehat{p}_h
\rangle_{\partial \mathcal{E}_h} \dif t = 0.
\]
Consequently,
\begin{align*}
	\lim_{h,\tau \to 0} 	\int_0^T \big((\bm{K}^{-1} (c_h) \bm{u}_h, \bm u_h )_{\mathcal{E}_h}  
	&= 	\int_0^T (\bm K^{-1}(\check{c}) \check{\bm u}, \check{\bm u})_{\Omega} \dif t,
\end{align*}
which yields strong convergence of $\bm{K}^{-1/2} (c_h) \bm{u}_h$ to $\bm K^{-1/2}(\check{c}) \check{\bm u}$ in $L^2(0,T;L^2(\Omega)^d)$, and hence strong convergence of $\bm u_h$ to $\check{\bm u}$ in $L^2(0,T;L^2(\Omega)^d)$. Strong convergence of the $H(\rm div)$-reconstructed velocity $\bm U_h$ to the velocity $\check{\bm u}$ in $L^2(0,T;L^2(\Omega))$ then follows from \eqref{eq:par_el_calc}. 

\end{proof}

\subsection{Passing to the limit in the transport problem}

\begin{theorem}\label{thm:convconclim}
The function $\check{c}$ defined in Theorem~\ref{thm:compconc}
satisfies \eqref{eq:weak3}, with $\bm u$ replaced by $\check{\bm u}$.
\end{theorem}
\begin{proof}
To facilitate passage to the limit, we rewrite the HDG scheme using the HDG gradient.
With \eqref{eq:firsteq}, we immediately have 
\[
\bm \theta_h^i = - \bm G_h(c_h^i, \widehat c_h^{\, i}).
\]
Therefore, we rewrite \eqref{eq:discrete_transport_step_a} as
\begin{equation}
\label{eq:lim_pass_a}
		-(\bm D(\bm{U}_h^i) \bm G_h(c_h^i, \widehat c_h^{\,i}), \bm{z}_h)_{\mathcal{E}_h} - (\bm q_h^i, \bm z_h)_{\mathcal{E}_h}= 0.
        \end{equation}

Let $\varphi \in C^\infty(0,T;C^\infty(\Omega))$ with $\varphi(T) = 0$ and
let $\Pi_{h,\tau}^{\rm SZ}\varphi$ be an interpolant of $\varphi$ of the Scott-Zhang type, 
continuous in time and in space, satisfying 
\[
\Pi_{h,\tau}^{\rm SZ} \varphi  \mbox{ converges strongly to } \varphi \text{ in } L^\infty(0,T;W^{1,\infty}(\Omega)).
\]
Similarly, let $\bm z \in C^\infty(0,T;C^\infty(\Omega)^d)$, and 
let $\bm \pi_{k}\bm z$ be continuous in time and the $L^2$ projection of $\bm z(t)$  in $\bm V_h$ for all $r$: 
\[
\bm \pi_{k} \bm z \mbox{ converges strongly to } \bm z  \text{ in } L^\infty(0,T;W^{1,\infty}(\Omega)^d).
\]
Now, test \eqref{eq:lim_pass_a} with $\bm z_h = \bm \pi_{k} \bm z(t)$ and integrate over $(0,T)$:
\[
	-\int_0^T(\bm D(\bm{U}_h) \bm G_h(c_h, \widehat c_h),\bm \pi_k \bm z)_{\mathcal{E}_h} \dif t = \int_0^T (\bm q_h, \bm \pi_k \bm z)_{\mathcal{E}_h} \dif t.
\]
The weak convergence of $\bm q_h$ to $\check{\bm q}$ in $L^2(0,T;L^{2d/(2d-1)}(\Omega)^d)$ (see 
Theorem~\ref{thm:compconc}) and the  strong convergence of  $\bm \pi_k \bm z$ to  $\bm z$ in $L^2(0,T;L^{2d}(\Omega)^d)$ yield
\[
\lim_{h,\tau \rightarrow 0} \int_0^T (\bm q_h, \bm \pi_k \bm z)_{\mathcal{E}_h} \dif t
= \int_0^T (\check{\bm q}, \bm z)_{\mathcal{E}_h} \dif t.
\]
Since $\bm U_h$ converges to $\check{\bm u}$ strongly in $L^2(0,T;L^2(\Omega)^d)$ (from Theorem~\ref{thm:convHDGflow}), 
and since $\bm D(\cdot)$ is Lipschitz, $\bm D(\bm U_h)$ converges strongly
to $\bm D(\check{\bm u})$ in $L^2(0,T;L^2(\Omega)^{d \times d})$. The weak convergence of $\bm G_h(c_h ,\widehat c_h)$ to $\check{c}$ in $L^2(0,T;L^2(\Omega)^d)$ (see Theorem~\ref{thm:compconc}) and the strong convergence of $\bm \pi_k \bm z$ to $\bm z$ imply
\begin{align*}
- \lim_{h,\tau \to 0} \int_0^T(\bm D(\bm{U}_h) \bm G_h(c_h, \widehat c_h),\bm \pi_k \bm z)_{\mathcal{E}_h} \dif t
= & - \int_0^T(\bm D(\check{\bm{u}})\nabla \check{c}, \bm z)_{\mathcal{E}_h} \dif t.
\end{align*}
Thus we have obtained
\[
- \int_0^T(\bm D(\check{\bm{u}})\nabla \check{c}, \bm z)_{\Omega} \dif t = \int_0^T (\check{\bm q}, \bm z)_{\mathcal{E}_h} \dif t.
\]
\Bk

Next, we test \eqref{eq:discrete_transport_step_c} and \eqref{eq:discrete_transport_step_d} with $(w_h, \widehat{w}_h) = (\Pi_{h,\tau}^{\rm SZ} \varphi, -(\Pi_{h,\tau}^{\rm{SZ}} \varphi)|_{\partial \mathcal{E}_h})$, sum the resulting equations, integrate by parts in space, integrate over $(t_{i-1},t_i)$ and sum over $i$  to find: 
\begin{align*}
	\int_0^T (\phi \partial_t \tilde{c}_h, \Pi_{h,\tau}^{\rm SZ} \varphi)_{\mathcal{E}_h} \dif t -
	\int_0^T ( \bm{q}_h, \nabla \Pi_{h,\tau}^{\rm SZ} \varphi)_{\mathcal{E}_h} \dif t -  \int_0^T (\bm{U}_h c_h, \nabla \Pi_{h,\tau}^{\rm SZ} \varphi)_{\mathcal{E}_h}\dif t \\ - \frac{1}{2}\int_0^T((\nabla \cdot \bm U_h)c_h, \Pi_{h,\tau}^{\rm SZ} \varphi)_{\mathcal{E}_h} \dif t    + \frac{1}{2} \int_0^T ((f_I + f_P)c_h, \Pi_{h,\tau}^{\rm SZ} \varphi)_{\mathcal{E}_h} \dif t \\
 = \int_0^T (f_I 
	\overline{c},\Pi_{h,\tau}^{\rm SZ} \varphi)_{\mathcal{E}_h} \dif t.
\end{align*}
Since $\Pi_{h,\tau}^{\rm SZ} \varphi$ belongs to $L^2(0,T; W^{1,\infty}(\Omega))$ and we have a uniform bound on discrete time derivative in $L^2(0,T;W^{1,2d}(\Omega)^\star)$, up to a subsequence
\begin{align*}
\lim_{h,\tau \to 0}\int_0^T (\phi \partial_t \tilde{c}_h, \Pi_{h,\tau}^{\rm SZ} \varphi)_{\mathcal{E}_h} \dif t &= \lim_{h,\tau \to 0}\int_0^T \langle \phi \partial_t \tilde{c}_h, \Pi_{h,\tau}^{\rm SZ} \varphi \rangle_{W^{1,2d}(\Omega)^\star, W^{1,2d}(\Omega)} \dif t
\\
& 
= \int_0^T \langle \phi \partial_t \check{c}, \varphi \rangle_{W^{1,2d}(\Omega)^\star, W^{1,2d}(\Omega)} \dif t,
\end{align*}
where the second equality follows from weak-strong convergence.
We also clearly have by weak-strong convergence, or simply by strong convergence:
\begin{align*}
\lim_{h,\tau\to 0} -\int_0^T ( \bm{q}_h, \nabla \Pi_{h,\tau}^{\rm SZ} \varphi)_{\mathcal{E}_h} \dif t
= -\int_0^T ( \check{\bm{q}}, \nabla \varphi)_{\mathcal{E}_h} \dif t
= \int_0^T(\bm D(\bm{u})\nabla \check{c}, \nabla \varphi)_{\Omega} \dif t
\\
\lim_{h,\tau\to 0} \frac{1}{2} \int_0^T ((f_I + f_P)c_h, \Pi_{h,\tau}^{\rm SZ} \varphi)_{\mathcal{E}_h} \dif t =
\frac{1}{2} \int_0^T ((f_I + f_P) \check{c}, \varphi)_{\mathcal{E}_h} \dif t\\
\lim_{h,\tau\to 0}\int_0^T (f_I \overline{c},\Pi_{h,\tau}^{\rm SZ} w)_{\mathcal{E}_h} \dif t
 = \int_0^T (f_I \overline{c},\varphi)_{\mathcal{E}_h} \dif t
\end{align*}
We also have
\[
\lim_{h,\tau\to 0} \int_0^T (\bm{U}_h c_h, \nabla \Pi_{h,\tau}^{\rm SZ} \varphi)_{\mathcal{E}_h}\dif t  = 
\int_0^T (\check{\bm{u}} \check{c}, \nabla \varphi)_{\Omega}\dif t, 
\]
and with \eqref{eq:projdivu}, we have
\[
\lim_{h,\tau\to 0} \frac{1}{2}\int_0^T((\nabla \cdot \bm U_h)c_h, \Pi_{h,\tau}^{\rm SZ} \varphi)_{\mathcal{E}_h} \dif t 
= \frac12 \int_0^T ((f_I-f_P) \check{c}, \varphi)_\Omega \dif t.
\]
Therefore the limit $\check{c}$ satisfies
 \begin{align*}
 \int_0^T \langle \phi\partial_t \check{c},  \varphi\rangle_{W^{1,2d}(\Omega)^*,W^{1,2d}(\Omega)}
 +\int_0^T (\bm D(\check{\bm u}) \nabla \check{c}, \varphi)_\Omega
 +\frac12 \int_0^T ((f_I+f_P) \check{c}, \varphi)_\Omega\\
 -\int_0^T (\check{\bm u} \check{c}, \nabla\varphi)_\Omega
 -\frac12 \int_0^T ((f_I-f_P) \check{c}, \varphi)_\Omega
 = \int_0^T (f_I \overline{c}, \varphi)_\Omega.
 \end{align*}
 which is equivalent to \eqref{eq:weak3} with $\bm u$ replaced by $\check{\bm u}$.
\end{proof}

\section{Numerical experiments}
In this section, we conduct a number of numerical experiments to support our theoretical findings. In the first experiment, we consider a simple test case with a (known) smooth solution on the unit square to test rates of convergence with and without the H(div) conforming reconstruction \crefrange{eq:post_process_a}{eq:post_process_b}. In the second experiment, we compare the performance both with and without the H(div) reconstruction on a problem with a low-regularity solution due to a sharp discontinuity in permeability coinciding with a re-entrant corner.

Our experiments suggest that for smooth solutions, the HDG method performs comparably with and without the H(div) conforming reconstruction, in agreement with the results reported in \cite{Fabien:2020a}. However, we observed that for low regularity solutions, stability becomes an issue without the reconstruction. We remark that our proof of unconditional stability of the transport problem (\Cref{lem:stab_transport}) hinges on H(div) conformity.

\subsection{Manufactured solution}

In this first numerical experiment, we investigate the convergence of our HDG scheme both with and without the H(div) conforming reconstruction on a problem with a known smooth solution.
We take $\Omega = (0,1)^2$, and select the problem data such that the exact pressure and concentration solutions are, respectively, 
\begin{align*}
 \quad p(x,y) &= -\cos(\pi x)\cos(\pi y), \\
 c(x,y,t) &= \tfrac{1}{2}\sin(\tfrac{\pi t}{2})( \sin^2(2\pi x) + \cos^2(2\pi y)).
\end{align*}
We compute the exact velocity $\bm u$ using \eqref{eq:Darcy_prob_a} and the auxiliary variables $\bm \theta$ and $\bm q$ using \eqref{eq:transport_prob_rewritten_a} and \eqref{eq:transport_prob_rewritten_b}. As for the problem parameters, we fix the permeability $\kappa=1$, the porosity  
$\phi = 0.2$, the molecular diffusion coefficient $d_0 = 1$, the transverse dispersion coefficient $\alpha_t = 1.8 \times 10^{-6}$, and the longitudinal dispersion coefficient $\alpha_l = 1.8 \times 10^{-5}$. We assume a quarter power mixing law with
 $\mu(c) = (c \mu_s^{-0.25} + (1-c) \mu_o^{-0.25})^{-4}$, and mobility ratio $\mu_o/\mu_s= 2$. 

To test the spatial rate of convergence for smooth solutions, we compute the solution on a sequence of uniformly refined meshes starting with an initial mesh of size $h = \sqrt{2}$. We vary the polynomial degree $k\in\{0,1,2\}$. To ensure that the temporal error does not dominate the spatial error, we set $\tau = \min(0.01,h^{k+1})$. The final time is $T=0.1$. In \Cref{fig:conv_rates_flow}, we plot the $L^2$ norm of the error in the reconstructed velocity, velocity, and pressure as a function of the refinement level. \Cref{fig:conv_rates_transport} displays a similar plot for the concentration and flux. We observe the rate of convergence in the $L^2$ norm for all variables is $k+1$, in agreement with the results of \cite{Fabien:2020a}.

\begin{figure}[H]
\centering
\includegraphics[width=0.7\textwidth]{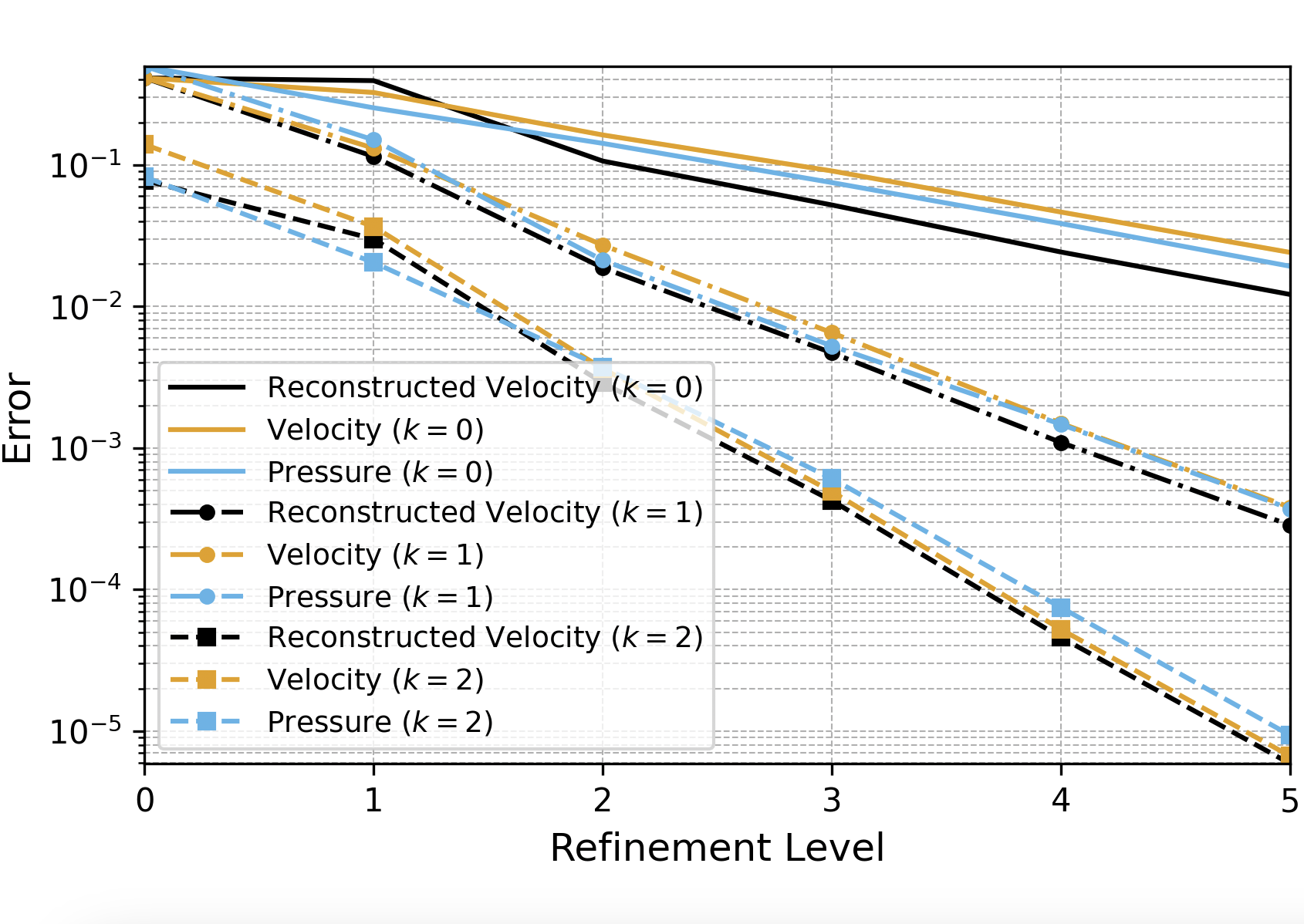}
    \caption{$L^2$ norm of the error as a function of refinement level, for different polynomial degree $k$, for the velocity $\bm u_h$, the reconstructed velocity $\bm U_h$ and the pressure $p_h$ at the final time.}
    \label{fig:conv_rates_flow}
\end{figure}

\begin{figure}[H]
\centering
\includegraphics[width=0.7\textwidth]{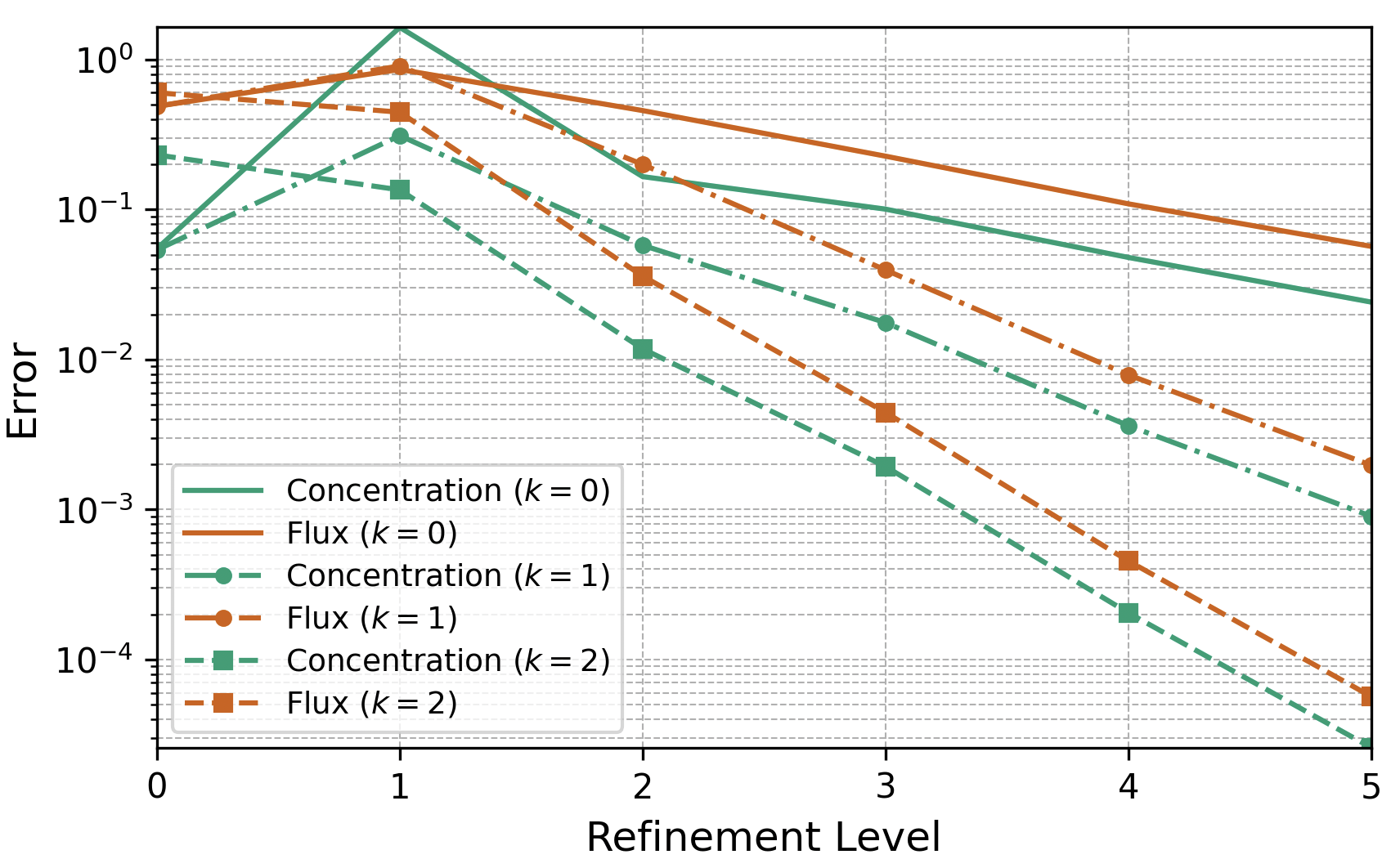}
    \caption{$L^2$ norm of the error as a function of refinement level, for different polynomial degree $k$, for the concentration $c_h$ and the flux $\bm q_h$ at the final time.}
    \label{fig:conv_rates_transport}
\end{figure}

\subsection{Singular velocity}
In the previous experiment, we observed that the HDG method performs 
We now test the performance of the proposed HDG method on a problem with low regularity found in \cite{Bartels:2009}. Consider the following $L$-shaped domain, partitioned into two porous regions $\Omega = \Omega_1 \cup \Omega_2$ with $\Omega_2 = \sbr{0,1/2}\times \sbr{1/2,1}$ and  $\Omega_1 = \Omega \setminus \Omega_2$ as illustrated in \Cref{fig:Lshape}. 
\begin{figure}[H]
    \centering
\tikzset{
pattern size/.store in=\mcSize, 
pattern size = 5pt,
pattern thickness/.store in=\mcThickness, 
pattern thickness = 0.3pt,
pattern radius/.store in=\mcRadius, 
pattern radius = 1pt}
\makeatletter
\pgfutil@ifundefined{pgf@pattern@name@_foblkcch2}{
\pgfdeclarepatternformonly[\mcThickness,\mcSize]{_foblkcch2}
{\pgfqpoint{0pt}{0pt}}
{\pgfpoint{\mcSize}{\mcSize}}
{\pgfpoint{\mcSize}{\mcSize}}
{
\pgfsetcolor{\tikz@pattern@color}
\pgfsetlinewidth{\mcThickness}
\pgfpathmoveto{\pgfqpoint{0pt}{\mcSize}}
\pgfpathlineto{\pgfpoint{\mcSize+\mcThickness}{-\mcThickness}}
\pgfpathmoveto{\pgfqpoint{0pt}{0pt}}
\pgfpathlineto{\pgfpoint{\mcSize+\mcThickness}{\mcSize+\mcThickness}}
\pgfusepath{stroke}
}}
\makeatother

\tikzset{
pattern size/.store in=\mcSize, 
pattern size = 5pt,
pattern thickness/.store in=\mcThickness, 
pattern thickness = 0.3pt,
pattern radius/.store in=\mcRadius, 
pattern radius = 1pt}\makeatletter
\pgfutil@ifundefined{pgf@pattern@name@_3oiuk3tp3}{
\pgfdeclarepatternformonly[\mcThickness,\mcSize]{_3oiuk3tp3}
{\pgfqpoint{-\mcThickness}{-\mcThickness}}
{\pgfpoint{\mcSize}{\mcSize}}
{\pgfpoint{\mcSize}{\mcSize}}
{\pgfsetcolor{\tikz@pattern@color}
\pgfsetlinewidth{\mcThickness}
\pgfpathmoveto{\pgfpointorigin}
\pgfpathlineto{\pgfpoint{\mcSize}{0}}
\pgfpathmoveto{\pgfpointorigin}
\pgfpathlineto{\pgfpoint{0}{\mcSize}}
\pgfusepath{stroke}}}
\makeatother

\tikzset{
pattern size/.store in=\mcSize, 
pattern size = 5pt,
pattern thickness/.store in=\mcThickness, 
pattern thickness = 0.3pt,
pattern radius/.store in=\mcRadius, 
pattern radius = 1pt}\makeatletter
\pgfutil@ifundefined{pgf@pattern@name@_3wo0asl88}{
\pgfdeclarepatternformonly[\mcThickness,\mcSize]{_3wo0asl88}
{\pgfqpoint{-\mcThickness}{-\mcThickness}}
{\pgfpoint{\mcSize}{\mcSize}}
{\pgfpoint{\mcSize}{\mcSize}}
{\pgfsetcolor{\tikz@pattern@color}
\pgfsetlinewidth{\mcThickness}
\pgfpathmoveto{\pgfpointorigin}
\pgfpathlineto{\pgfpoint{\mcSize}{0}}
\pgfpathmoveto{\pgfpointorigin}
\pgfpathlineto{\pgfpoint{0}{\mcSize}}
\pgfusepath{stroke}}}
\makeatother

\tikzset{
pattern size/.store in=\mcSize, 
pattern size = 5pt,
pattern thickness/.store in=\mcThickness, 
pattern thickness = 0.3pt,
pattern radius/.store in=\mcRadius, 
pattern radius = 1pt}
\makeatletter
\pgfutil@ifundefined{pgf@pattern@name@_l01vbf72o}{
\pgfdeclarepatternformonly[\mcThickness,\mcSize]{_l01vbf72o}
{\pgfqpoint{0pt}{0pt}}
{\pgfpoint{\mcSize}{\mcSize}}
{\pgfpoint{\mcSize}{\mcSize}}
{
\pgfsetcolor{\tikz@pattern@color}
\pgfsetlinewidth{\mcThickness}
\pgfpathmoveto{\pgfqpoint{0pt}{\mcSize}}
\pgfpathlineto{\pgfpoint{\mcSize+\mcThickness}{-\mcThickness}}
\pgfpathmoveto{\pgfqpoint{0pt}{0pt}}
\pgfpathlineto{\pgfpoint{\mcSize+\mcThickness}{\mcSize+\mcThickness}}
\pgfusepath{stroke}
}}
\makeatother
\tikzset{every picture/.style={line width=0.75pt}} 

\begin{tikzpicture}[x=0.75pt,y=0.75pt,yscale=-1,xscale=1]

\draw    (100.75,109.81) -- (100.49,210.85) ;
\draw    (100.49,210.85) -- (200.71,210.94) ;
\draw    (100.83,10.67) -- (100.75,109.81) ;
\draw    (199.83,110.33) -- (200.71,210.94) ;
\draw    (199.83,110.33) -- (300.05,110.42) ;
\draw    (299.83,10.67) -- (300.05,110.42) ;
\draw    (100.83,10.67) -- (299.83,10.67) ;
\draw  [pattern=_foblkcch2,pattern size=5.25pt,pattern thickness=0.75pt,pattern radius=0pt, pattern color={rgb, 255:red, 0; green, 0; blue, 0}] (101.33,10.67) -- (200.33,10.67) -- (200.33,110.33) -- (101.33,110.33) -- cycle ;
\draw  [pattern=_3oiuk3tp3,pattern size=1.5pt,pattern thickness=0.75pt,pattern radius=0pt, pattern color={rgb, 255:red, 208; green, 2; blue, 27}] (288.95,21.4) -- (299.83,21.4) -- (299.83,10.67) -- (288.95,10.67) -- cycle ;
\draw  [pattern=_3wo0asl88,pattern size=1.5pt,pattern thickness=0.75pt,pattern radius=0pt, pattern color={rgb, 255:red, 74; green, 144; blue, 226}] (100.49,210.85) -- (111.38,210.85) -- (111.38,200.12) -- (100.49,200.12) -- cycle ;
\draw   (240.14,129.62) -- (260.14,129.62) -- (260.14,150.26) -- (240.14,150.26) -- cycle ;
\draw  [pattern=_l01vbf72o,pattern size=5.25pt,pattern thickness=0.75pt,pattern radius=0pt, pattern color={rgb, 255:red, 0; green, 0; blue, 0}] (239.33,170.67) -- (259.56,170.67) -- (259.56,191.23) -- (239.33,191.23) -- cycle ;

\draw (266,22.4) node [anchor=north west][inner sep=0.75pt]  [font=\Large]  {$f^{I}$};
\draw (110,175.4) node [anchor=north west][inner sep=0.75pt]  [font=\Large]  {$f^{P}$};
\draw (271,132.4) node [anchor=north west][inner sep=0.75pt]  [font=\Large]  {$\kappa _{1}$};
\draw (271,175.4) node [anchor=north west][inner sep=0.75pt]  [font=\Large]  {$\kappa _{2}$};
\end{tikzpicture}
    \caption{L-shape domain and source/sink functions set-up. Permeability is discontinuous and takes the values $\kappa_1$ in $\Omega_1$ and $\kappa_2$ in the rest of the domain.}
     \label{fig:Lshape}
\end{figure}
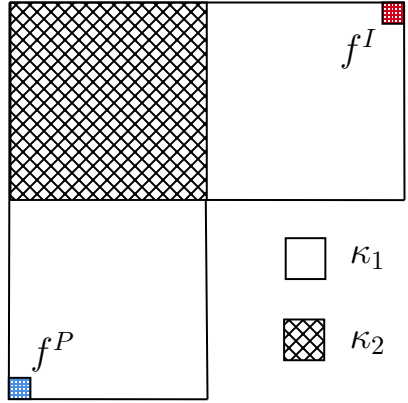
The permeability field is piecewise constant: we set $\kappa_1 = 1$ in $\Omega_1$ and $\kappa_2 = 10^{-6}$ in $\Omega_2$. As for the problem parameters, we set the porosity  
$\phi = 0.1$, the molecular diffusion coefficient $d_0 = 1.8 \times 10^{-6}$, the transverse dispersion coefficient $\alpha_t = 1.8 \times 10^{-5}$, and the longitudinal dispersion coefficient $\alpha_l = 1.8 \times 10^{-4}$. We assume a quarter power mixing law with
 $\mu(c) = (c \mu_s^{-0.25} + (1-c) \mu_o^{-0.25})^{-4}$, and mobility ratio $\mu_o/\mu_s= 4$. There is a source $f^I$ and a sink $f^P$, that are piecewise constant functions.  The support of $f^I$ (resp. $f^P$) is a small square of size $10^{-2}\times 10^{-2}$ located at the bottom left (resp. top right) corner.  The functions $f^I$ and $f^P$ are equal to $180$ on their respective supports, so that
 \begin{align*}
 \int_{\Omega} f^I \dif x = \int_{\Omega} f^P \dif x = 0.018.
 \end{align*}
 While we do not know the exact regularity of the velocity, we can get a rough estimate on the regularity of the solution of the flow problem via \cite[Theorem 7.5]{Petzoldt:2002}. In particular, if the concentration were fixed to be constant, we expect $p \in H^{1+s}(\Omega)$ and $u = -\bm K(\cdot) \nabla p \in H(\text{div};\Omega) \cap H^s(\Omega)$ with $s = \pi^{-1}\arctan\sqrt{\kappa_{2}/\kappa_{1}} \approx 0.0003$. 
 
Below, we compare the performance of the HDG method with and without $H(\text{div};\Omega)$ reconstruction on a fixed spatial mesh consisting of 36604 elements, sufficiently refined around the re-entrant corner and the injection and production wells to resolve the steep gradients in the velocity field. Throughout, we take $k=3$. Starting from an initial concentration $c_0(x) = 0$, we run each simulation until a final time of $T=5$. In~\Cref{fig:breakthrough}, we plot the average concentration at the production well as a function of time with time-step sizes $\tau \in \cbr{0.1, 0.05, 0.025, 0.0125}$ for the concentration computed with and without the $H(\text{div})$-reconstruction of the velocity. For the solution with $H(\text{div})$ reconstruction, we observe similar breakthrough curves for $\tau \in \cbr{0.1, 0.05, 0.025, 0.0125}$. While we do see that tighter coupling (i.e., smaller $\tau$) between the flow and transport increases the initial rate of recovery, for larger times (i.e. $t>3$) we see close agreement between the breakthrough curves for each $\tau$. 
For the solution without reconstruction, we initially see close agreement with the solution with reconstruction. However, once the concentration front reaches the re-entrant corner, we observe spurious oscillations in the average concentration that worsen as we reduce the time-step size, to the point of solver failure when $T = 4.5625$ in the case of $\tau = 0.0125$.

\begin{figure}[H]
\centering
\includegraphics[width=0.3\textwidth,trim={1 1cm 6cm 1cm},clip]{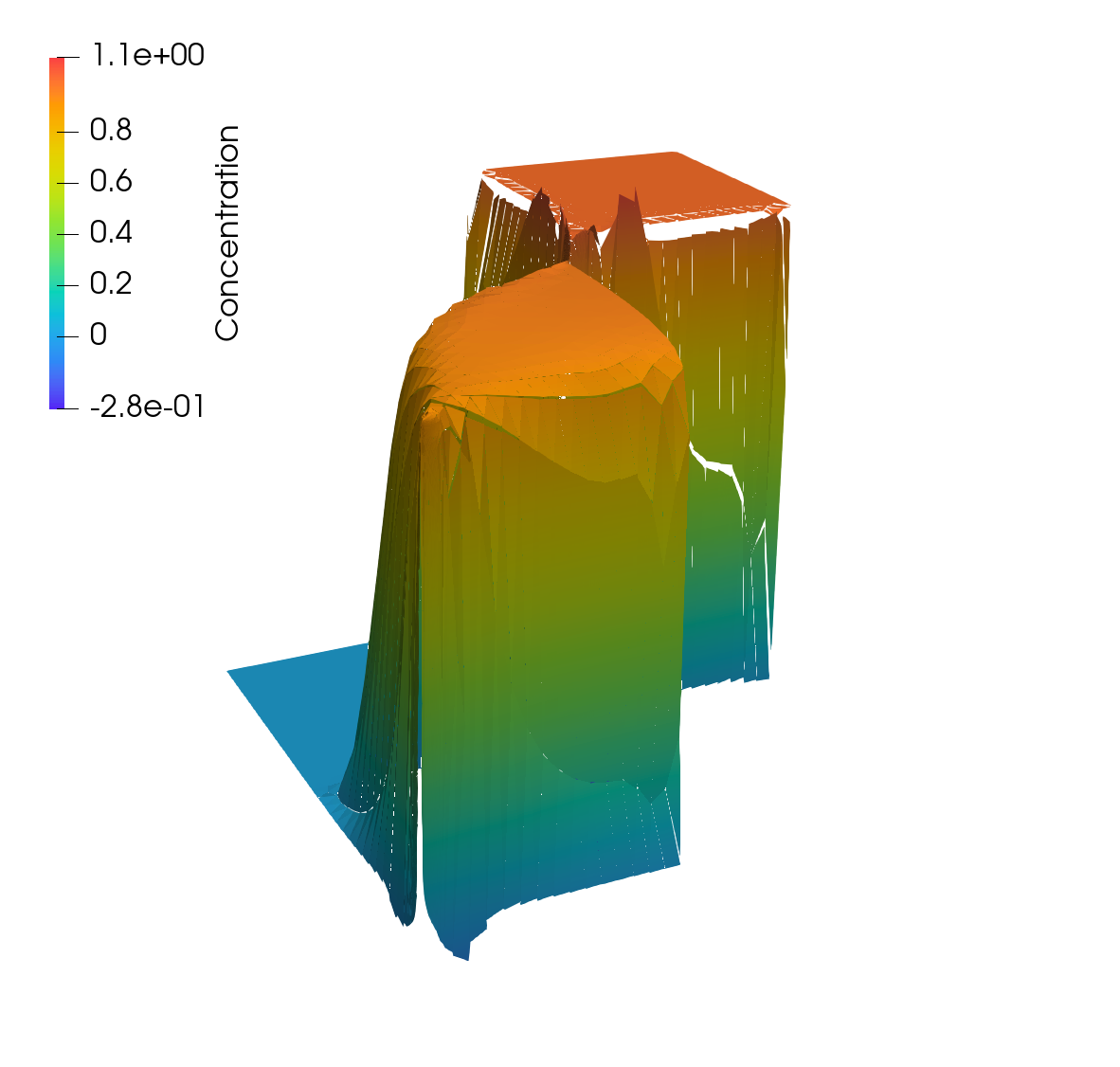}
\includegraphics[width=0.3\textwidth,trim={1 1cm 6cm 1cm},clip]{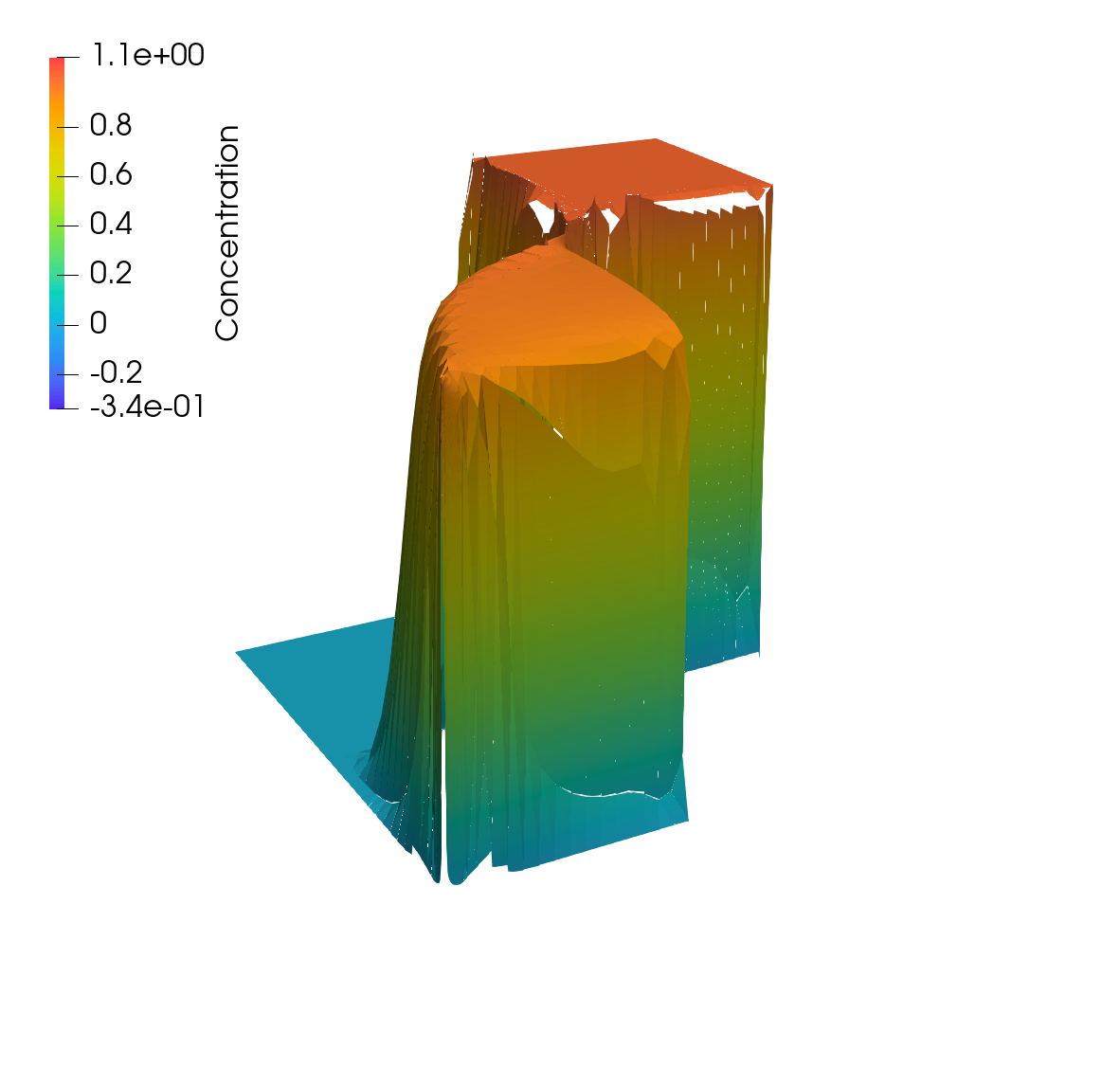}
\includegraphics[width=0.3\textwidth,trim={1 1cm 6cm 1cm},clip]{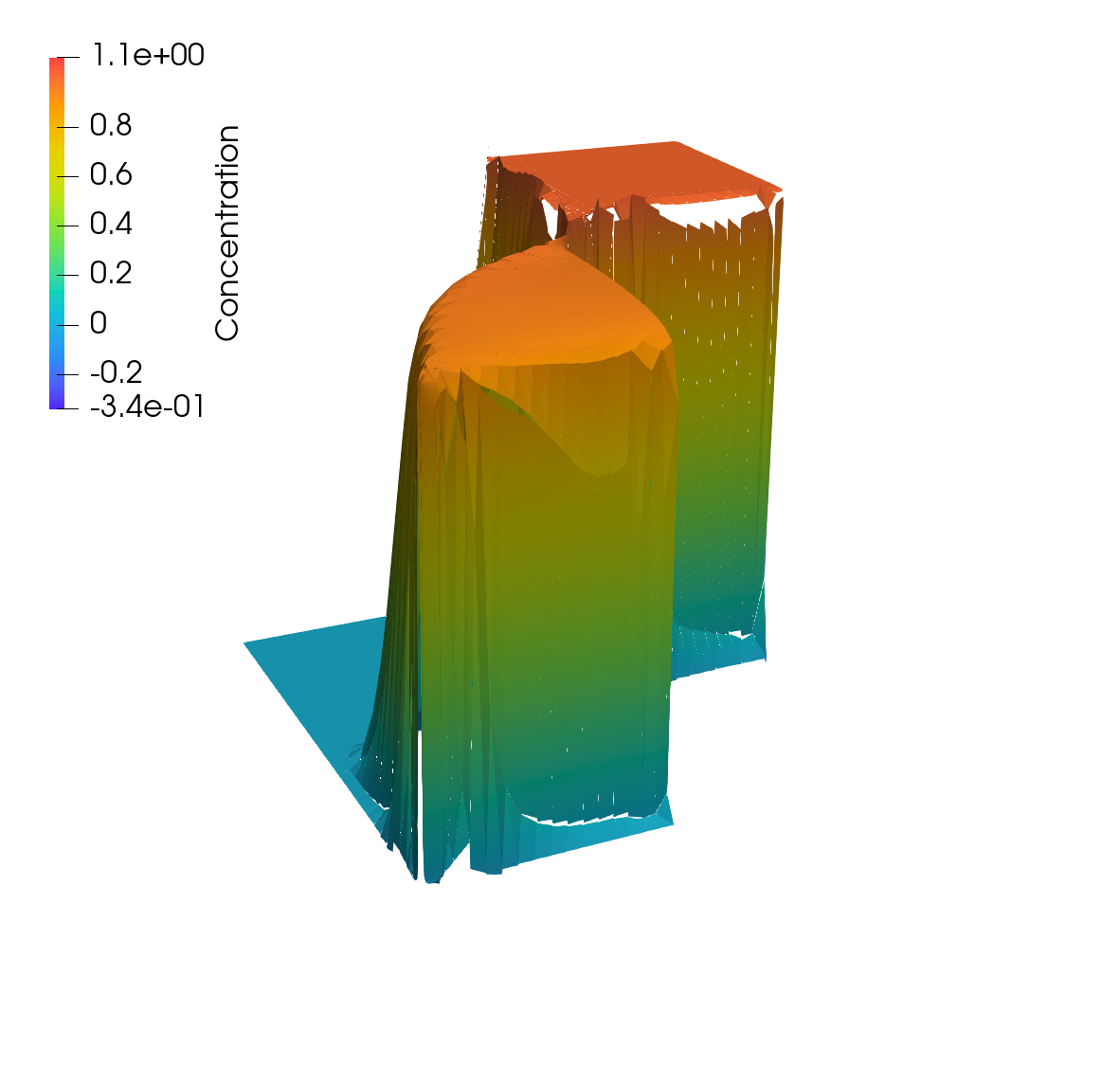}\\
\includegraphics[width=0.3\textwidth,trim={1 1cm 6cm 1cm},clip]{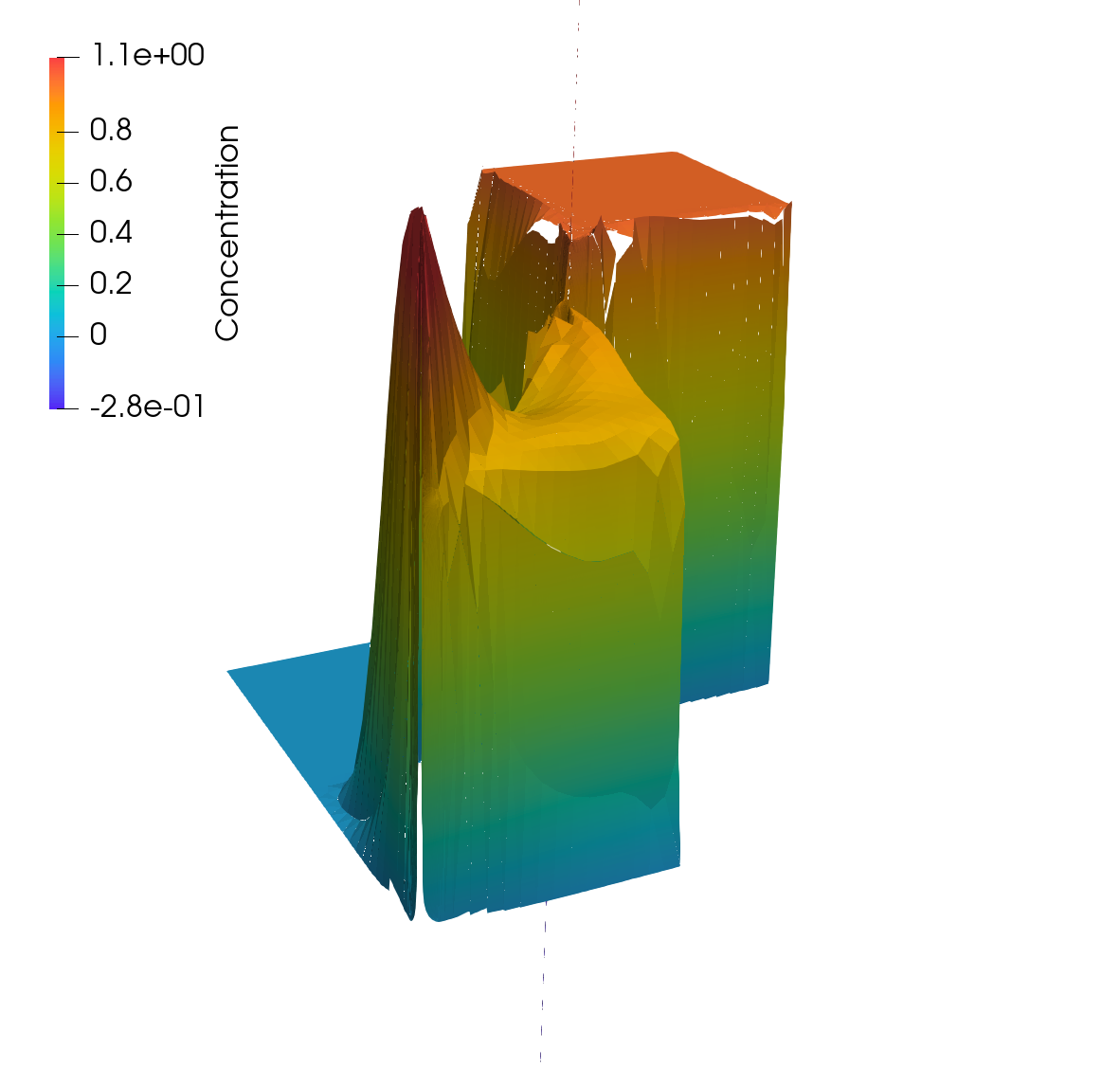}
\includegraphics[width=0.3\textwidth,trim={1 1cm 6cm 1cm},clip]{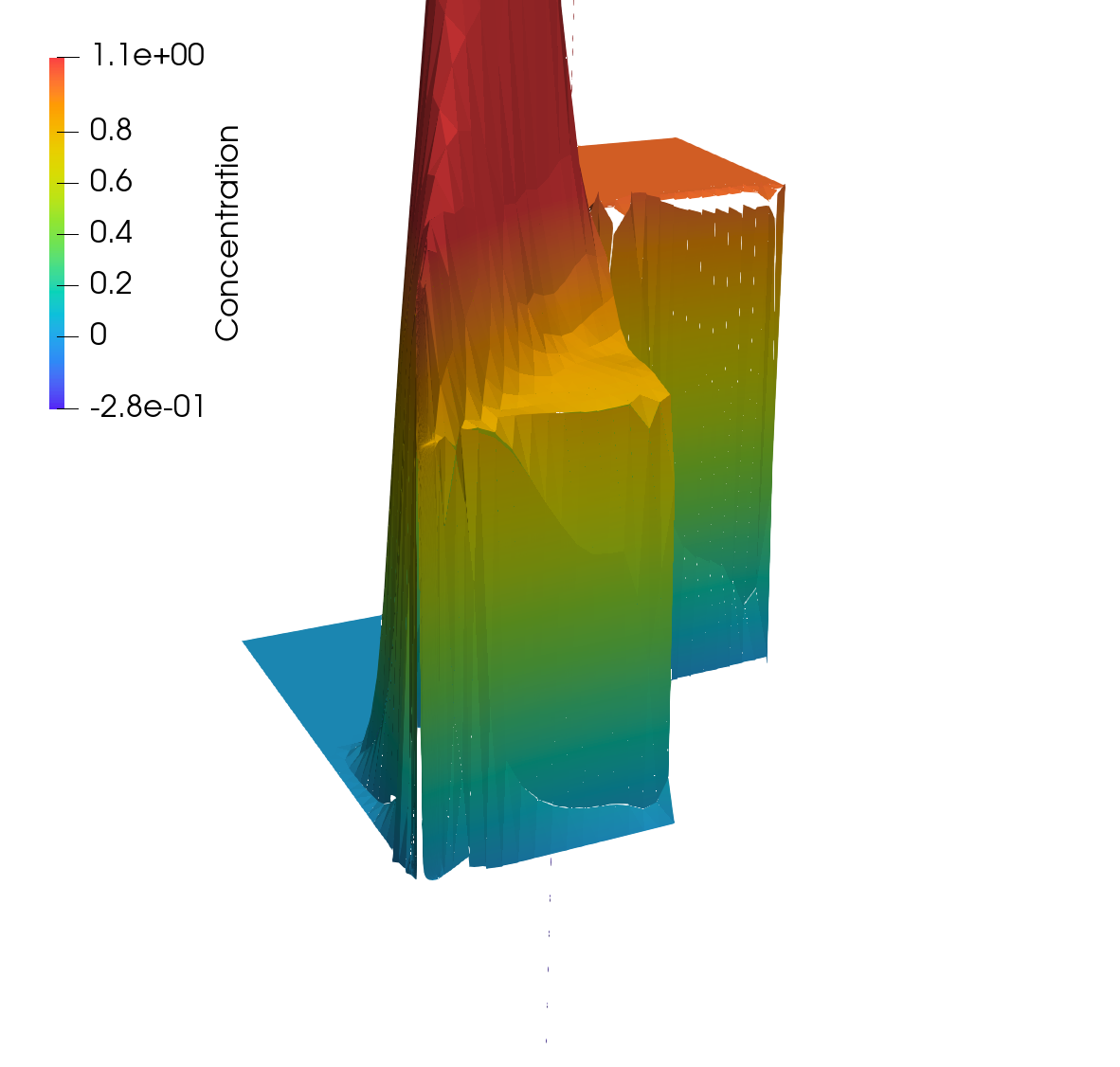}
\includegraphics[width=0.3\textwidth,trim={1 1cm 6cm 1cm},clip]{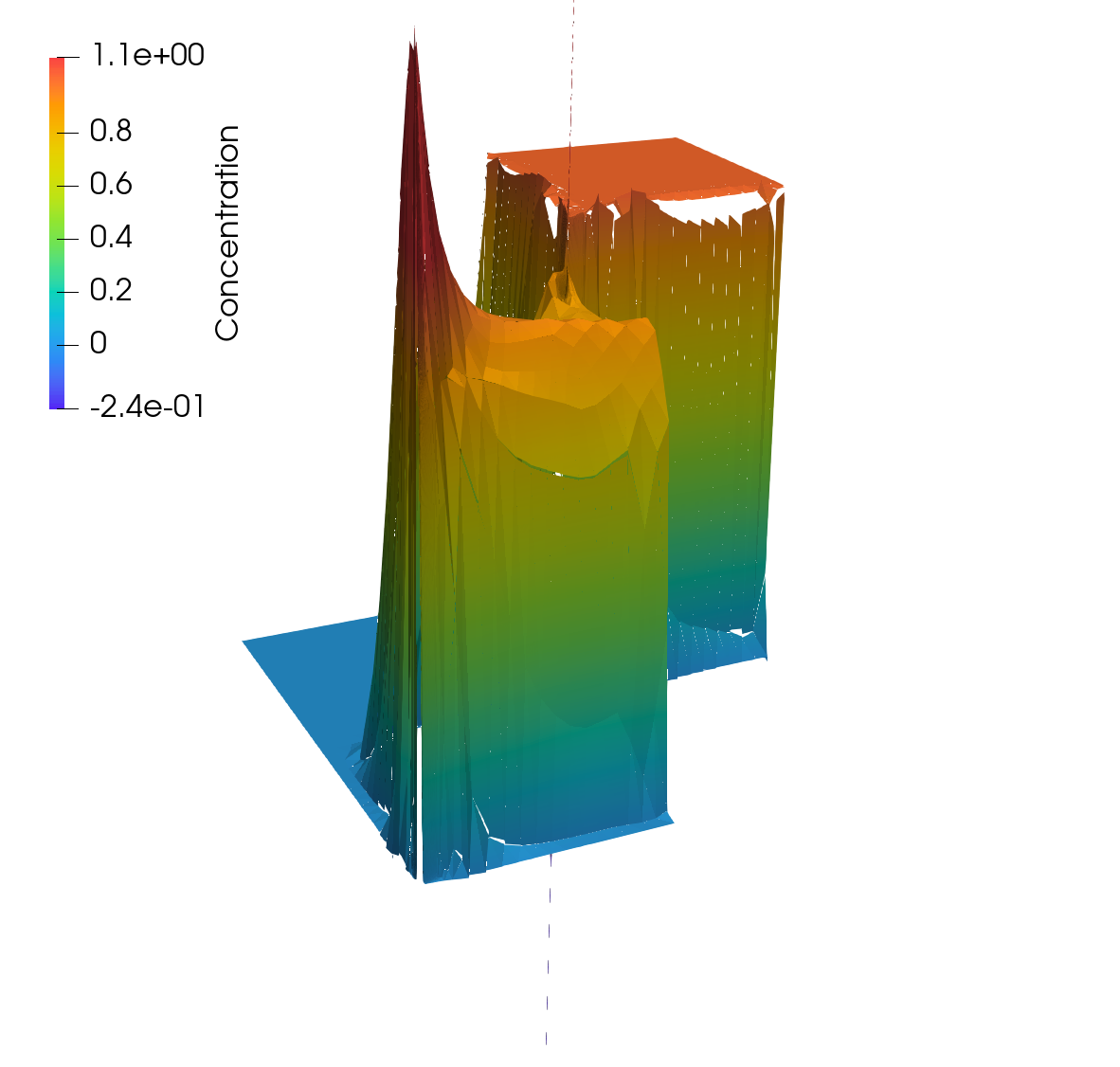}
\caption{Surface plots of the approximate concentration at $t=3$ for $\tau = 0.1$ (left), $\tau = 0.05$ (middle), $\tau = 0.025$ (right). Top row: with $H(\text{div})$ reconstruction; bottom row: without $H(\text{div})$ reconstruction.}
\label{fig:sing_conc}
\end{figure}
\begin{figure}[H]
\centering
\includegraphics[width=0.3\textwidth,trim={0 0cm 0cm 0},clip]{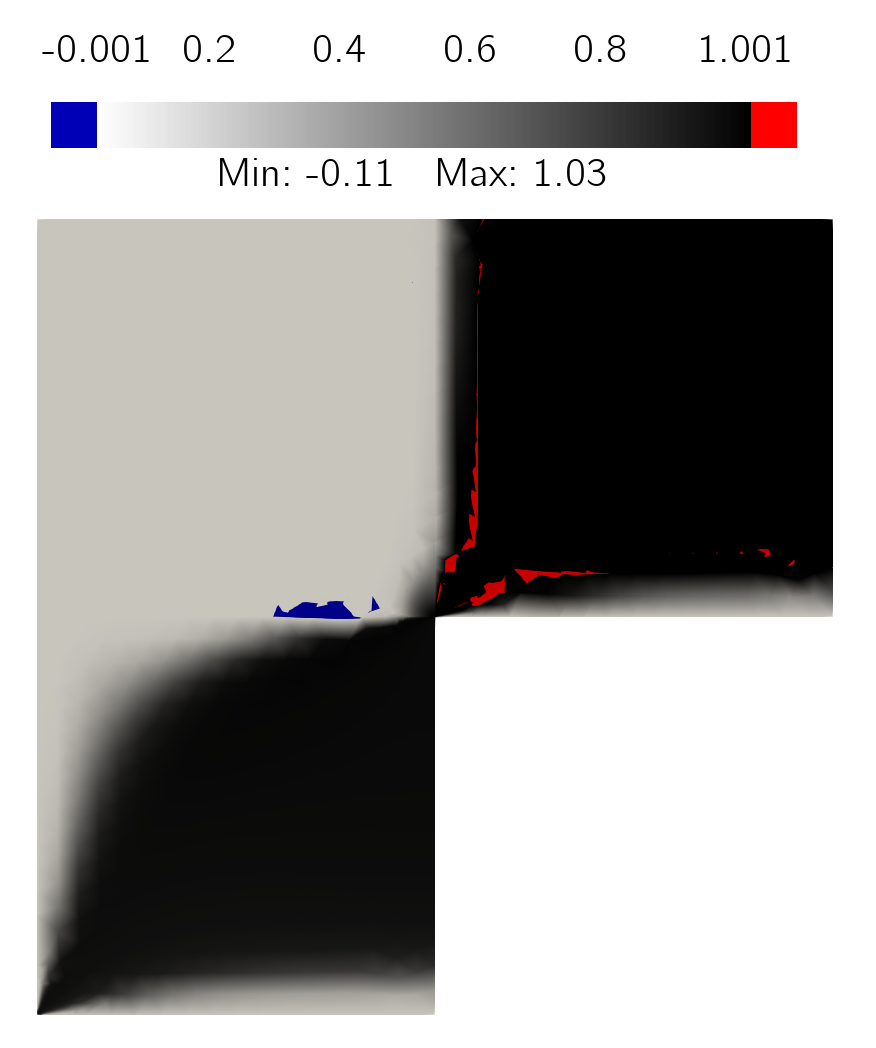} \hspace{1em}
\includegraphics[width=0.3\textwidth,trim={0 0cm 0cm 0},clip]{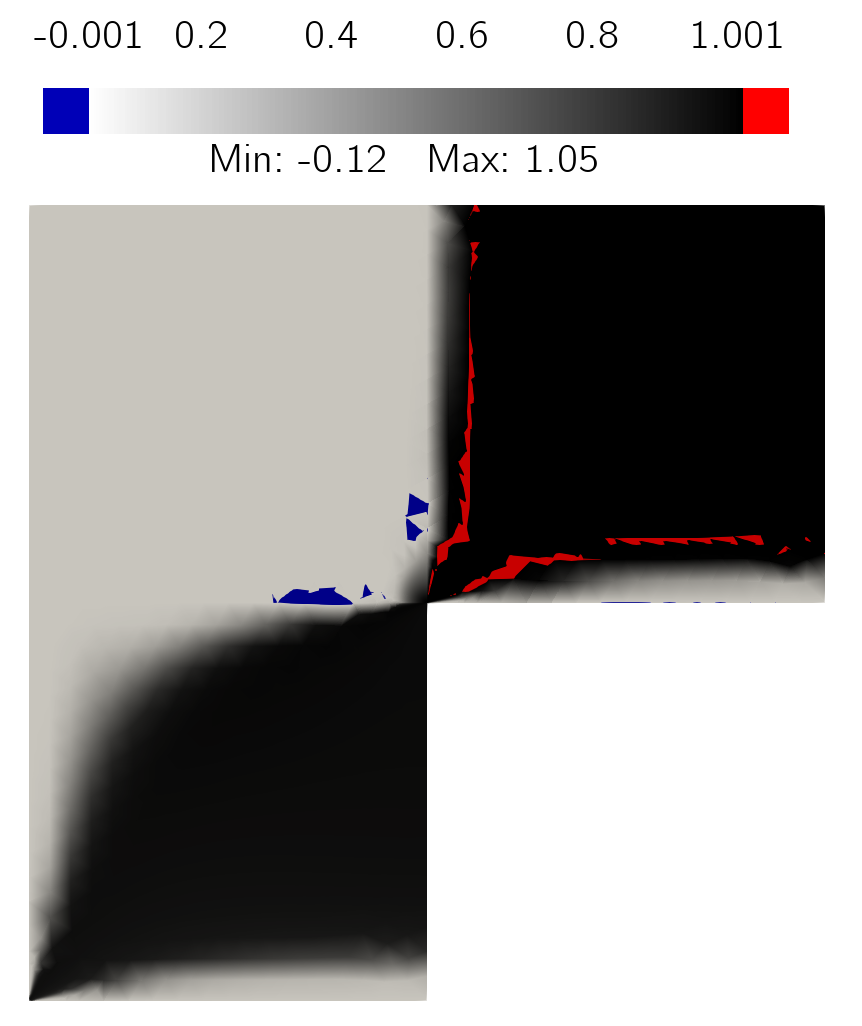}
\hspace{1em}
\includegraphics[width=0.3\textwidth,trim={0 0cm 0cm 0},clip]{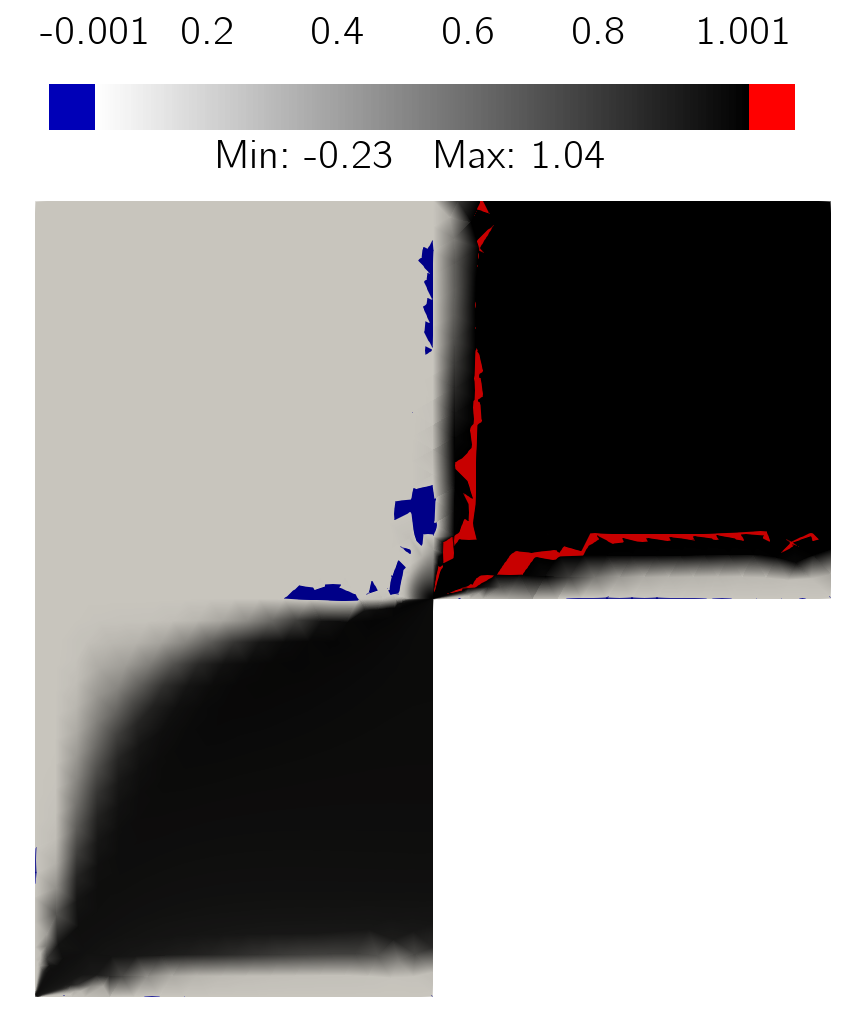}\\
\includegraphics[width=0.3\textwidth,trim={0 0cm 0cm 0},clip]{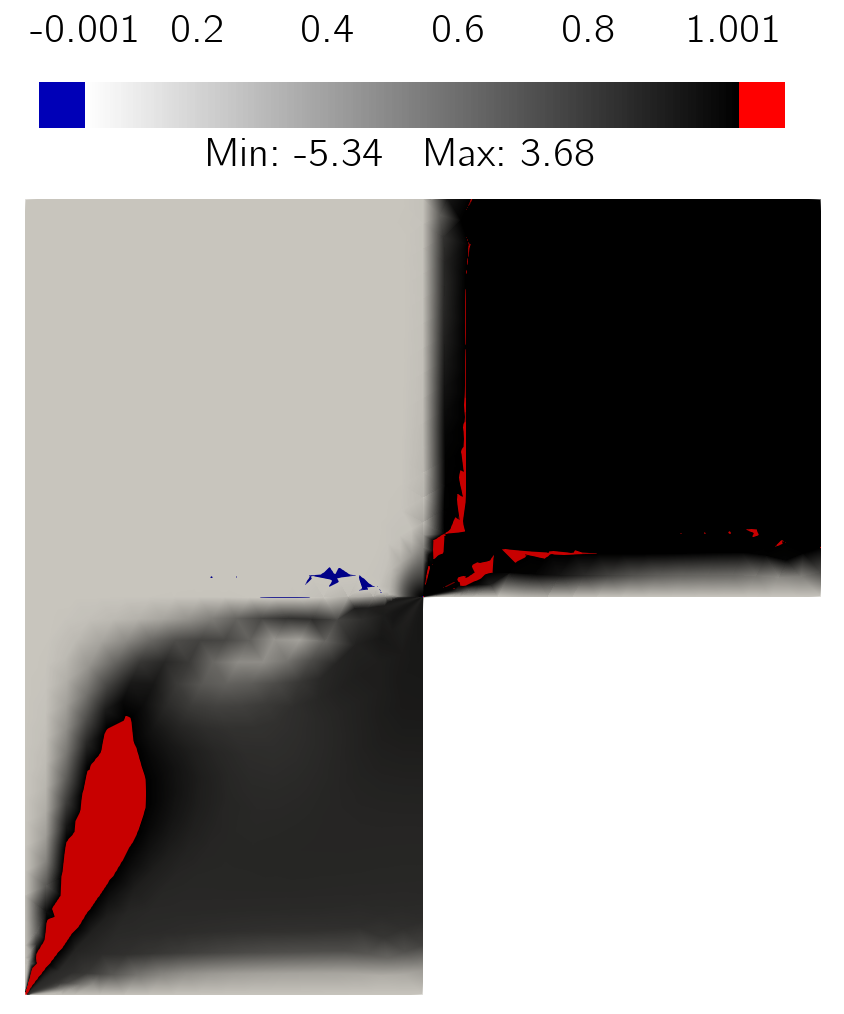}
\hspace{1em}
\includegraphics[width=0.3\textwidth,trim={0 0cm 0cm 0},clip]{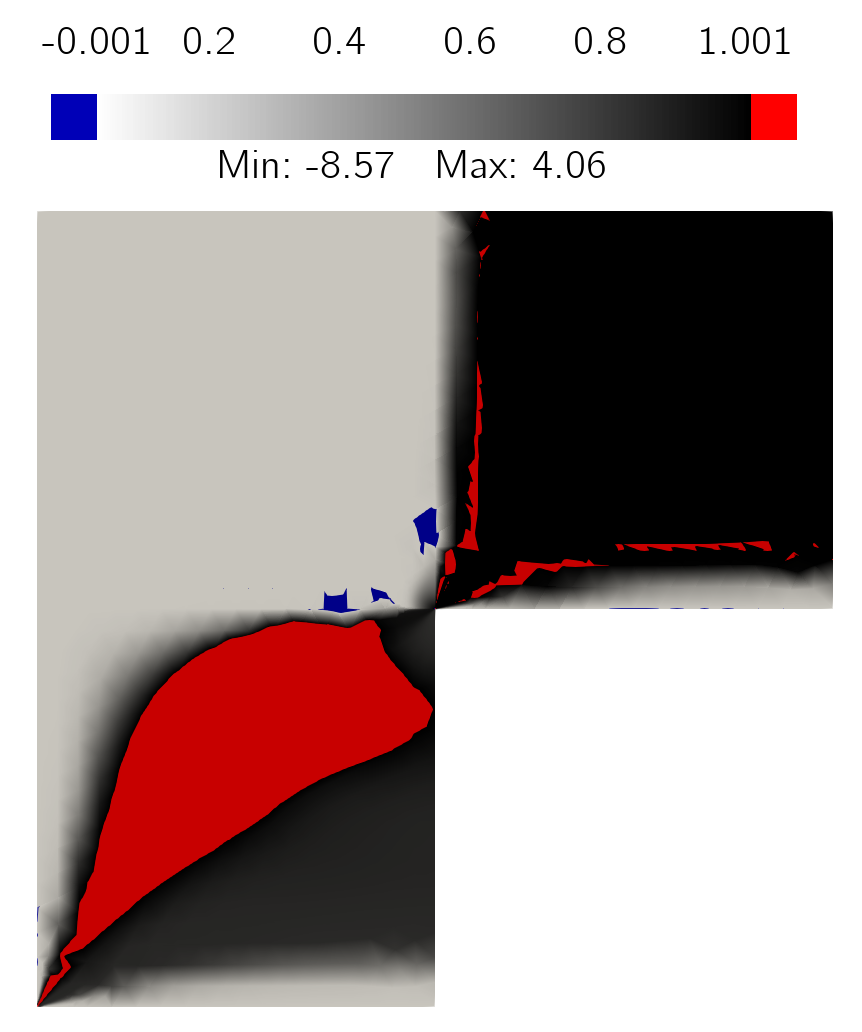}
\hspace{1em}
\includegraphics[width=0.3\textwidth,trim={0 0cm 0cm 0},clip]{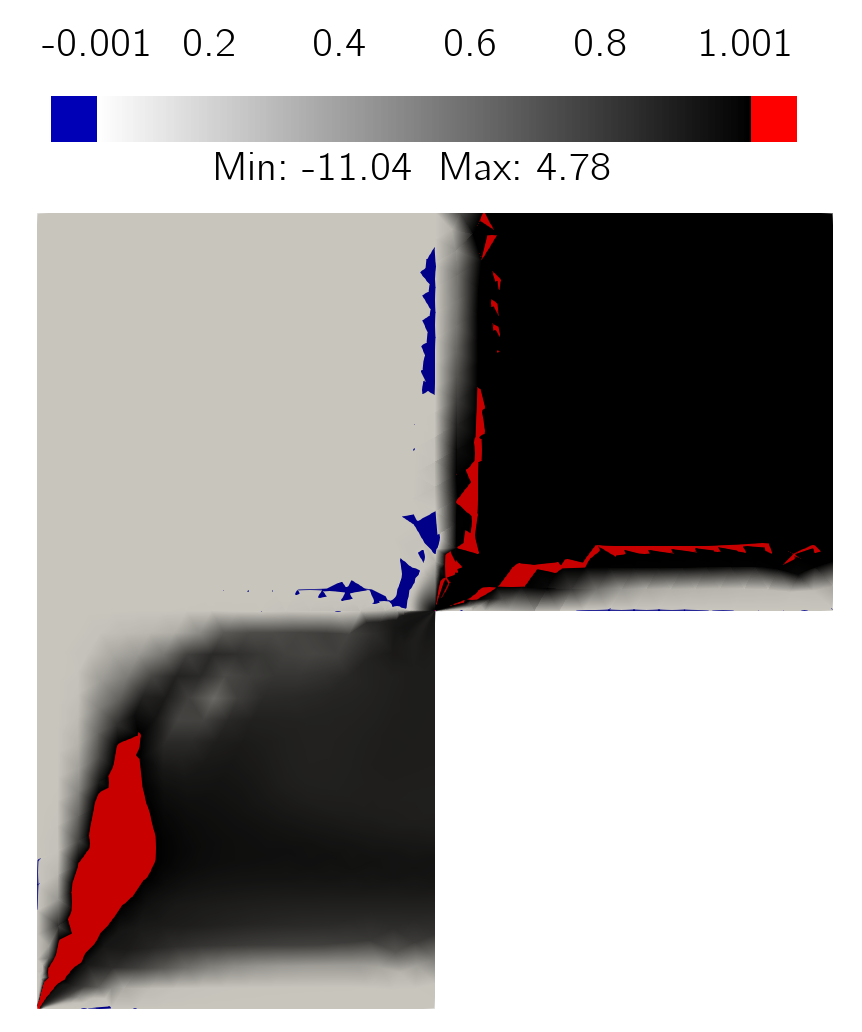}
\caption{Contour plots of the approximate concentration at $t=3$ for $\tau = 0.1$ (left), $\tau = 0.05$ (middle), $\tau = 0.025$ (right). Top row: with $H(\text{div})$ reconstruction; bottom row: without $H(\text{div})$ reconstruction. Overshoots (values above $1.001$) and undershoots (values below $-0.001$) are shaded in red and blue, respectively.}
\label{fig:sing_conc_2D}
\end{figure}

\begin{figure}[H]
\centering
\includegraphics[width=0.45\textwidth]{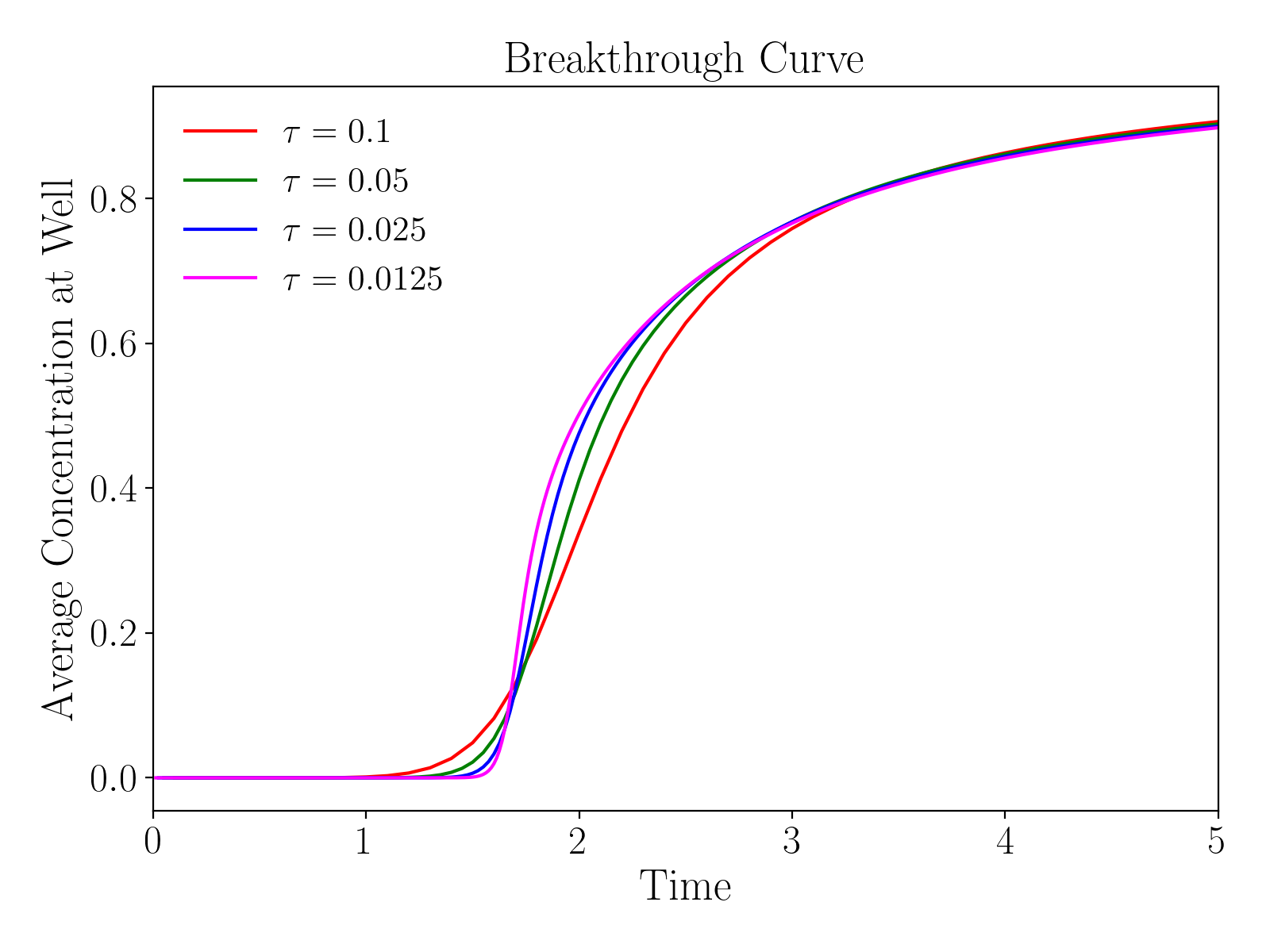}
\includegraphics[width=0.45\textwidth]{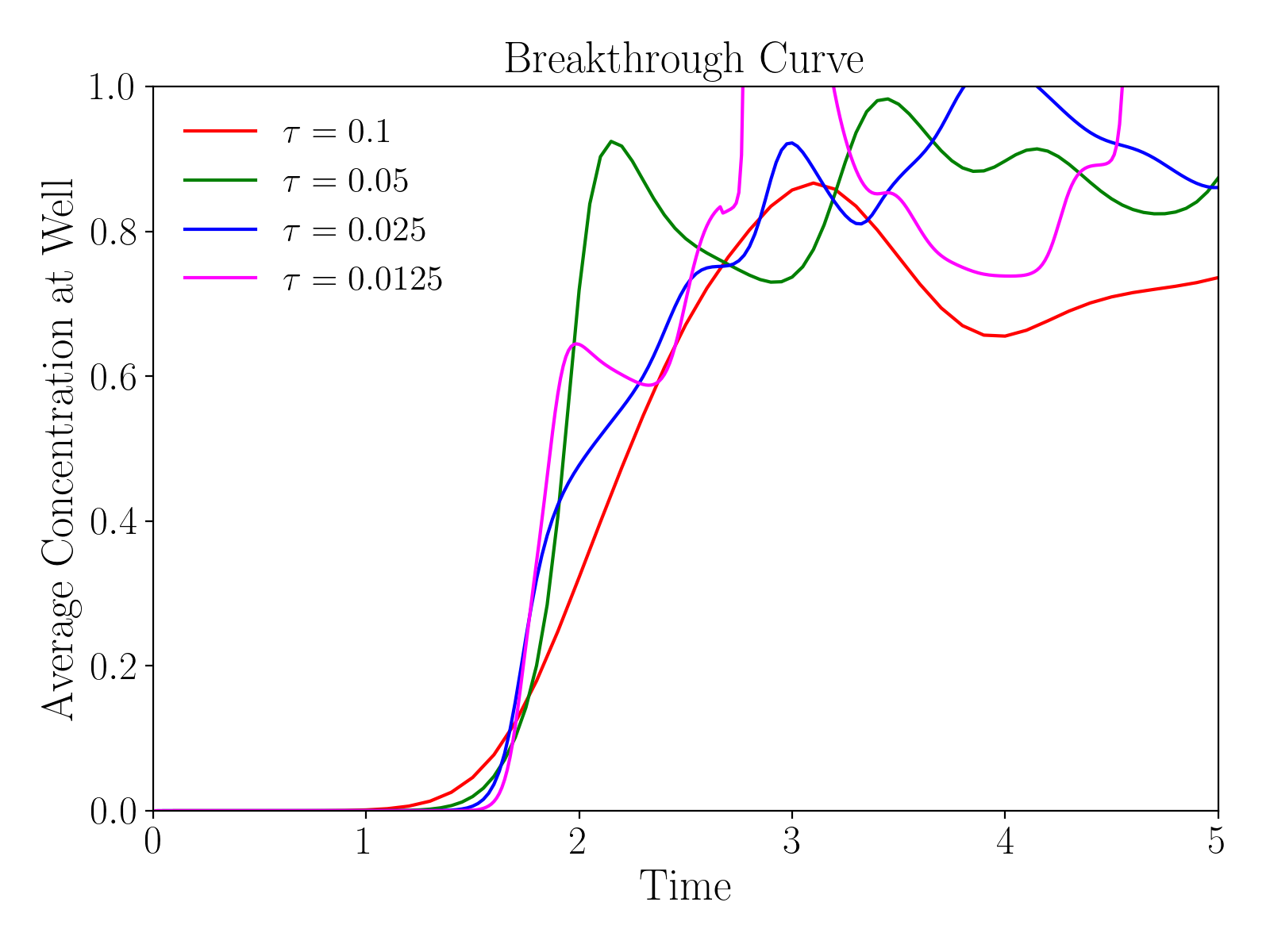}
\caption{$k=3$, with reconstruction (left) and no reconstuction (right) comparing $\tau = 0.1$, $\tau = 0.05$,$\tau = 0.025$, $\tau = 0.0125$ from $t=0$ to $t=5$. Note: for $\tau=0.0125$, solver fails at $t = 4.5625$ for case of no reconstruction}
\label{fig:breakthrough}
\end{figure}

\section{Conclusions}

Well-posedness of an HDG method of arbitrary order for the miscible displacement is proved. 
The convergence analysis  is based on a compactness argument because of the minimal regularity of the weak solutions. Convergence is proved  for any polynomial order greater than or equal to one.
The scheme is carefully designed to handle the coupling and the nonlinearities, in particular for the concentration equation. As a by-product, a compactness result for time-dependent HDG approximations is obtained.  Numerical simulations for low regularity solution show the importance of having an H(div) conforming velocity.

\appendix

\section{Miscellaneous proofs}

\subsection{Proof of \Cref{lem:HDG_poincare}}
\label{ss:appendix_poincare}

\begin{proof}
The proof of \cref{eq:disc_embed} can be found in \cite{Jiang:2023}. The proof of \cref{eq:poinc_on_Qh} follows from a slight modification of the arguments presented in \cite{Yue:2024}; for completeness, we give the details of the proof below.

 		Let $\widehat \pi_e^0: L^2(e) \to \mathbb{P}_0(e)$ be the orthogonal $L^2$-projection onto $\mathbb{P}_0(e)$ and let  $\pi_E^0: L^2(E) \to \mathbb{P}_0(E)$ be the orthogonal $L^2$-projection onto $\mathbb{P}_0(E)$.
 		Following \cite{Jiang:2023}, we will proceed by ``lifting'' $\widehat{w}_h$ into the Crouzeix--Raviart space
 		\[
 		\mathcal{CR}_h = \cbr{ v \in L^2(\Omega): \, v|_E \in \mathbb{P}_1(E), \, \forall E \in \mathcal{E}_h, \quad \int_e \jump{v} \dif s = 0, \, \forall e \in \partial \mathcal{E}_h^{\mathrm{int}}},
 		\]
 		so that we may leverage the broken Poincar\'e inequalities from \cite{Brenner:2004}. 
 		We define the lifting $\mathcal{L}_h^{\rm CR}( \widehat{w}_h)$ to be the unique element of the Crouzeix--Raviart space
 		satisfying for all $e \in \partial  \mathcal{E}_h$, 
 		\[
 		\frac{1}{|e|}\int_e \mathcal{L}_h^{\rm CR}(	\widehat w_{h}) \dif s =  \widehat \pi_e^0 \widehat{w}_h.
 		\]
 		By the triangle inequality,
 		\begin{equation} \label{eq:disc_sobolev_1}
 			\norm{w_h}_{L^p(\Omega)} \le \norm[1]{w_h - \mathcal{L}_h^{\rm CR}(\widehat w_{h})}_{L^p(\Omega)} + \norm[1]{\mathcal{L}_h^{\rm CR}(	\widehat w_{h})}_{L^p(\Omega)}.
 		\end{equation}
		To bound the first term on the right hand side of \eqref{eq:disc_sobolev_1}, observe that for each $E \in \mathcal{E}_h$  and a fixed face $e_i\subset\partial E$, and for any $p\geq 1$
 		\begin{align}
 			&\norm[1]{w_h - \mathcal{L}_h^{\rm CR}(\widehat w_{h})}_{L^p(E)} \le \nonumber\\ &  \norm[1]{(w_h - \mathcal{L}_h^{\rm CR}(\widehat w_{h})) - \widehat \pi_{e_i}^0 (w_h - \mathcal{L}_h^{\rm CR}(\widehat w_{h}) )}_{L^p(E)}  +  \norm[1]{\widehat \pi_{e_i}^0 (w_h - \mathcal{L}_h^{\rm CR}(\widehat{w}_h))}_{L^p(E)}. \label{eq:pihatint}
 		\end{align}
 		Now,  since $\widehat{\pi}_{e_i}^0$ preserves constants (with $h_E$ denoting the diameter of $E$), we obtain for the first term
		\[
 		\norm[1]{(w_h - \mathcal{L}_h^{\rm CR}(\widehat w_{h})) - \widehat \pi_{e_i}^0 (w_h - \mathcal{L}_h^{\rm CR}(\widehat w_{h}) )}_{L^p(E)} \lesssim h_E \norm{\nabla w_h}_{L^p(E)},
		\]
 		and thus using the discrete inverse inequality, we have
 		\[
 		\norm[1]{(w_h - \mathcal{L}_h^{\rm CR}(\widehat w_{h})) - \widehat \pi_{e_i}^0 (w_h - \mathcal{L}_h^{\rm CR}(\widehat w_{h}) )}_{L^p(E)} \lesssim h_E^{d(1/p - 1/2) + 1} \norm{\nabla w_h}_{L^2(E)}.
 		\]
 		For the second term in \eqref{eq:pihatint}, we write
 		\begin{align*}
 			\norm[1]{\widehat \pi_{e_i}^0 (w_h - \mathcal{L}_h^{\rm CR}(\widehat w_{h}))}_{L^p(E)} &= \frac{|E|^{1/p}}{|e_i|} \int_{e_i} (w_h - \widehat w_h) \dif s \\
			& \le  \frac{|E|^{1/p}}{|e_i|^{1/2}} \norm{w_h - \widehat w_h}_{L^2(e_i)} \\
 			& \lesssim h_E^{d(1/p - 1/2) + 1/2}  \norm{w_h - \widehat w_h}_{L^2(\partial E)}.
 		\end{align*}
        This implies, with $1 \le p \le 2d/(d-1)$,
		\begin{equation} \label{eq:disc_sobolev_2}
 			\norm[1]{w_h - \mathcal{L}_h^{\rm CR}(\widehat w_{h})}_{L^p(E)}  \lesssim  h_E^{d(1/p-1/2)+1/2}  \left(h_E^{1/2}\norm{\nabla w_h}_{L^2(E)} + \norm{w_h - \widehat w_h}_{L^2(\partial E)}\right).
 		\end{equation}
 		To bound the gradient, we have from the triangle inequality that
		\begin{align*}
 			\norm{\nabla w_h}_{L^2(E)}  \lesssim \norm{\bm{G}_h(w_h,\widehat w_h) - \nabla w_h}_{L^2(E)} +  \norm{\bm{G}_h(w_h,\widehat w_h)}_{L^2(E)},
 		\end{align*}
 		while \eqref{eq:disc_grad}, the Cauchy--Schwarz's inequality, and the discrete trace inequality yield
		\begin{align}
 			\Vert\bm{G}_h(w_h,\widehat w_h) - & \nabla w_h\Vert_{L^2(E)}^2 
 			 = \langle (\bm{G}_h(w_h,\widehat w_h) - \nabla w_h) \cdot \bm n, \widehat{w}_h -  w_h \rangle_{\partial E} \nonumber\\
 			& \lesssim h_E^{-1/2} \norm{\bm{G}_h(w_h,\widehat w_h) - \nabla w_h}_{L^2(E)} \norm{w_h - \widehat w_h}_{L^2(\partial E)}. \label{eq:usefulboundforG}
 		\end{align}
      Therefore we have
         \begin{equation}\label{eq:boundgradint}
         h_E^{1/2} \Vert \nabla w_h \Vert_{L^2(E)} \lesssim h_E^{1/2} \Vert \bm{G}_h(w_h,\widehat w_h)\Vert_{L^2(E)} +  \Vert w_h -\widehat{w}_h\Vert_{L^2(\partial E)},
         \end{equation}
 		which means that
 		\begin{multline}\label{eq:disc_sobolev_3}
 			\norm[1]{w_h - \mathcal{L}_h^{\rm CR}(\widehat w_{h})}_{L^p(E)} \lesssim  h_E^{d(1/p-1/2)+1/2} \\ \left(h_E^{1/2}\norm{\bm{G}_h(w_h,\widehat w_h)}_{L^2(E)} + \norm{w_h - \widehat w_h}_{L^2(\partial E)}\right).
 		\end{multline}
When $2\leq p\leq 2d/(d-1)$, Jensen's inequality yields:
\begin{equation}\label{eq:intboundOmega}
\norm[1]{w_h - \mathcal{L}_h^{\rm CR}(\widehat w_{h})}_{L^p(\Omega)} \lesssim h^{1/2}\norm{\bm{G}_h(w_h,\widehat w_h)}_{L^2(\Omega)} + \norm{w_h - \widehat w_h}_{L^2(\partial \mathcal{E}_h)}.
 \end{equation}
        Next, for the second term on the right hand side of \cref{eq:disc_sobolev_1}, we use the ``enrichment" operator $E_h: \mathcal{CR}_h \to \mathbb{P}_d(\mathcal{E}_h) \cap H^1(\Omega)$ and the ``restriction" operator (also called the ``forgetting" operator) $F_h: \mathbb{P}_d(\mathcal{E}_h) \cap H^1(\Omega)  \to \mathcal{CR}_h$ defined in \cite[Section 3]{Brenner:2004}, which satisfy $F_h(E_h(v_h))   = v_h  $   for all $v_h \in \mathcal{CR}_h$. In \cite[Lemma 3.2, Corollary 3.3]{Brenner:2004}, the following bounds are shown: for all $v_h \in \mathcal{CR}_h$,
		\begin{align}\label{eq:boundsEh}
 		& \norm{E_h(v_h)}_{L^2(\Omega)} \lesssim  \norm{v_h}_{L^2(\Omega)}, \quad \norm{\nabla E_h(v_h)}_{L^2(\Omega)} \lesssim \norm{\nabla_h v_h}_{L^2(\Omega)},\\
   & \forall E\in\mathcal{E}_h, \quad \Vert v_h - E_h(v_h)\Vert_{L^2(E)} \lesssim  h_E \Vert \nabla_h v_h\Vert_{L^2(\Delta_E)}, \label{eq:EnrichApproxLocal}
 		\end{align}
   where $\Delta_E$ is a macro-element containing $E$, 
 		and for all $v_h \in  \mathbb{P}_d(\mathcal{E}_h) \cap H^1(\Omega)$ and $E \in \mathcal{E}_h$,
 		\[
 		\norm{F_h(v_h)}_{L^2(E)} \lesssim  h_E \Vert \nabla v_h \Vert_{L^2(E)} + \Vert v_h \Vert_{L^2(E)} \lesssim  \norm{v_h}_{L^2(E)}.
 		\]

 		Using the discrete inverse inequality, we extend this latter stability result to $L^p(\Omega)$:
 		\begin{align*}
 			\norm{F_h(v_h)}_{L^p(\Omega)}^p&= \sum_{E \in \mathcal{E}_h} \norm{F_h(v_h)}_{L^p(E)}^p\\
 			& \lesssim \sum_{E \in \mathcal{E}_h} h_E^{d(1 - p/2)} \norm{F_h(v_h)}_{L^2(E)}^p \\
 			& \lesssim \sum_{E \in \mathcal{E}_h} h_E^{d(1-p/2)} \norm{v_h}_{L^2(E)}^p \\
 			& \lesssim \sum_{E \in \mathcal{E}_h} h_E^{d(1-p/2)} h_E^{d(p/2-1)} \norm{v_h}_{L^p(E)}^p
 			\\ & \lesssim \norm{v_h}_{L^p(\Omega)}^p.
		\end{align*}
        The bound \eqref{eq:EnrichApproxLocal} easily implies
  \begin{equation}
   \Vert v_h - E_h(v_h)\Vert_{L^2(\Omega)} \lesssim  h \Vert \nabla_h v_h\Vert_{L^2(\Omega)}. \label{eq:EnrichApprox}
   \end{equation}
 		It then follows from these stability bounds and the Sobolev embedding theorem that for all $2 \le p \le 2d/(d-1)$,
 		\begin{align*}
 			\norm[1]{\mathcal{L}_h^{\rm CR}(	\widehat w_{h})}_{L^p(\Omega)}& = \norm[1]{F_h(E_h(\mathcal{L}_h^{\rm CR}(	\widehat w_{h})))}_{L^p(\Omega)} 
            \\
 			& \lesssim \norm[1]{E_h(\mathcal{L}_h^{\rm CR}(	\widehat w_{h}))}_{H^1(\Omega)} \\
 			& \lesssim \norm[1]{\mathcal{L}_h^{\rm CR}(	\widehat w_{h})}_{L^2(\Omega)} + \norm[1]{ \nabla_h \mathcal{L}_h^{\rm CR}(	\widehat w_{h})}_{L^2(\Omega)}.
		\end{align*}
 		Finally, with \eqref{eq:disc_sobolev_3} for $p=2$, we have
   \begin{align}
 			\norm[1]{\mathcal{L}_h^{\rm CR}(	\widehat w_{h})}_{L^2(E)} & \le \norm[1]{w_h - \mathcal{L}_h^{\rm CR}(	\widehat w_{h})}_{L^2(E)} + \norm[1]{w_h}_{L^2(E)}  \notag \\ 
 			& \lesssim h_E^{1/2} \norm{\bm G_h(w_h,\widehat{w}_h)}_{L^2(E)} + \norm{w_h - \widehat w_h}_{L^2(\partial E)} + \norm[1]{w_h}_{L^2(E)}. 
		\end{align}
  This implies of course
   \begin{align}\label{eq:disc_sobolev_4}
 			\norm[1]{\mathcal{L}_h^{\rm CR}(	\widehat w_{h})}_{L^2(\Omega)} 
 			& \lesssim h^{1/2} \norm{\bm G_h(w_h,\widehat{w}_h)}_{L^2(\Omega)} + \norm{w_h - \widehat w_h}_{L^2(\partial \mathcal{E}_h)} + \norm[1]{w_h}_{L^2(\Omega)}. 
 		\end{align}
  To bound $\nabla_h \mathcal{L}_h^{\rm CR}(	\widehat w_{h})$, we remark that the broken gradient $\nabla_h \mathcal{L}_h^{\rm CR}(	\widehat w_{h})$ is the $L^2$-projection of $\bm G_h(w_h, \widehat w_h)$ onto $\mathbb{P}_{0}(\mathcal{E}_h)^d$. Indeed, this follows from the fact that, for all $\bm v_0 \in \mathbb{P}_{0}(\mathcal{E}_h)^d$, an element-wise integration by parts yields
 		\begin{align*}
 			(\bm G_h(w_h,\widehat{w}_h), \bm v_0)_{\mathcal{E}_h} &= (\nabla w_h, \bm v_0)_{\mathcal{E}_h} - \langle w_h - \widehat w_h, \bm q_h \cdot \bm n \rangle_{\partial \mathcal{E}_h} 
 			= ( \nabla \mathcal{L}_h^{\rm CR}(	\widehat w_{h}), \bm v_0)_{\mathcal{E}_h}.
 		\end{align*}
 		Consequently, the stability of the $L^2$-projection yields
 		\begin{equation} \label{eq:grad_LCR_HDG}
 		\norm[1]{  \nabla_h \mathcal{L}_h^{\rm CR}(	\widehat w_{h}) }_{L^2(\Omega)} \lesssim \norm{\bm G_h(w_h,\widehat{w}_h)}_{L^2(\Omega)}.
 		\end{equation}
        Therefore, we have
        \begin{equation}\label{eq:L2Lh}
       \norm[1]{\mathcal{L}_h^{\rm CR}(	\widehat w_{h})}_{L^p(\Omega)}
       \lesssim 
       \norm{\bm G_h(w_h,\widehat{w}_h)}_{L^2(\Omega)} + \norm{w_h - \widehat w_h}_{L^2(\partial \mathcal{E}_h)} + \norm[1]{w_h}_{L^2(\Omega)}.
        \end{equation}
        The inequality \eqref{eq:disc_embed} is then obtained by combining the bounds above.

        Next, assume that $w_h\in Q_h$.  To obtain \eqref{eq:poinc_on_Qh}, it suffices to find a bound for the term 
        $\norm[1]{\mathcal{L}_h^{\rm CR}(\widehat w_{h})}_{L^2(\Omega)}$ that does not contain the term $\Vert w_h \Vert_{L^2(\Omega)}$.
        By the usual broken Poincar\'{e} inequality for Crouzeix--Raviart functions (see, e.g. \cite[Theorem 4.1]{Brenner:2004}) and \cref{eq:grad_LCR_HDG}, we find
    \begin{align*}
     \norm[1]{\mathcal{L}_h^{\rm CR}(\widehat w_{h})}_{L^2(\Omega)} &\lesssim \norm[1]{ \nabla_h \mathcal{L}_h^{\rm CR}(	\widehat w_{h})}_{L^2(\Omega)} + \bigg|\int_{\Omega} \mathcal{L}_h^{\rm CR}(\widehat w_{h}) \dif x \bigg|\\
    & \lesssim \norm[1]{\bm G_h(w_h, \widehat w_h)}_{L^2(\Omega)} + \bigg|\int_{\Omega} \mathcal{L}_h^{\rm CR}(\widehat w_{h}) \dif x \bigg|.
    \end{align*}
Thus, to conclude,  it suffices to show that 
	\[
	\bigg| \int_\Omega \mathcal{L}_h^{\rm CR}(\widehat w_{h}) \dif x \bigg| \lesssim  \del[2]{\norm{ \bm{G}_h(w_h,\widehat w_h)}_{L^2(\Omega)}^2  + \sum_{E \in \mathcal{E}_h} \norm{ w_h - \widehat w_h}_{L^2(\partial E)}^2}^{1/2}. 
	\]
 	To this end, we follow the proof of \cite[Lemma 3.5]{Yue:2024}, reproduced here for completeness.
 	Given an element $E \in \mathcal{E}_h$, denote by $\bm c_{e,i}$ the midpoint of the $i^{\rm th}$ edge of $E$ in two dimensions, or the barycenter of the $i^{\rm th}$ face of $E$ in three dimensions. 
	Since $\mathcal{L}_h^{\rm CR}(\widehat w_{h})|_E \in \mathbb{P}_1(E)$ for all $E \in \mathcal{T}_h$, 
\begin{align*}
	\int_E\mathcal{L}_h^{\rm CR}(\widehat w_{h}) \dif x 
		&= \frac{|E|}{d+1}  \sum_{i=1}^{d+1} \mathcal{L}_h^{\rm CR}(\widehat w_{h})(\bm c_{E,i}) \\
		&= \frac{|E|}{d+1}   \sum_{i=1}^{d+1} \widehat{\pi}_{e_i}^0 \widehat{w}_h \\
		&= \frac{|E|}{d+1} \sum_{i=1}^{d+1}  \widehat{\pi}_{e_i}^0 (\widehat w_h - w_h)  + \frac{|E|}{d+1} \sum_{i=1}^{d+1}  (\widehat{\pi}_{e_i}^0 w_h - \pi_E^0 w_h) + |E| \pi_E^0 w_h.
	\end{align*}
	Therefore,
  \begin{equation*}
		\int_E\mathcal{L}_h^{\rm CR}(\widehat w_{h}) \dif x		=  \frac{|E|}{d+1} \sum_{i=1}^{d+1}  \widehat{\pi}_{e_i}^0 (\widehat w_h - w_h) + \frac{1}{d+1} \sum_{i=1}^{d+1} \int_E (\widehat{\pi}_{e_i}^0 w_h - w_h)\dif x + \int_E w_h \dif x.
	\end{equation*}
    Since $w_h \in L^2_0(\Omega)$, this yields
    \begin{equation} \label{eq:CR_integral_1}
		\int_\Omega \mathcal{L}_h^{\rm CR}(\widehat w_{h}) \dif x		=  \sum_{E\in\mathcal{E}_h} \frac{|E|}{d+1} \sum_{i=1}^{d+1}  \widehat{\pi}_{e_i}^0 (\widehat w_h - w_h) + \frac{1}{d+1} \sum_{E\in\mathcal{E}_h} \sum_{i=1}^{d+1} \int_E (\widehat{\pi}_{e_i}^0 w_h - w_h)\dif x.
	\end{equation}
 	For the first term on the right-hand side of \cref{eq:CR_integral_1}, we have
	\begin{align*}
 		\frac{|E|}{d+1} \sum_{i=1}^{d+1}  \widehat{\pi}_{e_i}^0 (\widehat w_h - w_h)  &= \frac{|E|}{d+1} \sum_{i=1}^{d+1} \frac{1}{|e_i|}  \int_{e_i} (\widehat w_h - w_h) \dif s \\
 		& \le  \frac{|E|}{d+1} \sum_{i=1}^{d+1} \frac{1}{|e_i|^{1/2}} \norm{\widehat{w}_h - w_h}_{L^2(e_i)} \\
 		& \lesssim \frac{|E|^{1/2}}{d+1} \del{\sum_{i=1}^{d+1} \frac{|E|}{|e_i|}}^{1/2} \norm{\widehat{w}_h - w_h}_{L^2(\partial E)}.
	\end{align*}
 	Recall that the shape-regularity assumption of the mesh $\mathcal{E}_h$ yields
	\[
 	|E| \approx h_E^d, \quad |e_i| \approx h_E^{d-1},
 	\]
 	and thus
 	\begin{align*}
 		\frac{|E|}{d+1} \sum_{i=1}^{d+1}  \widehat{\pi}_{e_i}^0 (\widehat w_h - w_h)
		& \lesssim |E|^{1/2} h_E^{1/2}  \norm{\widehat{w}_h - w_h}_{L^2(\partial E)}.
 	\end{align*}
 	For the second term on the right-hand side of \cref{eq:CR_integral_1}, since
  $\widehat{\pi}_{e_i}^0$ preserves the constants, we write
	\begin{align*}
 		\frac{1}{d+1} \sum_{i=1}^{d+1} \int_E (\widehat{\pi}_{e_i}^0 w_h - w_h)\dif x & \lesssim 	\frac{|E|^{1/2}}{d+1} \sum_{i=1}^{d+1} \norm{ \widehat{\pi}_{e_i}^0 w_h - w_h}_{L^2(E)} \\
		& \lesssim |E|^{1/2} h_E  \norm{\nabla w_h}_{L^2(E)}.
 	\end{align*}
These bounds imply
\begin{align*}
 		\int_\Omega \mathcal{L}_h^{\rm CR}(\widehat w_{h}) &\dif x	\lesssim  \sum_{E \in \mathcal{E}_h} |E|^{1/2} \del{ h_E  \norm{\nabla w_h}_{L^2(E)}+h_E^{1/2}  \norm{\widehat{w}_h - w_h}_{L^2(\partial E)} } \\
        & \lesssim |\Omega|^{1/2} \del[2]{  h^2 \norm{\nabla_h w_h}_{L^2(\Omega)}^2   +  \sum_{E \in \mathcal{E}_h} h_E  \norm{\widehat{w}_h - w_h}_{L^2(\partial E)}^2}^{1/2},
        \end{align*}
        which, with \eqref{eq:boundgradint}, gives the desired result.
  \end{proof}

    \subsection{Proof of \Cref{lem:HDG_RK}}
\label{ss:append_compact}

We reproduce the proof in \cite{Jiang:2023} here for completeness.
Since the sequence $\norm[1]{(w_h,\widehat{w}_h)}_{1,h}$ is uniformly bounded, the arguments in the proof of \Cref{lem:HDG_poincare} show that 
	$\norm[0]{ \mathcal{L}_h^{\rm CR}(	\widehat w_{h})}_{L^2(\Omega)} + \norm[0]{\nabla_h \mathcal{L}_h^{\rm CR}(	\widehat w_{h})}_{L^2(\Omega)}$ is also uniformly bounded. The compactness properties of the Crouzeix--Raviart element (see e.g. \cite[Section 2.4.3]{Brenner:2015}) yields the existence of a function $w \in H^1(\Omega)$ such that up to a subsequence,
		\[
		\mathcal{L}_h^{\rm CR}(	\widehat w_{h}) \to w, \quad \text{ in } L^2(\Omega).
 		\]
		Using the triangle inequality and the bound \eqref{eq:intboundOmega}, we find that
		\begin{align*}
			\norm[0]{w_h - w}_{L^2(\Omega)}  & \le \norm[1]{w_h - \mathcal{L}_h^{\rm CR}(	\widehat w_{h})}_{L^2(\Omega)} + \norm[1]{w - \mathcal{L}_h^{\rm CR}(	\widehat w_{h})}_{L^2(\Omega)} \\ & \lesssim h^{1/2}  \norm{(w_h,\widehat w_h)}_{1,h} + \norm[1]{w - \mathcal{L}_h^{\rm CR}(	\widehat w_{h})}_{L^2(\Omega)},
 		\end{align*}
		and thus passing to the limit as $h \to 0$,
		\[
		\lim_{h \to 0} \norm[0]{w_h - w}_{L^2(\Omega)} = 0.
        \]
		Moreover, the uniform boundedness of $\norm[1]{(w_h,\widehat{w}_h)}_{1,h}$ ensures the uniform boundedness of $\norm[1]{\bm G_h(w_h,\widehat{w}_h)}_{L^2(\Omega)}$
        and of $\norm[0]{\widehat{w}_h}_{L^2( \partial \Omega)}$. Indeed, one can show (see proof below)
        \begin{equation}\label{eq:boundpartOmega}
        \Vert \widehat w_h \Vert_{L^2(\partial\Omega)}
        \lesssim \Vert (w_h, \widehat w_h)\Vert_{1,h}.
        \end{equation}
        Thus, upon passage to a subsequence, there exist functions $\bm g \in L^2(\Omega)^d$ and $\widetilde w \in L^2(\partial \Omega)$ such that
		\[
 		\bm G_h(w_h,\widehat{w}_h) \rightharpoonup \bm g, \quad \text{ in } L^2(\Omega)^d, \quad \text{and}\quad \widehat{w}_h \to \widetilde{w}, \quad \text{ in } L^2(\partial \Omega).
		\]

		To see that in fact $\bm g = \nabla w$ and $\widetilde w = w|_{\partial \Omega}$, let $\bm \varphi \in C^\infty(\overline{\Omega})^d$ be arbitrary, define $\bm \varphi_h = \bm \Pi_h^{\rm BDM} \bm \varphi$ in the definition of the HDG gradient, and use weak and strong convergence results (we know that  $\varphi_h$ converges strongly to $\varphi$ in $L^2(\Omega)$) above to find:
		\begin{align*}
			(\bm g, \bm \varphi)_{\mathcal{E}_h} &= \lim_{h \to 0}  (\bm G(w_h,\widehat{w}_h),  \bm \Pi_h^{\rm BDM} \bm \varphi)_{\mathcal{E}_h} \\
 			&= \lim_{h \to 0} \del{  (\nabla w_h,  \bm \Pi_h^{\rm BDM} \bm \varphi)_{\mathcal{E}_h}  - \langle  w_h - \widehat w_h,  \bm \Pi_h^{\rm BDM} \bm \varphi \cdot \bm n\rangle_{ \partial \mathcal{E}_h} } \\
   &  = \lim_{h \to 0}  \del{  -(w_h,  \nabla \cdot  \bm \Pi_h^{\rm BDM}\bm \varphi)_{\mathcal{E}_h}  + \langle \widehat w_h,  \bm \Pi_h^{\rm BDM}\bm \varphi \cdot \bm n\rangle_{ \partial \mathcal{E}_h} } 
   \\
     &  = \lim_{h \to 0}  \del{  -(w_h,  \nabla \cdot  \bm \Pi_h^{\rm BDM}\bm \varphi)_{\mathcal{E}_h}  + \langle \widehat w_h,  \bm \varphi \cdot \bm n\rangle_{ \partial \mathcal{E}_h} } 
    \\
			&=   \lim_{h \to 0} \del{ -(w_h,  \pi_{k-1}\nabla \cdot \bm \varphi)_{\mathcal{E}_h} + \langle \widehat w_h,  \bm \varphi \cdot \bm n\rangle_{ \partial \Omega}}
 			\\ &= - (w,  \nabla \cdot \bm \varphi)_{\mathcal{E}_h} + \langle \widetilde w,  \bm \varphi \cdot \bm n\rangle_{ \partial \Omega}.
 		\end{align*}
   Thus, on the one hand, selecting $\bm \varphi \in C_c^\infty(\Omega)^d$ yields $\bm g = \nabla w$. On the other hand, this means that for all $\bm \varphi \in C^\infty(\overline{\Omega})^d$,
 		\begin{align*}
 			0 &= - (\nabla w, \bm \varphi)_{\mathcal{E}_h} - (w,  \nabla \cdot \bm \varphi)_{\mathcal{E}_h} + \langle \tilde w,  \bm \varphi \cdot \bm n\rangle_{ \partial \Omega} \\
			&=  \langle \tilde w - w,  \bm \varphi \cdot \bm n\rangle_{ \partial \Omega},
 		\end{align*}
		and thus $\tilde w = w|_{\partial \Omega}$. To complete the proof, we now show \eqref{eq:boundpartOmega}.
         Let $e \subset \partial \Omega$. Assume $e\subset E$. By the triangle inequality and a discrete trace inequality, we have
\begin{align*}
         \Vert &\widehat w_h \Vert_{L^2(e)}   \le \norm{w_h - \widehat w_h}_{L^2(e)} + \norm[1]{w_h - \mathcal{L}_h^{\rm CR}(	\widehat w_{h})}_{L^2(e)} +
         \norm[1]{\mathcal{L}_h^{\rm CR}(	\widehat w_{h})}_{L^2(e)} \\
         & \lesssim \norm{w_h - \widehat w_h}_{L^2(e)} + h_E^{-1/2}\norm[1]{w_h - \mathcal{L}_h^{\rm CR}(	\widehat w_{h})}_{L^2(E)} +
         \norm[1]{\mathcal{L}_h^{\rm CR}(\widehat w_{h})}_{L^2(e)} \\
        & \lesssim h_E^{1/2}\Vert \bm G_h(w_h,\widehat{w}_h)\Vert_{L^2(E)}
         + \Vert w_h -\widehat{w}_h\Vert_{L^2(\partial E)}
         + \norm[1]{F_h(E_h(\mathcal{L}_h^{\rm CR}(\widehat w_{h})))}_{L^2(e)}.
          \end{align*}
The last bound is obtained by \eqref{eq:disc_sobolev_3} and by using the identity
of the enrichment operator and the do-nothing operator: $\mathcal{L}_h^{\rm CR}(\widehat w_{h}) = F_h(E_h(\mathcal{L}_h^{\rm CR}(\widehat w_{h})))$.
For readibility, denote $z_h = E_h(\mathcal{L}_h^{\rm CR}(\widehat w_{h}))$. We note that
$F_h(z_h)$ is the piecewise linear interpolant of $z_h$. We thus have (with $e\subset E$)
\[
\Vert F_h(z_h)\Vert_{L^2(e)} \lesssim h_E^{1/2} \Vert \nabla z_h\Vert_{L^2(E)}
+ \Vert z_h \Vert_{L^2(e)}.
\]
Next, we square and sum the bound above for all $e\subset\partial\Omega$. We obtain
\begin{align*}
\Vert \widehat{w}_h\Vert_{L^2(\partial\Omega)}^2
\lesssim & h\Vert \bm{G}_h(w_h,\widehat w_h)\Vert_{L^2(\Omega)}^2 
+ \Vert w_h - \widehat w_h\Vert_{L^2(\partial\mathcal{E}_h)}^2\\
&+ h \Vert \nabla E_h(\mathcal{L}_h^{\rm CR}(\widehat w_{h}))\Vert_{L^2(\Omega)}^2
+\Vert E_h(\mathcal{L}_h^{\rm CR}(\widehat w_{h}))\Vert_{L^2(\partial\Omega)}^2\\
  \lesssim & \norm{(w_h,\widehat w_h)}_{1,h}^2
+ \Vert \nabla E_h(\mathcal{L}_h^{\rm CR}(\widehat w_{h}))\Vert_{L^2(\Omega)}^2
 + \Vert E_h(\mathcal{L}_h^{\rm CR}(\widehat w_{h}))\Vert_{L^2(\Omega)}^2\\
 \lesssim &\norm{(w_h,\widehat w_h)}_{1,h}^2
 + \Vert \nabla_h \mathcal{L}_h^{\rm CR}(\widehat w_{h})\Vert_{L^2(\Omega)}^2 + \Vert \mathcal{L}_h^{\rm CR}(\widehat w_{h})\Vert_{L^2(\Omega)}^2.
 \end{align*}
For the second inequality above, we use the
trace inequality to $z_h$ and for the third inequality, we apply \eqref{eq:boundsEh}. 
We then conclude with \eqref{eq:grad_LCR_HDG} and \eqref{eq:L2Lh}.

\subsection{Proof of bound \eqref{eq:par_el_calc}}
\label{app:scaling}

\begin{proof}
We recall that $\bm \eta_h^i$ belongs to the Raviart-Thomas space $\text{RT}_k(\Omega)$. 
Let $\tilde{E}$ be the reference element and let $\tilde{\bm\pi}_{k-1}$ denote the
$L^2$ projection on $\mathbb{P}_{k-1}(\tilde{E})^d$. Since 
\[
\tilde{\bm Z} \mapsto \Vert \tilde{\bm\pi}_{k-1}\tilde{\bm Z}\Vert_{L^2(\tilde{E})}
+ \sum_{\tilde{e}\subset\partial\tilde{E}} \Vert \tilde{\bm Z}\cdot\tilde{\bm n}\Vert_{L^2(\tilde{e})}
\]
is a norm on the space $RT_k(\tilde{E})$, the equivalence of norms on finite-dimensional spaces yields
\[
\Vert \tilde{\bm Z}\Vert_{L^2(\tilde E)} \lesssim \Vert \tilde{\bm\pi}_{k-1}\tilde{\bm Z}\Vert_{L^2(\tilde{E})}
+ \sum_{\tilde{e}\subset\partial\tilde{E}} \Vert \tilde{\bm Z}\cdot\tilde{\bm n}\Vert_{L^2(\tilde{e})}.
\]
Using the Piola transformation from the reference element to any element $E$, we can write 
\begin{align*}
\Vert \bm \eta_h^i \Vert_{L^2(E)}  &\lesssim h_E^{-d/2+1} \Vert \tilde{\bm \eta}_h^i
\Vert_{L^2(\tilde E)}\lesssim h_E^{-d/2+1} \sum_{\tilde{e}\subset\partial\tilde{E}} \Vert \tilde{\bm \eta}_h^i\cdot\tilde{\bm n}\Vert_{L^2(\tilde{e})},
\end{align*}
since $\bm \pi_{k-1} \bm \eta_h^i = \bm 0$.
We now pass back to the physical element:
\[
\Vert \bm \eta_h^i \Vert_{L^2(E)}  \lesssim h_E^{-d/2+1} 
h_E^{(d-1)/2} \sum_{e\subset\partial E} \Vert \bm\eta_h^i \cdot \bm n\Vert_{L^2(e)} \lesssim h^{1/2} \Vert \bm \eta_h^i \cdot \bm n\Vert_{L^2(\partial\mathcal{E}_h)}.
\]
\end{proof}

\subsection{Discrete Aubin--Lions} \label{ss:appendix_AubinLions}

First, we prove the following discrete variant of Ehrling's lemma:

\begin{lemma} \label{lem:Ehrling}
Let $\mathcal{H} \subset (0,|\Omega|)$ be a countable collection of mesh-sizes whose unique accumulation point is zero.
Let $\cbr{(w_h, \widehat w_h)}_{h \in \mathcal{H}} \in \cbr{W_h \times M_h}_{h \in \mathcal{H}}$.
For any $\epsilon > 0$, there a constant $C(\epsilon) \ge 0$ such that, we have
\[
\forall h\in\mathcal{H}, \quad \norm{w_h}_{L^2(\Omega)} \le \epsilon \norm[0]{(w_h, \widehat w_h)}_{1,h} + C(\epsilon) \norm{w_h}_{W^{1,2d}(\Omega)^\star}.
\]
\end{lemma}

\begin{proof}
    The proof, by contradiction, is standard (see, e.g. \cite{Gallouet:2012} in the context of numerical schemes for parabolic equations). Let $\cbr{h_j}_{j=1}^\infty$ be an enumeration of the countable set of mesh-sizes $\mathcal{H}$. Suppose to the contrary that the lemma is false: there exists an $\epsilon_0 > 0$ such that, for any $j\in\mathbb{N}^\ast$, we can find $(w_{h_{j}}, \widehat w_{h_{j}}) \in \mathcal{W}_{h_{j}} \times \mathcal{M}_{h_{j}}$ satisfying
    \begin{equation} \label{eq:Ehrling}
    \norm[0]{w_{h_{j}}}_{L^2(\Omega)} > \epsilon_0 \norm[0]{(w_{h_{j}}, \widehat w_{h_{j}})}_{1,h_{j}} + j\norm[0]{w_{h_{j}}}_{H^2(\Omega)^\star}.
    \end{equation}
    We may suppose that $\norm[0]{w_{h_{j}}}_{L^2(\Omega)} = 1$.
    On the one hand, \cref{eq:Ehrling} yields $\norm[0]{(w_{h_{j}}, \widehat w_{h_{j}})}_{1,h_{j}} < 1/\epsilon_0$. Owing to \Cref{lem:HDG_RK}, we can pass to a subsequence to find $w_{h_j} \to w$ in $L^2(\Omega)$ as $j \to \infty$ for some $w \in L^2(\Omega)$ with $\norm{w}_{L^2(\Omega)} = 1$.
    On the other hand, we find $\norm[0]{w_{h_{j}}}_{W^{1,2d}(\Omega)^\star} \le 1/j$, so that $\lim_{j \to \infty} \norm[0]{w_{h_{j}}}_{W^{1,2d}(\Omega)^\star} = 0$. 
    Since $L^2(\Omega) \subset W^{1,2d}(\Omega)^\star$, this contradicts that $w \ne 0$. 
\end{proof}

\begin{lemma}
\label{lem:convChapp}
Let $\mathcal{H} \subset (0,|\Omega|)$ be a countable collection of mesh-sizes whose unique accumulation point is zero.
    Let $\cbr{(c_h,\widehat c_h)}_{h \in \mathcal{H}} \subset \mathcal{W}_h \times \mathcal{M}_h$ be a sequence such that the following bounds hold
    \begin{align} \label{eq:compact_time_der_bnd}
    \tau \sum_{i=1}^N 
    \norm[0]{\delta_\tau c_h^i}_{W^{1,2d}(\Omega)^\star}^2 \le M_1, \\
\label{eq:compact_spat_der_bnd}
    \max_{1 \le i \le N} \norm[0]{c_h^i}_{L^2(\Omega)} +\tau \sum_{i=1}^N \norm[0]{(c_h^i, \widehat c_h^{\, i})}_{1,h}^2 \lesssim M_2,
    \end{align}
    for some constants $M_1,M_2>0$ independent of $h$ and $\tau$. Then, there exists $\widetilde c \in L^2(0,T;L^2(\Omega))$ and a (not relabeled) subsequence $\cbr{c_h}_{h \in \mathcal{H}}$ such that
    \[
    c_h \to \widetilde c, \quad \text{in } L^2(0,T;L^2(\Omega)).
    \]
\end{lemma}

\begin{proof}

In brief, we first show that $\cbr{c_h}_{h \in \mathcal{H}}$ is relatively compact in $L_{\rm loc}^2(0,T;W^{1,2d}(\Omega)^\star)$, and then use a discrete Ehrling inequality (c.f. \Cref{lem:Ehrling}) and a uniform integrability result to upgrade the compactness to $L^2(0,T;L^2(\Omega))$. The argument is broken into three steps:
\begin{enumerate}[noitemsep,topsep=2pt,leftmargin=!,labelwidth=\widthof{(ii)}]
    \item[(i)] 
    We mollify $\cbr{c_h}_{h \in \mathcal{H}}$ to obtain a sequence of temporally smoothed approximations $\cbr{c_h^{\eta}}_{h \in \mathcal{H}}$ so that we may leverage the  Arzel\'{a}--Ascoli theorem (see e.g. \cite[Lemma 1]{Simon:1986}) to $\cbr{c_h^{\eta}}_{h \in \mathcal{H}}$ is relatively compact in $C(0,T;W^{1,2d}(\Omega)^\star)$, and hence also in $L^2(0,T;W^{1,2d}(\Omega)^\star)$.
    As $L^2(0,T;W^{1,2d}(\Omega)^\star)$ is complete, we conclude $\cbr{c_h^{\eta}}_{h \in \mathcal{H}}$ is totally bounded and so admits a finite cover of $\epsilon/2$-balls.
    \item[(ii)] 
    If we can assert that for each $h \in \mathcal{H}$ and a given $\theta$ with $0 < \theta < T/2$, $c_h^{\eta}$ converges uniformly as $\eta \to 0$ in $L^2(\theta,T-\theta;W^{1,2d}(\Omega)^\star)$, then we can conclude by the triangle inequality that $\cbr{c_h}_{h \in \mathcal{H}}$ too can be covered by a finite number of $\epsilon$-balls, and hence is relatively compact in $L^2(\theta,T-\theta;W^{1,2d}(\Omega)^\star)$. The key to proving this uniform convergence is a uniform bound on the time-translates of $\cbr{c_h}_{h \in \mathcal{H}}$, and we establish below that the necessary bound on the time-translates follows from \cref{eq:compact_time_der_bnd}.
    \item[(iii)] Finally, to translate the compactness to $L^2(0,T;L^2(\Omega))$, we argue as follows. Using the previously establish compactness of $\cbr{c_h}_{h \in \mathcal{H}}$ in $L^2(\theta,T-\theta;W^{1,2d}(\Omega)^\star)$, we extract a (not relabeled) strongly converging subsequence. We then show that this subsequence is Cauchy in $L^2(\theta,T-\theta;L^2(\Omega))$ as a consequence of the Ehrling-type lemma \Cref{lem:Ehrling}. Finally, we show that we have compactness in $L^2(0,T;L^2(\Omega))$ by leveraging the uniform $L^\infty(0,T;L^2(\Omega))$ bound on $\cbr{c_h}_{h \in \mathcal{H}}$ in \cref{eq:compact_spat_der_bnd}.
\end{enumerate}
\medskip

\emph{ad. (i)}
Let $\varphi: \mathbb{R} \to \mathbb{R}$ be a smooth, non-negative mollifier compactly supported in the interval $(-1,1)$ with unit integral. For $\eta > 0$, define a family $\cbr{(c_h^\eta, \widehat{c}_h^{\, \eta})}_{h \in \mathcal{H}}$ of smoothed discrete functions as follows: first, extend each $(c_h,\widehat c_h)$ by zero outside of $[0,T]$. Then, set $c_h^\eta = \varphi_\eta * c_h$ and $\widehat c_h^{\, \eta} = \varphi_\eta   *  \widehat c_h$ for each $h \in \mathcal{H}$. 
Observe that for fixed $t \in [0,T]$, \cref{eq:compact_spat_der_bnd} yields
\begin{align*}
  \norm{(c_h^\eta(t), \widehat c^{\, \eta}_h(t))}_{1,h} &\leq \int_{|s| < \eta} |\varphi_{\eta}(t-s))|\norm{(c_h(s), \widehat c_h(s))}_{1,h} \dif s \\
  & \le \sup_{|s| \le \eta} |\varphi_{\eta}(t - s)| \int_0^T \norm{(c_h(s), \widehat c_h(s))}_{1,h} \dif s \\
  & \lesssim   M_2.
\end{align*}
By \cref{lem:HDG_RK}, $\cbr{ c_h^\eta(t) \, : \, h \in \mathcal{H}}$ is relatively compact in $L^2(\Omega)$ and thus also 
in $W^{1,2d}(\Omega)^\star$ since the former embeds continuously into the latter. From the uniform Lipschitz continuity of $\varphi_{\eta}$, 
\begin{align*}
&\norm{(c_h^\eta(t_1), \widehat c_h^{\,\eta}(t_1)) - (c_h^\eta(t_2), \widehat c_h^{\,\eta}(t_2))}_{1,h} 
\\ & \leq \int_{|s| < \eta} |\varphi_{\eta}(t_1-s) - \varphi_{\eta}(t_2-s)| \norm{(c_h(s), \widehat c_h(s))}_{1,h} \dif s \\
& \lesssim M_2|t_1 - t_2|.
\end{align*}
Thus, the family $\cbr{ c_h^\eta(t) \, : \, h \in \mathcal{H}, 0\leq t\leq T}$ is uniformly Lipschitz and hence uniformly equicontinuous. Using Arzel\'{a}--Ascoli, we conclude that the set $\cbr{ c_h^\eta \, : \, h \in \mathcal{H}}$ is relatively compact in $C(0,T;W^{1,2d}(\Omega)^\star)$, and hence also in the space $L^2(0,T;W^{1,2d}(\Omega)^\star)$ as the former embeds continuously into the latter. 
\medskip

\emph{ad. (ii)} We fix $\theta$ with $0<\theta< T/2$ and we show that $c_h^\eta \to c_h$ in $L^2(\theta,T-\theta;W^{1,2d}(\Omega)^\star)$ uniformly as $\eta \to 0$.
To this end, observe
by Cauchy--Schwarz's inequality
\begin{align*}
&\norm{c_h(t) - c_h^\eta(t)}_{W^{1,2d}(\Omega)^\star} \\&\le \int_{|s| < \eta} \norm{c_h(t) - c_h(t - s)}_{W^{1,2d}(\Omega)^\star} \varphi_\eta (s) \dif s \\
& \le \del[3]{\int_{|s| < \eta} \varphi_{\eta}(s) \dif s}^{1/2} \del{\int_{|s| < \eta} \varphi_{\eta}(s) \norm{c_h(t) - c_h(t - s)}_{W^{1,2d}(\Omega)^\star}^2 \dif s}^{1/2} \\
&= \del{\int_{|s| < \eta} \varphi_{\eta}(s) \norm{c_h(t) - c_h(t - s)}_{W^{1,2d}(\Omega)^\star}^2 \dif s}^{1/2}.
\end{align*}
Squaring both sides of the above inequality, integrating over $(\theta,T-\theta)$ for $\theta < T/2$, and applying Fubini's theorem, we find
\begin{align*}
\int_\theta^{T-\theta} &\norm{c_h(t) - c_h^\eta(t)}_{W^{1,2d}(\Omega)^\star}^2 \\ &\le \int_\theta^{T-\theta} \int_{|s| < \eta} \varphi_{\eta}(s) \norm{c_h(t) - c_h(t - s)}_{W^{1,2d}(\Omega)^\star}^2 \dif s \dif t \\
&=   \int_{|s| < \eta} \varphi_{\eta}(s) \int_\theta^{T-\theta} \norm{c_h(t) - c_h(t - s)}_{W^{1,2d}(\Omega)^\star}^2  \dif t \dif s 
\end{align*}
Choosing $\eta$ small enough, namely $\eta <  \theta$,  we have
\[
\int_\theta^{T-\theta} \norm{c_h(t) - c_h^\eta(t)}_{W^{1,2d}(\Omega)^\star}^2 
 \le \sup_{|\delta'| < \eta  } \int_{ \eta  }^{T} \norm{c_h(t) - c_h(t - \delta')}_{W^{1,2d}(\Omega)^\star}^2  \dif t
 \]\\
To estimate the right hand side of the above inequality, we note that since $c_h$ is piecewise constant in time, 
and if $\delta'\geq 0$
\begin{align*}
\norm{c_h(t) - c_h(t-\delta')}_{W^{1,2d}(\Omega)^\star} & \le \sum_{i: \, t_{i} \in (t-\delta',t)} \norm[0]{c_h^i - c_h^{i-1}}_{W^{1,2d}(\Omega)^\star} \\
&\leq \tau \sum_{i: \, t_{i} \in (t-\delta',t)} \norm[0]{\delta_\tau c_h^i}_{W^{1,2d}(\Omega)^\star} 
\end{align*}
Since there can be at most $N$ time intervals, we have
\begin{align*}
\norm{c_h(t) - c_h(t-\delta')}_{W^{1,2d}(\Omega)^\star} ^2 \leq &
T \tau \sum_{i: \, t_{i} \in (t-\delta',t)} \norm[0]{\delta_\tau c_h^i}_{W^{1,2d}(\Omega)^\star}^2\\
& \leq T \tau \sum_{j=1}^N \norm[0]{\delta_{\tau} c_h^j}_{W^{1,2d}(\Omega)^\star}^2
\chi_{(t_j,t_j+\delta')}(t),
\end{align*}
where $\chi_{(t_j,t_j+\delta')}$ is the characteristic function of the set $(t_j,t_j+\delta')$.  We then obtain with \eqref{eq:compact_time_der_bnd}
\Bk
\begin{align*}
    \int_\eta^{T} \norm{c_h(t) - c_h(t-\delta')}_{W^{1,2d}(\Omega)^\star}^2 \dif t &  
    \leq T \tau \sum_{j=1}^N \norm[0]{\delta_\tau c_h^j}_{W^{1,2d}(\Omega)^\star}^2
\int_\eta^T \chi_{(t_j,t_j+\delta')}(t) \dif t\\
    & \leq T^2 M_1 \vert \delta'\vert
    \leq T^2 M_1 \eta.
\end{align*}
The same bound can be obtained if $\delta'\leq 0$. 
This implies that for any $\epsilon>0$ and any $\eta < \epsilon/(2 T^2 M_1)$, we have
\begin{align*}
\int_\theta^{T-\theta} \norm{c_h(t) - c_h^\eta(t)}_{W^{1,2d}(\Omega)^\star}^2  \leq
T^2 M_1 \eta
 \leq \frac{\epsilon}{2}.
\end{align*}
\Bk
It follows that $\cbr{ c_h \, : \, h \in \mathcal{H}}$ is relatively compact in $L^2(\theta,T-\theta;W^{1,2d}(\Omega)^\star)$.
\medskip

\emph{ad. (iii)}
We select a Cauchy subsequence of $\cbr{ c_h \, : \, h \in \mathcal{H}}$ in $L^2(\theta,T-\theta;W^{1,2d}(\Omega)^\star)$ and apply \Cref{lem:Ehrling}. Fix $\epsilon > 0$.
There exists $C(\epsilon) > 0$ such that
for any $h_1, h_2 \in \mathcal{H}$,
\begin{align*}
    \int_\theta^{T-\theta} & \norm{c_{h_1} - c_{h_2}}_{L^2(\Omega)}^2 \dif t \\ &\le 2\epsilon^2 \int_0^T\norm[0]{(c_{h_1} - c_{h_2}, \widehat c_{h_1} - \widehat c_{h_2})}_{1,h}^2 \dif t + 2C(\epsilon)^2 \int_\theta^{T-\theta}\norm{c_{h_1} - c_{h_2}}_{W^{1,2d}(\Omega)^\star}^2 \dif t \\
    & \le 4 M_2\epsilon^2 + 2C(\epsilon)^2 \int_\theta^{T-\theta}\norm{c_{h_1} - c_{h_2}}_{W^{1,2d}(\Omega)^\star}^2 \dif t.
\end{align*}
The first term is bounded by \eqref{eq:compact_spat_der_bnd}. The second term
tends to zero as $h_1, h_2$ tend to zero. This implies 
\begin{equation*}
    \limsup_{h_1,h_2 \to 0} \int_\theta^{T-\theta}  \norm{c_{h_1} - c_{h_2}}_{L^2(\Omega)}^2 \dif t \le 4M_1\epsilon^2,
\end{equation*}
As $\epsilon>0$ was arbitrary, we find 
\begin{equation*}
    \lim_{h_1,h_2 \to 0} \int_\theta^{T-\theta}  \norm{c_{h_1} - c_{h_2}}_{L^2(\Omega)}^2 \dif t = 0.
\end{equation*}
It follows that $\cbr{ c_h \, : \, h \in \mathcal{H}}$ is also Cauchy in $L^2(\theta, T-\theta, L^2(\Omega))$, from which we deduce its compactness in $L^2(\theta,T-\theta;L^2(\Omega))$. Finally, 
suppose we have a Cauchy subsequence in $L^2(\theta, T-\theta, L^2(\Omega))$ of $\cbr{ c_h \, : \, h \in \mathcal{H}}$. We apply H\"older's inequality and the $L^\infty(0,T;L^2(\Omega))$-bound in \cref{eq:compact_spat_der_bnd} to find
\begin{align*}
   \int_0^\theta &\norm{c_{h_1}(t) - c_{h_2}(t)}_{L^2(\Omega)}^2 \dif t + \int_{T- \theta}^{T} \norm{c_{h_1}(t) - c_{h_2}(t)}_{L^2(\Omega)}^2 \dif t \le 8 M_2^2\theta.   
\end{align*}
Thus, if $0 < \theta < \epsilon/(16M_1^2)$, the selected subsequence is also Cauchy in $L^2(0,T;L^2(\Omega))$. The proof is now complete.
\end{proof}

\bibliographystyle{siamplain}
\bibliography{references}

\begin{thebibliography}{10}

\bibitem{Bartels:2009}
{\sc S.~Bartels, M.~Jensen, and R.~M\"{u}ller}, {\em Discontinuous {G}alerkin
  finite element convergence for incompressible miscible displacement problems
  of low regularity}, SIAM Journal on Numerical Analysis, 47 (2009),
  pp.~3720--3743.

\bibitem{Beirao2021}
{\sc L.~Beirao~da Veiga, A.~Pichler, and G.~Vacca}, {\em A virtual element
  method for the miscible displacement of incompressible fluids in porous
  media}, Computer Methods in Applied Mechanics and Engineering, 375 (2021),
  p.~113649.

\bibitem{Boffi:book}
{\sc D.~Boffi, F.~Brezzi, and M.~Fortin}, {\em Mixed {F}inite {E}lement
  {M}ethods and {A}pplications}, vol.~44 of Springer {S}eries in
  {C}omputational {M}athematics, Springer, 2013.

\bibitem{Brenner:2004}
{\sc S.~C. Brenner}, {\em Poincaré--{F}riedrichs {I}nequalities for
  {P}iecewise $h^1$ {F}unctions}, SIAM J. Numer. Anal., 41 (2004),
  pp.~306--324.

\bibitem{Brenner:2015}
{\sc S.~C. Brenner}, {\em Forty {Y}ears of the {C}rouzeix-{R}aviart element},
  Numerical Methods for Partial Differential Equations, 31 (2015),
  pp.~367--396.

\bibitem{Brezis:book}
{\sc H.~Brezis}, {\em Functional {A}nalysis, {S}obolev {S}paces, and {P}artial
  {D}ifferential {E}quations}, Springer--{V}erlag {N}ew {Y}ork, 2011.

\bibitem{Buffa:2009}
{\sc A.~Buffa and C.~Ortner}, {\em Compact embeddings of broken {S}obolev
  spaces and applications}, {IMA} Journal of Numerical Analysis, 29 (2009),
  pp.~827--855.

\bibitem{cesmelioglu2023hybridizable}
{\sc A.~Cesmelioglu, D.~D. Pham, and S.~Rhebergen}, {\em A hybridizable
  discontinuous {G}alerkin method for the fully coupled time-dependent
  stokes/darcy-transport problem}, ESAIM: Mathematical Modelling and Numerical
  Analysis, 57 (2023), pp.~1257--1296.

\bibitem{cockburn2023hybridizable}
{\sc B.~Cockburn}, {\em Hybridizable discontinuous {G}alerkin methods for
  second-order elliptic problems: overview, a new result and open problems},
  Japan Journal of Industrial and Applied Mathematics, 40 (2023),
  pp.~1637--1676.

\bibitem{cockburn2016bridging}
{\sc B.~Cockburn, D.~A. Di~Pietro, and A.~Ern}, {\em Bridging the hybrid
  high-order and hybridizable discontinuous {G}alerkin methods}, ESAIM:
  Mathematical Modelling and Numerical Analysis, 50 (2016), pp.~635--650.

\bibitem{Cockburn:2008}
{\sc B.~Cockburn, B.~Dong, and J.~Guzm\'{a}n}, {\em A superconvergent
  {LDG}-hybridizable {G}alerkin method for second-order elliptic problems},
  Mathematics of Computation, 77 (2008), pp.~1887--1916.

\bibitem{cockburn2009hybridizable}
{\sc B.~Cockburn, B.~Dong, J.~Guzm{\'a}n, M.~Restelli, and R.~Sacco}, {\em A
  hybridizable discontinuous {G}alerkin method for steady-state
  convection-diffusion-reaction problems}, SIAM Journal on Scientific
  Computing, 31 (2009), pp.~3827--3846.

\bibitem{cockburnH1}
{\sc B.~Cockburn, G.~Fu, and W.~Qiu}, {\em Discrete {H}$^1$-{I}nequalities for
  spaces admitting {M}-decompositions}, SIAM Journal on Numerical Analysis, 56
  (2018), pp.~3407--3429.

\bibitem{Cockburn:2009}
{\sc B.~Cockburn, J.~Gopalakrishnan, and R.~Lazarov}, {\em Unified
  hybridization of discontinuous {G}alerkin, mixed, and continuous {G}alerkin
  methods for second order elliptic problems}, SIAM Journal on Numerical
  Analysis, 47 (2009), pp.~1319--1365.

\bibitem{Cockburn:2009b}
{\sc B.~Cockburn, J.~Guzm\'{a}n, and H.~Wang}, {\em Superconvergent
  discontinuous {G}alerkin methods for second-order elliptic problems},
  Mathematics of Computation, 78 (2009), pp.~1--24.

\bibitem{DawsonSunWheeler2004}
{\sc C.~Dawson, S.~Sun, and M.~Wheeler}, {\em Compatible algorithms for coupled
  flow and transport}, Comput. Meth. Appl. Mech. Eng., 193 (2004),
  pp.~2565--2580.

\bibitem{Pietro:book2}
{\sc D.~A. Di~Pietro and J.~Droniou}, {\em The {Hybrid High-Order} method for
  polytopal meshes}, no.~19 in Modeling, Simulation and Application, Springer
  International Publishing, 2020.

\bibitem{douglasEwingWheeler1983}
{\sc J.~Douglas, R.~Ewing, and M.~Wheeler}, {\em A time-discretization
  procedure for a mixed finite element approximation of miscible displacement
  in porous media}, RAIRO Analyse Num\'erique, 17 (1983), pp.~249--265.

\bibitem{droniou2018gradient}
{\sc J.~Droniou, R.~Eymard, T.~Gallou{\"e}t, C.~Guichard, and R.~Herbin}, {\em
  The Dradient Discretisation Method}, vol.~82, 2018.

\bibitem{droniou2019unified}
{\sc J.~Droniou, R.~Eymard, A.~Prignet, and K.~S. Talbot}, {\em Unified
  convergence analysis of numerical schemes for a miscible displacement
  problem}, Foundations of Computational Mathematics, 19 (2019), pp.~333--374.

\bibitem{Sayas:book}
{\sc S.~Du and F.~J. Sayas}, {\em An {Invitation} to the {T}heory of the
  {H}ybridizable {D}iscontinuous {G}alerkin {M}ethod: {P}rojections,
  {E}stimates, {T}ools}, Springer Briefs in Mathematics, Springer International
  Publishing, 2019.

\bibitem{EpshteynRiviere2008}
{\sc Y.~Epshteyn and B.~Riviere}, {\em Convergence of high order methods for
  miscible displacement}, International Journal of Numerical Analysis and
  Modeling, 5 (2008), pp.~47--63.

\bibitem{Ern:booki}
{\sc A.~Ern and J.~L. Guermond}, {\em Finite {E}lements {I}}, Texts in
  {A}pplied {M}athematics, Springer International Publishing, 2021.

\bibitem{EwingMfw1980}
{\sc R.~Ewing and M.~Wheeler}, {\em {G}alerkin methods for miscible
  displacement problems in porous media}, SIAM Journal on Numerical Analysis,
  17 (1980), pp.~351--365.

\bibitem{Fabien:2020b}
{\sc M.~Fabien, M.~Knepley, and B.~Riviere}, {\em Families of interior penalty
  hybridizable discontinuous {G}alerkin methods for second order elliptic
  problems}, Journal of Numerical Mathematics, 28 (2020), pp.~161--174.

\bibitem{Fabien:2020a}
{\sc M.~Fabien, M.~Knepley, and B.~Riviere}, {\em A high order hybridizable
  discontinuous {G}alerkin method for incompressible miscible displacement in
  heterogeneous media}, Results Appl. Math., 8 (2020).

\bibitem{Feng:1995}
{\sc X.~Feng}, {\em On existence and uniqueness results for a coupled system
  modeling miscible displacement in porous media}, Journal of Mathematical
  Analysis and Applications, 194 (1995), pp.~883--910.

\bibitem{Folland:book}
{\sc G.~B. Folland}, {\em Real {A}nalysis: {M}odern {T}echniques and {T}heir
  {A}pplications}, Wiley, 2nd~ed., 1999.

\bibitem{Gallouet:2012}
{\sc T.~Gallouët and J.-C. Latché}, {\em Compactness of discrete approximate
  solutions to parabolic {PDE}s - {A}pplication to a turbulence model},
  Communications on Pure and Applied Analysis, 11 (2012), pp.~2371--2391.

\bibitem{Girault:2016}
{\sc V.~Girault, J.~Li, and B.~Riviere}, {\em Strong convergence of discrete
  {DG} solutions of the heat equation}, Journal of Numerical Mathematics, 24
  (2016), pp.~235--252.

\bibitem{Jensen2010}
{\sc M.~Jensen and R.~M\"uller}, {\em Stable {C}rank-{N}icolson discretization
  for incompressible miscible displacement problem}, Numer. Math. Adv. Appl,
  (2010), pp.~469--477.

\bibitem{Jiang:2023}
{\sc J.~Jiang, N.~Walkington, and Y.~Yue}, {\em Stability and {C}onvergence of
  {HDG} {S}chemes {U}nder {M}inimal {R}egularity}, ArXiv e-prints 2310.18448,
  (2023).

\bibitem{Cockburn:nonlin}
{\sc H.~Kabaria, A.~J. Lew, and B.~Cockburn}, {\em A hybridizable discontinuous
  {G}alerkin formulation for non-linear elasticity}, Comput. Methods Appl.
  Mech. Eng., 283 (2015), pp.~303--329.

\bibitem{Kikuchi:2012}
{\sc F.~Kikuchi}, {\em Rellich-type discrete compactness for some discontinuous
  {G}alerkin {FEM}}, Japan J. Indust. Appl. Math., 29 (2012), pp.~269--288.

\bibitem{Kirk:2023a}
{\sc K.~Kirk, A.~Cesmelioglu, and S.~Rhebergen}, {\em Convergence to weak
  solutions of a space-time hybridized discontinuous {G}alerkin method for the
  incompressible {N}avier–{S}tokes equations}, Mathematics of Computation, 92
  (2023), pp.~147--174.

\bibitem{kirk2024combinedflow}
{\sc K.~L. Kirk and B.~Riviere}, {\em A combined mixed hybrid and hybridizable
  discontinuous {G}alerkin method for {D}arcy flow and transport}, Journal of
  Scientific Computing, 100 (2024), p.~57.

\bibitem{lehrenfeld2010hybrid}
{\sc C.~Lehrenfeld}, {\em Hybrid discontinuous {G}alerkin methods for solving
  incompressible flow problems}, Rheinisch-Westfalischen Technischen Hochschule
  Aachen, 111 (2010), p.~11.

\bibitem{lew2004optimal}
{\sc A.~Lew, P.~Neff, D.~Sulsky, and M.~Ortiz}, {\em Optimal {BV} estimates for
  a discontinuous {G}alerkin method for linear elasticity}, Applied Mathematics
  Research Express, 2004 (2004), pp.~73--106.

\bibitem{Li:2015}
{\sc J.~Li, B.~Riviere, and N.~J. Walkington}, {\em Convergence of a high order
  method in time and space for the miscible displacement equations}, ESAIM:
  Mathematical Modelling and Numerical Analysis, 49 (2015), pp.~953--976.

\bibitem{Nedelec:1980}
{\sc J.~C. N\'{e}d\'{e}lec}, {\em Mixed finite elements in $\mathbb{R}^3$},
  Numerische Mathematik, 35 (1980), pp.~315--341.

\bibitem{Petzoldt:2002}
{\sc M.~Petzoldt}, {\em Regularity results for {L}aplace interface problems in
  two dimensions}, Zeitschrift für Analysis und ihre Anwendungen, 21 (2002),
  pp.~881--898.

\bibitem{DiPietro:2010}
{\sc D.~A.~D. Pietro and A.~Ern}, {\em Discrete functional analysis tools for
  discontinuous {G}alerkin methods with application to the incompressible
  {N}avier--{S}tokes equations}, Mathematics of Computation, 79 (2010),
  pp.~1303--1330.

\bibitem{riviere_walkington_2011}
{\sc B.~M. Riviere and N.~J. Walkington}, {\em Convergence of a discontinuous
  {G}alerkin method for the miscible displacement equation under low
  regularity}, SIAM Journal on Numerical Analysis, 49 (2011), pp.~1085--1110.

\bibitem{Russell85}
{\sc T.~Russell}, {\em Time stepping along characteristics with incomplete
  iteration for a {G}alerkin approximation of miscible displacement in porous
  media}, SIAM Journal on Numerical Analysis, 22 (1985), pp.~970--1013.

\bibitem{Simon:1986}
{\sc J.~Simon}, {\em Compact sets in the space ${L}^p(0,{T}; {B})$}, Annali di
  Matematica Pura ed Applicata, 146 (1986), pp.~65--96.

\bibitem{Wells:2011}
{\sc G.~N. Wells}, {\em Analysis of an interface stabilized finite element
  method: the advection-diffusion-reaction equation}, SIAM Journal on Numerical
  Analysis, 49 (2011), pp.~87--109.

\bibitem{Yue:2024}
{\sc Y.~Yue}, {\em Discrete {P}oincar\'e inequality and discrete trace
  inequality in piece-wise polynomial hybridizable spaces}, 2024.
\newblock arxiv 2403.19004v1.

\end{thebibliography}
\end{document}